\documentclass[psamsfonts,reqno,11pt]{amsart}

\usepackage{amssymb,amsfonts}
\usepackage[all,arc]{xy}
\usepackage{enumerate}
\usepackage{mathrsfs}
\usepackage[toc,page]{appendix}
\usepackage{adjustbox,lipsum}
\usepackage{graphicx}
\usepackage[font=small,labelfont=bf]{caption}
\usepackage[left=3cm, right=3cm, bottom=3cm]{geometry}

\newtheorem{thm}{Theorem}[section]
\newtheorem{theorem}[thm]{Theorem}
\newtheorem{corollary}[thm]{Corollary}
\newtheorem{proposition}[thm]{Proposition}
\newtheorem{lemma}[thm]{Lemma}

\theoremstyle{definition}
\newtheorem{definition}[thm]{Definition}

\newtheorem{example}[thm]{Example}

\theoremstyle{remark}
\newtheorem{remark}[thm]{Remark}

\makeatletter
\let\c@equation\c@thm
\makeatother
\numberwithin{equation}{section}

\title[Model Higgs bundles in exceptional components]{Model Higgs bundles in exceptional components of the $\text{Sp(4}\text{,}\mathbb{R}\text{)}$-character variety}

\author{Georgios Kydonakis}

\date{31 May 2021
}

\begin{document}
\maketitle
\begin{abstract}
We establish a gluing construction for Higgs bundles over a connected sum of Riemann surfaces in terms of solutions to the $\text{Sp(4}\text{,}\mathbb{R}\text{)}$-Hitchin equations using the linearization of a relevant elliptic operator. The construction can be used to provide model Higgs bundles in all the $2g-3$ exceptional components of the maximal $\text{Sp(4}\text{,}\mathbb{R}\text{)}$-Higgs bundle moduli space, which correspond to components solely consisting of Zariski dense representations. This also allows a comparison between the invariants for maximal Higgs bundles and the topological invariants for Anosov representations constructed by O. Guichard and A. Wienhard.
\end{abstract}

\section{Introduction}
Let $\Sigma $ be a closed connected and oriented surface of genus $g\ge 2$ and $G$ be a connected semisimple Lie group. The moduli space of reductive representations of ${{\pi }_{1}}\left( \Sigma \right)$ into $G$ modulo conjugation \[\mathsf{\mathcal{R}}\left( G \right)=\text{Ho}{{\text{m}}^{+}}\left( {{\pi }_{1}}\left( \Sigma \right),G \right)/G\]
has been an object of extensive study and interest. Fixing a complex structure $J$ on the surface $\Sigma $ transforms this into a Riemann surface $X=(\Sigma,J)$ and opens the way for holomorphic techniques using the theory of Higgs bundles. The non-abelian Hodge theory correspondence provides a real-analytic isomorphism between the character variety $\mathsf{\mathcal{R}}\left( G \right)$ and the moduli space $\mathsf{\mathcal{M}}\left( G \right)$ of polystable $G$-Higgs bundles. The case when $G=\text{Sp(4}\text{,}\mathbb{R}\text{)}$ has received considerable attention by many authors, who studied the geometry and topology of the moduli space $\mathsf{\mathcal{M}}\left( \text{Sp(4}\text{,}\mathbb{R}\text{)} \right)$; see for instance \cite{BGGDeformations}, \cite{Collier}, \cite{GW}.  The subspace of maximal $\text{Sp(4}\text{,}\mathbb{R}\text{)}$-Higgs bundles ${{\mathsf{\mathcal{M}}}^{\max }}$, that is, the one containing Higgs bundles with extremal Toledo invariant, has been shown to have $3\cdot {{2}^{2g}}+2g-4$ connected components \cite{Gothen}.

Among the connected components of ${{\mathsf{\mathcal{M}}}^{\max }}$ there are $2g-3$ exceptional components of this moduli space. These components are all smooth but topologically non-trivial, and representations in these do not factor through any proper reductive subgroup of $\text{Sp}\left( 4,\mathbb{R} \right)$, thus have Zariski-dense image in $\text{Sp}\left( 4,\mathbb{R} \right)$. On the other hand, in the remaining $3\cdot {{2}^{2g}}-1$ components model Higgs bundles can be obtained by embedding stable $\text{SL(2}\text{,}\mathbb{R}\text{)}$-Higgs data into $\text{Sp(4}\text{,}\mathbb{R}\text{)}$ using appropriate embeddings $\phi :\text{SL(2}\text{,}\mathbb{R}\text{)}\hookrightarrow \text{Sp(4}\text{,}\mathbb{R}\text{)}$ (see \cite{BGGDeformations}). The construction of $\text{Sp(4}\text{,}\mathbb{R}\text{)}$-Higgs bundles that lie in the $2g-3$ exceptional components of the moduli space ${{\mathsf{\mathcal{M}}}^{\max }}$ is the principal objective in this article.

From the point of view of the character variety ${{\mathsf{\mathcal{R}}}^{\max }}$, model representations in a subfamily of the $2g-3$ special components have been effectively constructed by O. Guichard and A. Wienhard in \cite{GW} by amalgamating certain fundamental group representations defined over topological surfaces with one boundary component.

The first step in establishing a gluing construction from the holomorphic viewpoint is to describe holomorphic objects corresponding to $\text{Sp(4}\text{,}\mathbb{R}\text{)}$-representations over a surface with boundary with fixed arbitrary holonomy around the boundary. These objects are Higgs bundles defined over a Riemann surface with a divisor together with a weighted flag on the fibers over the points in the divisor, namely \emph{parabolic $\mathrm{Sp(4}\text{,}\mathbb{R}\text{)}$-Higgs bundles}. As in the non-parabolic case, a notion of maximality can still be defined for these objects.

It is important that a gluing construction for parabolic Higgs bundles over the complex connected sum ${{X}_{\#}}$ of two distinct compact Riemann surfaces $X_{1}$ and $X_{2}$ with a divisor of $s$-many distinct points on each, is formulated so that the gluing of stable parabolic pairs is providing a \emph{polystable} Higgs bundle over ${{X}_{\#}}$. Moreover, in order to construct new models in the components of $\mathsf{\mathcal{M}}\left( X_{\#}, \text{Sp(4}\text{,}\mathbb{R}\text{)} \right)$, the parabolic gluing data over $X_{1}$ and $X_{2}$ are chosen to be coming from different embeddings of $\text{SL(2}\text{,}\mathbb{R}\text{)}$-parabolic data into $\text{Sp(4}\text{,}\mathbb{R}\text{)}$, and so a priori do not agree over disks around the points in the divisors. We choose to switch to the language of solutions to Hitchin's equations and make use of the analytic techniques of C. Taubes for gluing instantons over 4-manifolds \cite{Taubes} in order to control the stability condition. These techniques have been applied to establish similar gluing constructions for solutions to gauge-theoretic equations, as for instance in \cite{DonKron}, \cite{Foscolo}, \cite{He}, \cite{Safari}.

The problem involves perturbing the initial data into model solutions which are identified locally over the annuli around the points in the divisors, thus allowing the construction of a pair over ${{X}_{\#}}$ that combines initial data over $X_{1}$ and $X_{2}$. The existence of these perturbations in terms of appropriate gauge transformations is provided for $\text{SL(2}\text{,}\mathbb{R}\text{)}$-data, and then we use the embeddings of $\text{SL(2}\text{,}\mathbb{R}\text{)}$ into $\text{Sp(4}\text{,}\mathbb{R}\text{)}$ to extend this deformation argument for our initial pairs. This produces an approximate solution to the $\text{Sp(4}\text{,}\mathbb{R}\text{)}$-Hitchin equations $\left( A_{R}^{app},\Phi _{R}^{app} \right)$ over ${{X}_{\#}}$, with respect to a parameter $R>0$ which describes the size of the neck region in the construction of ${{X}_{\#}}$. The pair $\left( A_{R}^{app},\Phi _{R}^{app} \right)$ coincides with the initial data over each hand side Riemann surface and with the model solution over the neck region.

The next step is to correct this approximate solution to an exact solution of the $\text{Sp(4}\text{,}\mathbb{R}\text{)}$-Hitchin equations over the complex connected sum of Riemann surfaces. In other words, we seek for a complex gauge transformation $g$ such that ${{g}^{*}}\left( A_{R}^{app},\Phi _{R}^{app} \right)$ is an exact solution of the $\text{Sp(4}\text{,}\mathbb{R}\text{)}$-Hitchin equations. The argument providing the existence of such a gauge is translated into a Banach fixed point theorem argument and involves the study of the linearization of a relevant elliptic operator. For Higgs bundles this was first studied by R. Mazzeo, J. Swoboda, H. Weiss and F. Witt in \cite{MSWW}, who described solutions to the $\text{SL}(2,\mathbb{C})$-Hitchin equations near the ends of the moduli space. A crucial step in this argument is to show that the linearization of the $G$-Hitchin operator at our approximate solution $\left( A_{R}^{app},\Phi _{R}^{app} \right)$ is invertible; this is obtained by showing that an appropriate self-adjoint Dirac-type operator has no small eigenvalues. The method is also used by J. Swoboda in \cite{Swoboda} to produce a family of smooth solutions of the $\text{SL(2}\text{,}\mathbb{C}\text{)}$-Hitchin equations, which may be viewed as desingularizing a solution with logarithmic singularities over a nodal Riemann surface. The analytic techniques from \cite{Swoboda} are extended to provide the main theorem from that article for solutions to the $\text{Sp(4}\text{,}\mathbb{R}\text{)}$-Hitchin equations as well, and moreover to obtain our main result:

\begin{theorem}[Theorem \ref{main_theorem}]
Let $X_{1}$ be a closed Riemann surface of genus $g_{1}$ and ${{D}_{1}}=\left\{ {{p}_{1}},\ldots ,{{p}_{s}} \right\}$ be a collection of $s$-many distinct points on $X_{1}$. Consider respectively a closed Riemann surface $X_{2}$ of genus $g_{2}$ and a collection of also $s$-many distinct points ${{D}_{2}}=\left\{ {{q}_{1}},\ldots ,{{q}_{s}} \right\}$ on $X_{2}$. Let $\left( {{E}_{1}},{{\Phi }_{1}} \right)\to {{X}_{1}}$ and $\left( {{E}_{2}},{{\Phi }_{2}} \right)\to {{X}_{2}}$ be parabolic polystable $\mathrm{Sp(4}\text{,}\mathbb{R}\text{)}$-Higgs bundles with corresponding solutions to the Hitchin equations $\left( {{A}_{1}},{{\Phi }_{1}} \right)$ and $\left( {{A}_{2}},{{\Phi }_{2}} \right)$. Assume that these solutions agree with model solutions $\left( A_{1,{{p}_{i}}}^{\bmod },\Phi _{1,{{p}_{i}}}^{\bmod } \right)$ and $\left( A_{2,{{q}_{j}}}^{\bmod },\Phi _{2,{{q}_{j}}}^{\bmod } \right)$ near the points ${{p}_{i}}\in {{D}_{1}}$ and ${{q}_{j}}\in {{D}_{2}}$, and that the model solutions satisfy $\left( A_{1,{{p}_{i}}}^{\bmod },\Phi _{1,{{p}_{i}}}^{\bmod } \right)=-\left( A_{2,{{q}_{j}}}^{\bmod },\Phi _{2,{{q}_{j}}}^{\bmod } \right)$, for $s$-many possible pairs of points $\left( {{p}_{i}},{{q}_{j}} \right)$. Then there is a polystable $\mathrm{Sp(4}\text{,}\mathbb{R}\text{)}$-Higgs bundle $\left( E_{\#},\Phi_{\#}  \right)\to {{X}_{\#}}$ over the connected sum of Riemann surfaces ${{X}_{\#}}={{X}_{1}}\#{{X}_{2}}$ of genus $g_{1}+g_{2}+s-1$, which agrees with the initial data over ${{X}_{\#}}\backslash {{X}_{1}}$ and ${{X}_{\#}}\backslash {{X}_{2}}$.
\end{theorem}

In analogy with the terminology introduced by O. Guichard and A. Wienhard in their construction of hybrid representations, we call the polystable Higgs bundles corresponding to such exact solutions \emph{hybrid}. The construction can have wider applicability in obtaining particular points in moduli spaces of polystable $G$-Higgs bundles. As an application, we construct here Higgs bundles corresponding to Zariski dense representations into $\text{Sp(4}\text{,}\mathbb{R}\text{)}$. For this purpose, we look at how the Higgs bundle topological invariants behave under the complex connected sum operation. We first show the following:

\begin{proposition}[Proposition \ref{additivity_pardeg}]
Let ${{X}_{\#}}=X_{1}\#X_{2}$ be the complex connected sum of two closed Riemann surfaces $X_{1}$ and $X_{2}$ with divisors $D_{1}$ and $D_{2}$ of $s$-many distinct points on each surface, and let ${{V}_{1}},{{V}_{2}}$ be parabolic principal ${{H}^{\mathbb{C}}}$-bundles over $X_{1}$ and $X_{2}$ respectively. Fix an antidominant character $\chi$ of $\mathrm{Lie} \left(P \right)$ and let $\sigma_{1}$, $\sigma_{2}$ be holomorphic reductions of the structure group of $V_{1}$, $V_{2}$ respectively from $H^{\mathbb{C}}$ to $P$. Assuming that the parabolic bundles $V_{1}$ and $V_{2}$ are glued to a bundle 
${{V}_{1}}\#{{V}_{2}}$, denote by $\sigma_{\#}$ the holomorphic reduction of the structure group of ${{V}_{1}}\#{{V}_{2}} $  from ${{H}^{\mathbb{C}}}$ to $P$ induced by $\sigma_{1}$ and $\sigma_{2}$. Then, the following identity holds:
\[\deg \left( {{V}_{1}}\#{{V}_{2}} \right)\left( \sigma_{\#} ,\chi  \right)=par{{\deg }_{{{\alpha }_{1}}}}\left( {{V}_{1}} \right)\left( \sigma_{1} ,\chi  \right)+par{{\deg }_{{{\alpha }_{2}}}}\left( {{V}_{2}} \right)\left( \sigma_{2} ,\chi  \right).\]
\end{proposition}
\vspace{2mm}
Note that an analogous additivity property for the Toledo invariant was established by M. Burger, A. Iozzi and A. Wienhard in \cite{BIW} from the point of view of fundamental group representations. It implies in particular that the connected sum of maximal parabolic $G$-Higgs bundles is again a maximal (non-parabolic) $G$-Higgs bundle. This property provides, however, a useful tool in order to construct models in components of Higgs bundle moduli spaces that are not necessarily maximal. 

We find model Higgs bundles in \emph{all} exceptional components of the maximal $\text{Sp(4}\text{,}\mathbb{R}\text{)}$-Higgs bundle moduli space; these models are described by hybrid Higgs bundles. In the case when $G=\text{Sp(4}\text{,}\mathbb{R}\text{)}$, considering all possible decompositions of a surface $\Sigma$ along a simple, closed, separating geodesic is sufficient in order to obtain representations in the desired components of ${{\mathsf{\mathcal{M}}}^{\max }}$, which are fully distinguished by the calculation of the degree of a line bundle. This degree can be identified with the Euler class for a hybrid representation as defined by O. Guichard and A. Wienhard, although these invariants live naturally in different cohomology groups.

The content of the article is described next. Sections 2 and 3 include the necessary definitions on $\text{Sp}\left( 4,\mathbb{R} \right)$-Higgs bundles and parabolic  $\text{Sp}\left( 4,\mathbb{R} \right)$-Higgs bundles respectively. We provide the set-up on which our gluing construction will be developed and no new results are included here. Sections 4, 5 and 6 contain the analytic machinery for the establishment of a general gluing construction for parabolic $G$-Higgs bundles over a complex connected sum of Riemann surfaces, while in Section 7 we are combining the arguments from these three sections to derive our main theorems. The final Section 8 deals with the question of obtaining models in the desired exceptional components of the $\text{Sp}\left( 4,\mathbb{R} \right)$-Higgs bundle moduli space. Moreover, we include here the discussion on the comparison between the invariants for maximal Higgs bundles and the topological invariants for Anosov representations constructed by O. Guichard and A. Wienhard. 

\vspace{2mm}
\textbf{Acknowledgments}.
I am indebted to an anonymous referee for a very careful reading of the original manuscript and a series of constructive comments and corrections which have improved the article. This work was part of the author's requirements for the Ph. D. degree at the University of Illinois at Urbana-Champaign. I am particularly grateful to my doctorate advisor, Professor Steven Bradlow, for his continuous support and guidance towards the completion of this project, Indranil Biswas, Olivier Guichard, Jan Swoboda, Nicolaus Treib, Hartmut Weiss and Richard Wentworth for shared insights, and Evgenia Kydonaki for helping me with drawing the figures included in this article. A very special thanks to Rafe Mazzeo for a series of illuminating discussions and a wonderful hospitality during a visit to Stanford University in April 2016. The author acknowledges support from U.S. National Science Foundation grants DMS 1107452, 1107263, 1107367 ``RNMS: GEometric structures And Representation varieties'' (the GEAR Network).
\vspace{2mm}

\section{$\text{Sp(}4,\mathbb{R}\text{)}$-Higgs bundles and surface group representations}

\subsection{Non-abelian Hodge theory}

Let $X$ be a compact Riemann surface and let $G$ be a real reductive group. The latter involves considering \emph{Cartan data} $\left( G,H,\theta ,B \right)$, where $H\subset G$ is a maximal compact subgroup, $\theta :\mathfrak{g}\to \mathfrak{g}$ is a Cartan involution and $B$ is a non-degenerate bilinear form on $\mathfrak{g}$, which is $\text{Ad}\left( G \right)$-invariant and $\theta $-invariant. The Cartan involution $\theta$ gives a decomposition (called the \emph{Cartan decomposition})
  \[\mathfrak{g}=\mathfrak{h}\oplus \mathfrak{m}\]
into its $\pm 1$-eigenspaces, where $\mathfrak{h}$ is the Lie algebra of $H$. Consider ${{\mathfrak{g}}^{\mathbb{C}}}={{\mathfrak{h}}^{\mathbb{C}}}\oplus {{\mathfrak{m}}^{\mathbb{C}}}$ the complexification of the Cartan decomposition. 

\begin{definition}
Let $K$ be the canonical line bundle over a compact Riemann surface $X$. A \emph{$G$-Higgs bundle} is a pair $\left( E,\varphi  \right)$ where
\begin{itemize}
  \item $E$ is a principal holomorphic ${{H}^{\mathbb{C}}}$-bundle over $X$ and
  \item $\varphi $ is a holomorphic section of the vector bundle $E\left( {{\mathfrak{m}}^{\mathbb{C}}} \right)\otimes K=\left( E{{\times }_{\iota }}{{\mathfrak{m}}^{\mathbb{C}}} \right)\otimes K$,
\end{itemize}
where  $\iota :{{H}^{\mathbb{C}}}\to \text{GL(}{{\mathfrak{m}}^{\mathbb{C}}}\text{)}$ is the complexified isotropy representation.\\
The section $\varphi $ is called the \emph{Higgs field}. Two $G$-Higgs bundles $\left( E,\varphi  \right)$ and $\left( {E}',{\varphi }' \right)$ are said to be \emph{isomorphic} if there is a principal bundle isomorphism $E\cong {E}'$ which takes the induced $\varphi $ to ${\varphi }'$ under the induced isomorphism $E\left( {{\mathfrak{m}}^{\mathbb{C}}} \right)\cong {E}'\left( {{\mathfrak{m}}^{\mathbb{C}}} \right)$.
\end{definition}

To define a moduli space of $G$-Higgs bundles we need to consider a notion of semistability, stability and polystability. These notions are defined in terms of an antidominant character for a parabolic subgroup ${{P}}\subseteq {{H}^{\mathbb{C}}}$ and a holomorphic reduction $\sigma $ of the structure group of the bundle $E$ from ${{H}^{\mathbb{C}}}$ to ${{P}}$ (see \cite{GGMHitchin-Kob} for the precise definitions). When the group $G$ is connected, principal ${{H}^{\mathbb{C}}}$-bundles $E$ are topologically classified by a characteristic class  $c\left( E \right)\in {{H}^{2}}\left( X,{{\pi }_{1}}\left( {{H}^{\mathbb{C}}} \right) \right)\cong {{\pi }_{1}}\left( {{H}^{\mathbb{C}}} \right) \cong {{\pi }_{1}}\left( H \right) \cong {{\pi }_{1}}\left( G \right)$.

\begin{definition}
For a fixed class $d\in {{\pi }_{1}}\left( G \right)$, the \emph{moduli space of polystable $G$-Higgs bundles} with respect to the group of complex gauge transformations is defined as the set of isomorphism classes of polystable $G$-Higgs bundles $\left( E,\varphi  \right)$ such that $c\left( E \right)=d$. We will denote this set by ${{\mathsf{\mathcal{M}}}_{d}}\left( G \right)$.
\end{definition}

Let $\left( E,\varphi  \right)$ be a $G$-Higgs bundle over a compact Riemann surface $X$. By a slight abuse of notation we shall denote the underlying smooth objects of $E$ and $\varphi $ by the same symbols. The Higgs field can be thus viewed as a $\left( 1,0 \right)$-form $\varphi \in {{\Omega }^{1,0}}\left( E\left( {{\mathfrak{m}}^{\mathbb{C}}} \right) \right)$. Given a reduction $h$ of structure group to $H$ in the smooth ${{H}^{\mathbb{C}}}$-bundle $E$, we denote by ${{F}_{h}}$ the curvature of the unique connection compatible with $h$ and the holomorphic structure on $E$. Let ${{\tau }_{h}}:{{\Omega }^{1,0}}\left( E\left( {{\mathfrak{g}}^{\mathbb{C}}} \right) \right)\to {{\Omega }^{0,1}}\left( E\left( {{\mathfrak{g}}^{\mathbb{C}}} \right) \right)$ be defined by the compact conjugation of ${{\mathfrak{g}}^{\mathbb{C}}}$ which is given fiberwise by the reduction $h$, combined with complex conjugation on complex 1-forms. The next theorem was proved in \cite{GGMHitchin-Kob} for an arbitrary reductive real Lie group $G$.

\begin{theorem}[Theorem 3.21 in \cite{GGMHitchin-Kob}] There exists a reduction $h$ of the structure group of $E$ from ${{H}^{\mathbb{C}}}$ to $H$ satisfying the Hitchin equation
	\[{{F}_{h}}-\left[ \varphi ,{{\tau }_{h}}\left( \varphi  \right) \right]=0\]
if and only if $\left( E,\varphi  \right)$ is polystable.
\end{theorem}

A solution to the Hitchin equation corresponds to a reductive fundamental group representation $\rho :{{\pi }_{1}}\left( \Sigma  \right)\to G$, where $\Sigma$ is the closed oriented topological surface underlying $X$. This is seen using that any solution $h$ to Hitchin's equations defines a flat reductive $G$-connection
\begin{equation}\label{flat_reductive}
  D={{D}_{h}}+\varphi -\tau \left( \varphi  \right),
\end{equation} where ${{D}_{h}}$ is the unique $H$-connection on $E$ compatible with its holomorphic structure. Conversely, given a flat reductive connection $D$ on a $G$-bundle $E_{G}$, there exists a harmonic metric, in other words, a reduction of structure group to $H\subset G$ corresponding to a harmonic section of ${{{E}_{G}}}/{H}\;\to X$. This reduction produces a solution to Hitchin's equations such that Equation (\ref{flat_reductive}) holds. 

For a closed oriented topological surface $\Sigma$ of genus $g$, define the \emph{moduli space of reductive representations of ${{\pi }_{1}}\left( \Sigma  \right)$ into $G$} to be the orbit space
	\[\mathsf{\mathcal{R}}\left( G \right)={\text{Ho}{{\text{m}}^{\text{red}}}\left( {{\pi }_{1}}\left( \Sigma  \right),G \right)}/{G}\;.\]
This space has a stratification by real analytic varieties indexed by the stabilizers of representations (see \cite{Goldman}) and so $\mathsf{\mathcal{R}}\left( G \right)$ is usually called the \emph{character variety}. We can assign a topological invariant to a representation $\rho \in \mathsf{\mathcal{R}}\left( G \right)$ by considering its corresponding flat $G$-bundle on $\Sigma$, ${{E}_{\rho }}=\tilde{\Sigma}{{\times }_{\rho }}G$, as the characteristic class $c\left( \rho  \right):=c\left( {{E}_{\rho }} \right)\in {{\pi }_{1}}\left( G \right)\simeq {{\pi }_{1}}\left( H \right)$, for $H\subseteq G$ a maximal compact subgroup of $G$.

In summary, equipping the surface $\Sigma$ with a complex structure $J$, a reductive representation of ${{\pi }_{1}}\left( \Sigma  \right)$ into $G$ corresponds to a polystable $G$-Higgs bundle over the Riemann surface $X=\left( \Sigma, J \right)$; this is the content of the \textit{non-abelian Hodge correspondence}; its proof is based on combined work by N. Hitchin \cite{Hit87}, C. Simpson \cite{Simpson-variations}, \cite{Simpson-Higgs}, S. Donaldson \cite{Donaldson} and K. Corlette \cite{Corlette}:
\begin{theorem}[Non-abelian Hodge correspondence]
Let $G$ be a connected semisimple real Lie group with maximal compact subgroup $H\subseteq G$ and let $d\in {{\pi }_{1}}\left( G \right)\simeq {{\pi }_{1}}\left( H \right)$. Then there exists a homeomorphism
	\[{{\mathsf{\mathcal{R}}}_{d}}\left( G \right)\cong {{\mathsf{\mathcal{M}}}_{d}}\left( G \right),\]
where ${{\mathsf{\mathcal{R}}}_{d}}\left( G \right), {{\mathsf{\mathcal{M}}}_{d}}\left( G \right)$ denote the subvarieties of points with fixed topological invariant $d$.
\end{theorem}

\subsection{$\text{Sp(}4\text{,}\mathbb{R}\text{)}$-Higgs bundles} In this article, we are particularly interested in the case when the group $G$ is the semisimple real subgroup of $\text{SL(4}\text{,}\mathbb{R}\text{)}$ that preserves a symplectic form on ${{\mathbb{R}}^{4}}$:
	\[\text{Sp(4}\text{,}\mathbb{R}\text{)}=\left\{ A\in \text{SL(4}\text{,}\mathbb{R}\text{)}\left| {{A}^{T}}{{J}}A={{J}} \right. \right\},\]
where ${{J}}=\left( \begin{matrix}
   0 & {{I}_{2}}  \\
   -{{I}_{2}} & 0  \\
\end{matrix} \right)$ defines a symplectic form on ${{\mathbb{R}}^{4}}$, for ${{I}_{2}}$ the $2\times 2$ identity matrix.

The Cartan involution $\theta :\mathfrak{sp}\left( 4,\mathbb{C} \right)\to \mathfrak{sp}\left( 4,\mathbb{C} \right)$ with $\theta \left( X \right)=-{{X}^{T}}$ determines a Cartan decomposition for a choice of maximal compact subgroup  $H\simeq \text{U}\left( 2 \right)\subset \text{Sp(4}\text{,}\mathbb{R}\text{)}$ as
	\[\mathfrak{sp}\left( 4,\mathbb{R} \right)=\mathfrak{u}\left( 2 \right)\oplus \mathfrak{m}\]
with complexification $\mathfrak{sp}\left( 4,\mathbb{C} \right)=\mathfrak{gl}\left( 2,\mathbb{C} \right)\oplus {{\mathfrak{m}}^{\mathbb{C}}}$. It is shown in \cite{GGMsymplectic} that the general definition for a $G$-Higgs bundle specializes to the following:

\begin{definition}
An \emph{$\mathrm{Sp(4}\text{,}\mathbb{R}\text{)}$-Higgs bundle} over a compact Riemann surface $X$ is defined by a triple $\left( V,\beta ,\gamma  \right)$, where $V$ is a rank $2$ holomorphic vector bundle over $X$ and $\beta , \gamma$ are symmetric homomorphisms $\beta :{{V}^{*}}\to V\otimes K$  and  $\gamma :V\to {{V}^{*}}\otimes K$,
where $K$ is the canonical line bundle over $X$.
\end{definition}

The embedding $\text{Sp(4}\text{,}\mathbb{R}\text{)}\hookrightarrow \text{SL(4}\text{,}\mathbb{C}\text{)}$ allows one to reinterpret the defining $\text{Sp(4}\text{,}\mathbb{R}\text{)}$-Higgs bundle data as special $\text{SL(4}\text{,}\mathbb{C}\text{)}$-data in the original sense of N. Hitchin \cite{Hit87}. In particular, an $\text{Sp(4}\text{,}\mathbb{R}\text{)}$-Higgs bundle is alternatively defined as a pair $\left( E,\Phi  \right)$, where 
\begin{enumerate}
\item $E=V\oplus {{V}^{*}}$ is a rank $4$ holomorphic vector bundle over $X$ and
\item $\Phi :E\to E\otimes K$ is a holomorphic $K$-valued endomorphism of $E$ with $\Phi =\left( \begin{matrix}
   0 & \beta   \\
   \gamma  & 0  \\
\end{matrix} \right)$,
\end{enumerate}
for $V$, $\beta$, $\gamma$ as above.

\subsection{$\text{Sp(4}\text{,}\mathbb{R}\text{)}$-Hitchin equations}
For the complexified Lie algebra $\mathfrak{sp}\left( 4,\mathbb{C} \right)$, notice that the involution $\sigma :\mathfrak{sp}\left( 4,\mathbb{C} \right)\to \mathfrak{sp}\left( 4,\mathbb{C} \right)$, $\sigma \left( X \right)=\bar{X}$ defines the split real form
\[\mathfrak{sp}\left( 4,\mathbb{R} \right) = \left\{ X\in \mathfrak{sp}\left( 4,\mathbb{C} \right)\left| \sigma \left( X \right)=X \right. \right\},\]
while the involution $\tau :\mathfrak{sp}\left( 4,\mathbb{C} \right)\to \mathfrak{sp}\left( 4,\mathbb{C} \right)$, $\tau \left( X \right)=-{{X}^{*}}$ defines the compact real form
\[\mathfrak{sp}\left( 2 \right) = \mathfrak{sp}\left( 4,\mathbb{C} \right)\cap \mathfrak{u}\left( 4 \right)= \left\{ X\in \mathfrak{sp}\left( 4,\mathbb{C} \right)\left| \tau \left( X \right)=X \right. \right\}.\]
Since $\tau $ and the Cartan involution commute, we have $\tau \left( {{\mathfrak{m}}^{\mathbb{C}}} \right)\subseteq {{\mathfrak{m}}^{\mathbb{C}}}$ and then $\tau $ preserves the Cartan decomposition $\mathfrak{sp}\left( 4,\mathbb{C} \right)=\mathfrak{gl}\left( 2,\mathbb{C} \right)\oplus {{\mathfrak{m}}^{\mathbb{C}}}$. Thus, there is an induced real form on $E\left( {{\mathfrak{m}}^{\mathbb{C}}} \right)$ which we shall call $\tau $ as well for simplicity. Now, it makes sense to apply $\tau $ on a section $\Phi \in {{\Omega }^{1,0}}\left( E\left( {{\mathfrak{m}}^{\mathbb{C}}} \right) \right)$.

Moreover, for $\Phi =\left( \begin{matrix}
   0 & \beta   \\
   \gamma  & 0  \\
\end{matrix} \right)$ we check that
\[-\left[ \Phi ,\tau \left( \Phi  \right) \right]=\left[ \Phi ,{{\Phi }^{*}} \right]=\left( \begin{matrix}
   \beta \bar{\beta }-\bar{\gamma }\gamma  & 0  \\
   0 & \gamma \bar{\gamma }-\bar{\beta }\beta   \\
\end{matrix} \right).\]

Thus, the general $G$-Hitchin equations when $G=\text{Sp(4}\text{,}\mathbb{R}\text{)}$ can be described in terms of the special  $\text{SL(4}\text{,}\mathbb{C}\text{)}$-data $\left( E=V\oplus {{V}^{*}},\Phi =\left( \begin{matrix}
   0 & \beta   \\
   \gamma  & 0  \\
\end{matrix} \right)  \right)$ as 

\begin{align*}
   F_{A}-\left[ \Phi ,\tau \left( \Phi  \right) \right] & = 0  \\
   {{{\bar{\partial }}}_{A}}\Phi & = 0,
\end{align*}
where $ F_{A}$ is the curvature of a connection on $E \to X$ and ${{\bar{\partial }}_{A}}$ is the anti-holomorphic covariant derivative induced by $A$.

Recall that a $\text{GL(4}\text{,}\mathbb{C}\text{)}$-Higgs bundle $\left( E,\Phi  \right)$ is called \emph{stable} if any proper non-zero $\Phi$-invariant subbundle $F\subseteq E$ satisfies $\mu \left( F \right)<\mu \left( E \right)$, for $\mu \left( F \right)={\deg \left( F \right)}/{\text{rk}\left( F \right)}\;$, the slope of the bundle. One has the following proposition:

\begin{proposition}[Theorem 3.26 in \cite{GGMsymplectic}] An $\mathrm{Sp(4}\text{,}\mathbb{R}\text{)}$-Higgs bundle $\left( V,\beta ,\gamma  \right)$ is polystable if and only if the $\mathrm{GL(4}\text{,}\mathbb{C}\text{)}$-Higgs bundle $\left( E = V\oplus {{V}^{*}},\Phi =\left( \begin{matrix}
   0 & \beta   \\
   \gamma  & 0  \\
\end{matrix} \right) \right)$ is polystable. Moreover, even though the polystability conditions coincide, the stability condition for an $\mathrm{Sp(4}\text{,}\mathbb{R}\text{)}$-Higgs bundle is in general weaker than the stability condition for the corresponding $\mathrm{GL(4}\text{,}\mathbb{C}\text{)}$-Higgs bundle.
\end{proposition}

\subsection{Toledo invariant and Cayley partner}\label{Toledo_inv}
A basic topological invariant for an $\text{Sp}\left( 4,\mathbb{R} \right)$-Higgs bundle $\left( V,\beta ,\gamma  \right)$ is given by the degree of the underlying bundle \[d=\deg \left( V \right).\]
This invariant, called the \emph{Toledo invariant}, labels only partially the connected components of the moduli space $\mathsf{\mathcal{M}}\left( \text{Sp(}4,\mathbb{R}\text{)} \right)$. We use the notation ${{\mathsf{\mathcal{M}}}_{d}}={{\mathsf{\mathcal{M}}}_{d}}\left( \text{Sp}(4,\mathbb{R}) \right)$ to denote the moduli space parameterizing isomorphism classes of polystable $\text{Sp}\left( 4,\mathbb{R} \right)$-Higgs bundles with $\deg \left( V \right)=d$. The sharp bound below for the Toledo invariant when $G=\text{Sp}\left( 4,\mathbb{R} \right)$ was first given by V. Turaev \cite{Turaev}:

\begin{proposition}[Milnor-Wood inequality]\label{Milnor_Wood_classical}
Let $\left( V,\beta ,\gamma  \right)$ be a semistable $\mathrm{Sp}\left( 4,\mathbb{R} \right)$-Higgs bundle. Then $\left| d \right|\le 2g-2$.
\end{proposition}

\begin{definition}
We shall call $\text{Sp}\left( 4,\mathbb{R} \right)$-Higgs bundles with Toledo invariant $d=2g-2$ \emph{maximal} and denote the subspace of $\mathsf{\mathcal{M}}\left( \text{Sp}(4,\mathbb{R}) \right)$ consisting of Higgs bundles with maximal positive Toledo invariant by ${{\mathsf{\mathcal{M}}}^{\max }}\simeq {{\mathsf{\mathcal{M}}}_{2g-2}}$.
\end{definition}

The proof of Proposition \ref{Milnor_Wood_classical} given by P. Gothen in \cite{Gothen} in the language of Higgs bundles opens the way to considering new topological invariants for our Higgs bundles in order to count the exact number of components of ${{\mathsf{\mathcal{M}}}^{\max }}$. Namely, one sees from that proof that for a maximal semistable $\text{Sp}\left( 4,\mathbb{R} \right)$-Higgs bundle  $\left( V,\beta ,\gamma  \right)$, the map $\gamma :V\to {{V}^{*}}\otimes K$ is an \emph{isomorphism}.

For a fixed square root of the canonical bundle $K$, that is, a line bundle $L_{0}$ such that $L_{0}^{2} \cong K$, the isomorphism $\gamma$ can be used to construct an $\text{O}\left( 2,\mathbb{C} \right)$-holomorphic bundle $\left( W,{{q}_{W}} \right)$, for $W:={{V}^{*}}\otimes {{L}_{0}}$ and ${{q}_{W}}:=\gamma \otimes {{I}_{L_{0}^{-1}}}:{{W}^{*}}\xrightarrow{\simeq } W$. The Stiefel-Whitney classes $w_{1}$, $w_{2}$ of this orthogonal bundle $\left( W,{{q}_{W}} \right)$, which is called the \emph{Cayley partner} of the $\text{Sp}\left( 4,\mathbb{R} \right)$-Higgs bundle $\left( V,\beta ,\gamma  \right)$, are now appropriate topological invariants to study the topology of the moduli space ${{\mathsf{\mathcal{M}}}^{\max }}$. The classification of $\text{O}\left( 2,\mathbb{C} \right)$-holomorphic bundles by D. Mumford in \cite{Mumford} provides that for a rank 2 orthogonal bundle $\left( W,{{q}_{W}} \right)$ with ${{w}_{1}}\left( W,{{q}_{W}} \right)=0$, one has $W=L\oplus {{L}^{-1}}$ for $L\to X$ a line bundle and ${{q}_{W}}=\left( \begin{matrix}
   0 & 1  \\
   1 & 0  \\
\end{matrix} \right)$, whereas the stability condition imposes that
\[0\le \deg \left( L \right)\le 2g-2.\]
This way, the degree $\deg \left( L \right)$ introduces an additional invariant into the study of components of the moduli space ${{\mathsf{\mathcal{M}}}^{\max }}$. In fact, when $\deg \left( L \right)=2g-2$ the connected components are parameterized by spin structures on the surface $\Sigma$. Using Morse theory techniques and a careful study of the closed subvarieties corresponding to all possible values of the invariants $w_{1}$, $w_{2}$ and $\text{deg}(L)$, it was shown in \cite{Gothen} that the total number of connected components of the moduli space ${{\mathsf{\mathcal{M}}}^{\max }}$ is $3\cdot {{2}^{2g}}+2g-4$, for $g$ the genus of the Riemann surface $X$.

\subsection{Maximal fundamental group representations into $\text{Sp}\left( 4,\mathbb{R} \right)$}\label{maximal_ representations}
From an alternative point of view, the non-abelian Hodge theorem provides a homeomorphism of ${{\mathsf{\mathcal{M}}}^{\max }}$ to a moduli space of representations ${{\mathsf{\mathcal{R}}}^{\max }}$, particular points of which will be  briefly described next.

Let $G$ be a Hermitian Lie group of non-compact type, that is, the symmetric space associated to $G$ is an irreducible Hermitian symmetric space of non-compact type. Using the identification ${{H}^{2}}\left( {{\pi }_{1}}\left( \Sigma  \right),\mathbb{R} \right)\simeq {{H}^{2}}\left( \Sigma ,\mathbb{R} \right)$, the \emph{Toledo invariant} of a representation $\rho :{{\pi }_{1}}\left( \Sigma  \right)\to G$ is defined as the integer
	\[{{T}_{\rho }}:=\left\langle {{\rho }^{*}}\left( {{\kappa }_{G}} \right),\left[ \Sigma  \right] \right\rangle, \] where ${{\rho }^{*}}\left( {{\kappa }_{G}} \right)$ is the pullback of the K{\"a}hler class ${{\kappa }_{G}}\in H_{c}^{2}\left( G,\mathbb{R} \right)$ of $G$ and $\left[ \Sigma  \right]\in {{H}_{2}}\left( \Sigma ,\mathbb{R} \right)$ is the orientation class. It is bounded in absolute value,	$\left| {{T}_{\rho }} \right|\le -C\left( G \right)\chi \left( \Sigma  \right)$, where $C\left( G \right)$ is an explicit constant depending only on $G$; we refer the reader to \cite{BIW} for more details.

\begin{definition} A representation $\rho :{{\pi }_{1}}\left( \Sigma  \right)\to G$ is called \emph{maximal} whenever the Toledo invariant ${{T}_{\rho }}=-C\left( G \right)\chi \left( \Sigma  \right)$.
\end{definition}

Note that the Toledo invariant of a representation $\rho :{{\pi }_{1}}\left( \Sigma  \right)\to \text{Sp}\left( 4,\mathbb{R} \right)$ coincides with the Toledo invariant of an $\text{Sp}\left( 4,\mathbb{R} \right)$-Higgs bundle as reviewed in \S \ref{Toledo_inv}; we refer to \cite{HaOt} for a broader discussion relating the Milnor-Wood inequality in these two contexts. The next theorem distinguishes a family of connected components of maximal representations $\rho :{{\pi }_{1}}\left( \Sigma  \right)\to \text{Sp}\left( 4,\mathbb{R} \right)$ of special geometric significance; its proof was obtained from the Higgs bundle point of view: 

\begin{theorem}[Theorem 1.1 in \cite{BGGDeformations}]
There are $2g-3$ connected components of ${{\mathsf{\mathcal{M}}}^{\max }}\simeq {{\mathsf{\mathcal{R}}}^{\max }}$, in which the corresponding representations do not factor through any proper reductive subgroup of $\mathrm{Sp}\left( 4,\mathbb{R} \right)$, thus they have Zariski-dense image in $\mathrm{Sp}\left( 4,\mathbb{R} \right)$.
\end{theorem}

In \cite{GW}, O. Guichard and A. Wienhard describe model maximal fundamental group representations $\rho :{{\pi }_{1}}\left( \Sigma  \right)\to \text{Sp(4}\text{,}\mathbb{R}\text{)}$ in components of ${{\mathsf{\mathcal{R}}}^{\max }}$. These models are distinguished into two subcategories, namely \emph{standard representations} and \emph{hybrid representations}.

We review next the construction of these model representations in further detail with particular attention towards the construction of the hybrid representations. Fix a discrete embedding $i:{{\pi }_{1}}\left( \Sigma  \right)\to \text{SL}\left( 2,\mathbb{R} \right)$.

i) \emph{Irreducible Fuchsian representations}

Let ${{V}_{0}}={{\mathbb{R}}_{1}}\left[ X,Y \right]\cong {{\mathbb{R}}^{2}}$ be the space of homogeneous polynomials of degree 1 in the variables $X$ and $Y$ with the symplectic form ${{\omega }_{0}}\left( X,Y \right)=1$. The induced action of $\text{Sp}\left( {{V}_{0}} \right)\cong \text{SL}\left( 2,\mathbb{R} \right)$ on $V=\text{Sy}{{\text{m}}^{3}}{{V}_{0}}\cong {{\mathbb{R}}_{3}}\left[ X,Y \right]\cong {{\mathbb{R}}^{4}}$ preserves the symplectic form ${{\omega }_{2}}=\text{Sy}{{\text{m}}^{3}}{{\omega }_{0}}$. Choose the symplectic identification $\left( {{\mathbb{R}}_{3}}\left[ X,Y \right],-{{\omega }_{2}} \right)\cong \left( {{\mathbb{R}}^{4}},\omega  \right)$ given by ${{X}^{3}}={{e}_{1}},\,{{X}^{2}}Y=-{{e}_{2}},\,{{Y}^{3}}=-{{e}_{3}},\,X{{Y}^{2}}=\frac{-{{e}_{4}}}{\sqrt{3}}$, where $\omega$ is the symplectic form given by the antisymmetric matrix $J=\left( \begin{matrix}
   0 & \text{I}{{\text{d}}_{n}}  \\
   -\text{I}{{\text{d}}_{n}} & 0  \\
\end{matrix} \right)$. With respect to this identification the irreducible representation ${{\phi }_{irr}}:\text{SL}\left( 2,\mathbb{R} \right)\to \text{Sp}\left( 4,\mathbb{R} \right)$ is given by

\begin{equation}\label{irreducible_rep}
{{\phi }_{irr}}\left( \begin{matrix}
   a & b  \\
   c & d  \\
\end{matrix} \right)=\left( \begin{matrix}
   {{a}^{3}} & -\sqrt{3}{{a}^{2}}b & -{{b}^{3}} & -\sqrt{3}a{{b}^{2}}  \\
   -\sqrt{3}{{a}^{2}}c & 2abc+{{a}^{2}}d & \sqrt{3}{{b}^{2}}d & 2abd+{{b}^{2}}c  \\
   -{{c}^{3}} & \sqrt{3}{{c}^{2}}d & {{d}^{3}} & \sqrt{3}c{{d}^{2}}  \\
   -\sqrt{3}a{{c}^{2}} & 2acd+b{{c}^{2}} & \sqrt{3}b{{d}^{2}} & 2bcd+a{{d}^{2}}  \\
\end{matrix} \right).
\end{equation}
Precomposition with $i:{{\pi }_{1}}\left( \Sigma  \right)\to \text{SL}\left( 2,\mathbb{R} \right)$ gives rise to an \emph{irreducible Fuchsian representation} ${{\rho }_{irr}}:{{\pi }_{1}}\left( \Sigma  \right)\xrightarrow{i}\text{SL}\left( 2,\mathbb{R} \right)\xrightarrow{{{\phi }_{irr}}}\text{Sp}\left( 4,\mathbb{R} \right)$.

ii) \emph{Diagonal Fuchsian representations}

Let ${{\mathbb{R}}^{4}}={{W}_{1}}\oplus {{W}_{2}}$, with ${{W}_{i}}=\text{span}\left( {{e}_{i}},{{e}_{2+i}} \right)$ be a symplectic splitting of ${{\mathbb{R}}^{4}}$ with respect to the symplectic basis ${{\left( {{e}_{i}} \right)}_{i=1,\ldots ,4}}$. This splitting gives rise to an embedding $\psi :\text{SL}{{\left( 2,\mathbb{R} \right)}^{2}}\to \text{Sp}\left( {{W}_{1}} \right)\times \text{Sp}\left( {{W}_{2}} \right)\subset \text{Sp}\left( 4,\mathbb{R} \right)$ given by
\begin{equation}\label{diagonal_Fuchsian}
\psi \left( \left( \begin{matrix}
   a & b  \\
   c & d  \\
\end{matrix} \right),\left( \begin{matrix}
   \alpha  & \beta   \\
   \gamma  & \delta   \\
\end{matrix} \right) \right)=\left( \begin{matrix}
   a & 0 & b & 0  \\
   0 & \alpha  & 0 & \beta   \\
   c & 0 & d & 0  \\
   0 & \gamma  & 0 & \delta   \\
\end{matrix} \right).
\end{equation}
Precomposition with the diagonal embedding of $\text{SL}\left( 2,\mathbb{R} \right)\to \text{SL}{{\left( 2,\mathbb{R} \right)}^{2}}$ gives rise to the diagonal embedding ${{\phi }_{\Delta }}:\text{SL}\left( 2,\mathbb{R} \right)\to \text{Sp}\left( 4,\mathbb{R} \right)$, while precomposition with $i:{{\pi }_{1}}\left( \Sigma  \right)\to \text{SL}\left( 2,\mathbb{R} \right)$ gives now rise to a \emph{diagonal Fuchsian representation} ${{\rho }_{\Delta}}:{{\pi }_{1}}\left( \Sigma  \right)\xrightarrow{i}\text{SL}\left( 2,\mathbb{R} \right)\xrightarrow{{{\phi }_{\Delta}}}\text{Sp}\left( 4,\mathbb{R} \right)$.

iii) \emph{Twisted diagonal representations}

For any maximal representation $\rho :{{\pi }_{1}}\left( \Sigma  \right)\to \text{Sp}\left( 4,\mathbb{R} \right)$ the centralizer $\rho \left( {{\pi }_{1}}\left( \Sigma  \right) \right)$ is a subgroup of $\text{O}\left( 2 \right)$. Considering a representation $\Theta :{{\pi }_{1}}\left( \Sigma  \right)\to \text{O}\left( 2 \right)$, set \begin{align*}
{{\rho }_{\Theta }}=i\otimes \Theta :{{\pi }_{1}}\left( \Sigma  \right) & \to \text{Sp}\left( 4,\mathbb{R} \right)\\
  \gamma & \mapsto {{\phi }_{\Delta }}\left( i\left( \gamma  \right),\Theta \left( \gamma  \right) \right).
\end{align*}
Such a representation will be called a \emph{twisted diagonal representation}.

\begin{remark}
The representations in the families (i)-(iii) above are the so-called \emph{standard representations}.
\end{remark}

iv) \emph{Hybrid representations}

The definition of hybrid representations involves a gluing construction for fundamental group representations over a connected sum of surfaces. Let $\Sigma ={{\Sigma }_{l}}{{\cup }_{\gamma }}{{\Sigma }_{r}}$ be a decomposition of the surface $\Sigma $ along a simple closed oriented separating geodesic $\gamma $ into two subsurfaces ${{\Sigma }_{l}}$ and ${{\Sigma }_{r}}$. Pick ${{\rho }_{irr}}:{{\pi }_{1}}\left( \Sigma  \right)\to \text{SL}\left( 2,\mathbb{R} \right)\xrightarrow{{{\phi }_{irr}}}\text{Sp}\left( 4,\mathbb{R} \right)$ an irreducible Fuchsian representation and ${{\rho }_{\Delta }}:{{\pi }_{1}}\left( \Sigma  \right)\to \text{SL}\left( 2,\mathbb{R} \right)\xrightarrow{\Delta }\text{SL}{{\left( 2,\mathbb{R} \right)}^{2}}\to \text{Sp}\left( 4,\mathbb{R} \right)$ a diagonal Fuchsian representation. One could amalgamate the restriction of the irreducible Fuchsian representation ${{\rho }_{irr}}$ to ${{\Sigma }_{l}}$ with the restriction of the diagonal Fuchsian representation ${{\rho }_{\Delta }}$ to ${{\Sigma }_{r}}$, however the holonomies of those along $\gamma $ a priori do not agree. A deformation of ${{\rho }_{\Delta }}$ on ${{\pi }_{1}}\left( \Sigma  \right)$ can be considered, such that the holonomies would agree along $\gamma $, thus allowing the amalgamation operation. This continuous deformation is defined in \S 3.3.1 of \cite{GW} using continuous paths of embeddings ${{\pi }_{1}}\left( \Sigma  \right)\to \text{Sp}\left( 4,\mathbb{R} \right)$, which have the fixed discrete embedding $\iota: {{\pi }_{1}}\left( \Sigma  \right)\to \text{SL}\left( 2,\mathbb{R} \right)$ as their initial point and as an end point, appropriately chosen embeddings, say $\tau_{1}$, $\tau_{2}$, with diagonal holonomy. Composing the pair $\left( \tau_{1}, \tau_{2} \right)$ with the map $\psi$ defined in (\ref{diagonal_Fuchsian}) finally gives rise to a continuous deformation $\rho_{r}$ of $\rho_{\Delta}$, such that by construction it satisfies ${{\rho }_{r}}\left( \gamma  \right)={{\rho }_{l}}\left( \gamma  \right)$, where $\rho_{l} \equiv \rho_{irr}$. This introduces new representations by gluing:

\begin{definition}  A \emph{hybrid representation} is defined as the amalgamated representation
\[\rho :={{\rho }_{l}}\left| _{{{\pi }_{1}}\left( {{\Sigma }_{l}} \right)} \right.*{{\rho }_{r}}\left| _{{{\pi }_{1}}\left( {{\Sigma }_{r}} \right)} \right.:{{\pi }_{1}}\left( \Sigma  \right)\simeq {{\pi }_{1}}\left( {{\Sigma }_{l}} \right){{*}_{\left\langle \gamma  \right\rangle }}{{\pi }_{1}}\left( {{\Sigma }_{r}} \right)\to \text{Sp}\left( 4,\mathbb{R} \right).\]
\end{definition}

The following important result was established in \cite{GW}:

\begin{theorem}[Theorem 14 in \cite{GW}]
Every maximal representation $\rho :{{\pi }_{1}}\left( \Sigma  \right)\to \mathrm{Sp}\left( 4,\mathbb{R} \right)$ can be deformed to a standard representation or a hybrid representation.
\end{theorem}

The subsurfaces ${{\Sigma }_{l}}$ and ${{\Sigma }_{r}}$ that we are considering here are surfaces with boundary. A notion of Toledo invariant can be also defined for representations over such surfaces and it thus makes sense to talk about maximal representations over surfaces with boundary as well; see \cite{BIW} for a detailed definition. Moreover, the authors in \cite{BIW} have established an additivity property for the Toledo invariant over a connected sum of surfaces. In particular:

\begin{proposition}[Proposition 3.2 in \cite{BIW}]\label{additivity_Toledo_rep}
 If $\Sigma ={{\Sigma }_{l}}{{\cup }_{\gamma}}{{\Sigma }_{r}}$ is the connected sum of two subsurfaces ${{\Sigma }_{i}}$ along a separating loop $\gamma$, then
\[{{\text{T}}_{\rho }}={{\text{T}}_{\rho_{1} }} + {{\text{T}}_{\rho_{2}}},\]
where ${{\rho }_{i}}=\rho \left| _{{{\pi }_{1}}\left( {{\Sigma }_{i}} \right)} \right.$, for $i=l,r$.
\end{proposition}
Note that this property implies that the amalgamated product of two maximal representations is again a maximal representation defined over the compact surface $\Sigma$.

\section{Parabolic $\text{Sp(4}\text{,}\mathbb{R}\text{)}$-Higgs bundles}

Parabolic Higgs bundles were first considered as the holomorphic objects over a non-compact curve that correspond to fundamental group representations with fixed arbitrary holonomy around the boundary of the surface. Examples of primary reference include \cite{BoYo}, \cite{GGmunoz}, \cite{Konno}, \cite{Simpson-noncompact}. For our considerations in this article, we will be using specific parabolic $\text{SL(2}\text{,}\mathbb{R}\text{)}$ and $\text{Sp(4}\text{,}\mathbb{R}\text{)}$-Higgs bundle pairs to be described in this section.

\subsection{Parabolic $\text{SL}\left( 2,\mathbb{R} \right)$-Higgs bundles}\label{section_parabolic_SL2}

In \cite{BiswasAresGovin}, the authors consider parabolic Higgs bundles that correspond to Fuchsian representations on a punctured Riemann surface, thus generalizing the fundamental result of N. Hitchin in \cite{Hit92} on the construction of the Teichm\"{u}ller space via Higgs bundles in the absence of punctures. For this result, a specific choice of a parabolic structure is made in \cite{BiswasAresGovin}, namely a trivial flag with weight $\frac{1}{2}$ is giving rise to a Poincar\'{e} metric on the holomorphic tangent bundle on the punctured Riemann surface. We review this family of parabolic   $\text{SL}\left( 2,\mathbb{R} \right)$-Higgs bundles next as it will play an important role in constructing our higher rank models. 

Let $X$ be a compact Riemann surface of genus $g$ and let a divisor of $s$-many distinct points $D=\left\{ {{x}_{1}},\ldots ,{{x}_{s}} \right\}$ from $X$, such that $2g-2+s > 0$. The punctured surface $X\backslash D$ thus admits a metric of constant negative curvature (-4) and let $K$ be the canonical line bundle over $X$. We consider a pair $\left( E, \Phi \right)$ as follows: 
\begin{enumerate}
  \item $E:={{\left( L\otimes \iota  \right)}^{*}}\oplus L$, \\
  where $L$ is a line bundle with ${{L}^{2}}={{K}}$ and $\iota :={{\mathsf{\mathcal{O}}}_{X}}\left( D \right)$ denotes the line bundle over the divisor $D$; we equip the bundle $E$ with a parabolic structure given by a trivial flag ${{E}_{{{x}_{i}}}}\supset \left\{ 0 \right\}$ and weight $\frac{1}{2}$ for every $1\le i\le s$,
  \item $\Phi :=\left( \begin{matrix}
   0 & 1  \\
   0 & 0  \\
\end{matrix} \right)\in {{H}^{0}}\left( X,\text{End}\left( E \right)\otimes K \otimes \iota \right)$.
\end{enumerate}
The pair $\left( E, \Phi \right)$ is a parabolic Higgs bundle (for us, a parabolic $\text{SL}\left( 2,\mathbb{R} \right)$-Higgs bundle) and one can easily see that it is stable and of parabolic degree zero (Lemma 2.1 in \cite{BiswasAresGovin}). Therefore, from the non-abelian Hodge correspondence on non-compact curves \cite{Simpson-noncompact}, the vector bundle $E$ supports a tame harmonic metric; the local estimate for this hermitian metric on $E$ restricted to the line bundle $L$ is
\[{{r}^{\frac{1}{2}}}{{\left| \log r \right|}^{\frac{1}{2}}},\] 
for $r=\left| z \right|$. Consequently, the metric on the tangent bundle ${{L}^{-2}}$ is locally ${{r}^{-1}}{{\left| \log r \right|}^{-1}}$ and is therefore the Poincar\'{e} metric of the punctured disk on $\mathbb{C}$. The authors in \cite{BiswasAresGovin} now showed that Fuchsian representations of ${{\pi }_{1}}\left( X\backslash D \right)$ into $\mathrm{PSL}\left( 2,\mathbb{R} \right)$ are in one-to-one correspondence with parabolic $\mathrm{SL}\left( 2,\mathbb{R} \right)$-Higgs bundles of the form $\left( E,\theta  \right)$ for $E \to X$ a parabolic rank 2 bundle as above and $\theta :=\left( \begin{matrix}
   0 & 1  \\
   a & 0  \\
\end{matrix} \right)\in {{H}^{0}}\left( X,\mathrm{End}\left( E \right)\otimes K \otimes \iota \right)$, for a meromorphic quadratic differential $a\in {{H}^{0}}\left( X,{{K}^{2}} \otimes \iota \right)$.

The family $\left( E,\theta  \right)$ describes a Hitchin-Teichm\"{u}ller component over a punctured Riemann surface; it has real dimension $6g-6+2s$. The result was extended in \cite{KSZ} for higher rank split Lie groups.

\subsection{Parabolic $\text{Sp(4}\text{,}\mathbb{R}\text{)}$-Higgs bundles}\label{section_parabolic_Sp4}

In this article we are interested in parabolic Higgs bundles with structure group $G=\text{Sp(4}\text{,}\mathbb{R}\text{)}$. We consider those as special parabolic $\text{GL(4}\text{,}\mathbb{C}\text{)}$-Higgs pairs in the sense of \cite{BoYo} or \cite{GGmunoz}. A general theory of parabolic $G$-Higgs bundles for a non-compact real reductive Lie group $G$ was provided by O. Biquard, O. Garc\'{i}a-Prada and I. Mundet i Riera in \cite{BiquardGM}; a detailed exposition on how the general definition of \cite{BiquardGM} specializes to the following in the case when $G=\text{Sp(4}\text{,}\mathbb{R}\text{)}$ can be found in Example A.25 in \cite{KSZ}.

\begin{definition} Let $X$ be a compact Riemann surface of genus $g$ and let the divisor $D:=\left\{ {{x}_{1}},\ldots ,{{x}_{s}} \right\}$ of $s$-many distinct points on $X$, assuming that $2g-2+s>0$. Fix a line bundle $\iota :={{\mathsf{\mathcal{O}}}_{X}}\left( D \right)$ over the divisor $D$. A \emph{parabolic $\mathrm{Sp(4}\text{,}\mathbb{R}\text{)}$-Higgs bundle} is defined as a triple $\left( V,\beta ,\gamma  \right)$, where
\begin{itemize}
\item $V$ is a rank 2 bundle on $X$, equipped with a parabolic structure given by a flag ${{V}_{x}}\supset {{L}_{x}}\supset 0$ and weights $0\le {{\alpha }_{1}}\left( x \right)<{{\alpha }_{2}}\left( x \right)<1$ for every $x\in D$, and
\item $\beta :{{V}^{\vee }}\to V\otimes K\otimes \iota$ and $\gamma :V\to {{V}^{\vee }}\otimes K\otimes \iota$ are strongly parabolic symmetric homomorphisms,
\end{itemize}
for ${{V}^{\vee }}:={{\left( V\otimes \iota  \right)}^{*}}$, the parabolic dual of the parabolic bundle $V$.
\end{definition}

The parabolic degree of the parabolic bundle $V$ is given by the rational number 
\[par \text{deg} \left( V \right) = \text{deg} \left( V \right) + \sum\limits_{{{x}_{i}}\in D}{\left( {{\alpha }_{1}}\left( {{x}_{i}} \right)+{{\alpha }_{2}}\left( {{x}_{i}} \right) \right)}.\]
The parabolic structures on $V$ and ${{V}^{\vee }}$ now induce a parabolic structure on the parabolic sum $E=V\oplus {{V}^{\vee }}$, for which $ par \text{deg} E=0$. We define alternatively a parabolic $\mathrm{Sp}\left( 4,\mathbb{R} \right)$-Higgs bundle on $\left( X, D \right)$ as a parabolic Higgs bundle $\left( E,\Phi  \right)$, where $E=V\oplus {{V}^{\vee }}$ and $\Phi =\left( \begin{matrix}
   0 & \beta   \\
   \gamma  & 0  \\
\end{matrix} \right):E\to E\otimes K\otimes \iota$.

A notion of parabolic Toledo invariant can still be considered:
\begin{definition}\label{par_toledo_Sp4}
The \emph{parabolic Toledo invariant} of a parabolic $\text{Sp}\left( 4,\mathbb{R} \right)$-Higgs bundle is defined as the rational number
\[\tau =par\deg \left( V \right).\]
\end{definition}

Moreover, we get a \emph{Milnor-Wood type inequality} for this topological invariant:
\begin{proposition}
Let $\left( E,\Phi  \right)$ be a semistable parabolic $\mathrm{Sp}\left( 4,\mathbb{R} \right)$-Higgs bundle over a pair $(X, D)$ defined as above. Then \[\left| \tau  \right|\le 2g-2+s,\]
where $s$ is the number of points in $D$.
\end{proposition}

\begin{proof}
See Proposition 5.4 in \cite{KSZ}.
\end{proof}

\begin{definition}
 The parabolic $\text{Sp}\left( 4,\mathbb{R} \right)$-Higgs bundles with parabolic Toledo invariant $\tau =2g-2+s$ will be called \emph{maximal} and we will denote the components containing such triples by $\mathsf{\mathcal{M}}_{par}^{\max }:=\mathsf{\mathcal{M}}_{par}^{2g-2+s}$.
\end{definition}

In \cite{KSZ} a component count for the moduli space $\mathsf{\mathcal{M}}_{par}^{\max }$ was obtained. Note that maximal parabolic  $\text{Sp}\left( 2n,\mathbb{R} \right)$-Higgs bundles can be considered for more general choices of weights, since the proof for the maximality of the Toledo invariant does not depend on the parabolic structure. A component count in these more general cases can be found in \cite{KSZ2}.

\section{Approximate solutions by gluing}

In this section we develop a gluing construction for solutions to the $\text{Sp(4}\text{,}\mathbb{R}\text{)}$-Hitchin equations over a connected sum of Riemann surfaces to produce an approximate solution to the equations. The necessary condition in order to combine the initial parabolic data over the connected sum operation is that this data is identified over annuli around the points in the divisors of the Riemann surfaces. Aiming to provide new model Higgs bundles in the exceptional components of ${{\mathsf{\mathcal{M}}}^{\max }}$, we consider parabolic data which around the punctures are \emph{a priori not identified}, but we will rather seek for deformations of those into model solutions of the Hitchin equations which will allow us to combine data over the complex connected sum. This deformation argument is coming from deformations of $\text{SL(2}\text{,}\mathbb{R}\text{)}$-solutions to the Hitchin equations over a punctured surface and subsequently we extend this for $\text{Sp(4}\text{,}\mathbb{R}\text{)}$-pairs using appropriate embeddings $\phi: \text{SL(2}\text{,}\mathbb{R}\text{)}\hookrightarrow \text{Sp(4}\text{,}\mathbb{R}\text{)}$. Therefore, our gluing construction involves parabolic $\text{Sp(4}\text{,}\mathbb{R}\text{)}$-pairs which arise from $\text{SL(2}\text{,}\mathbb{R}\text{)}$-pairs via extensions by such embeddings.

\subsection{The local model}\label{section_local_model}
Similarly to the non-parabolic case, the moduli space of stable parabolic Higgs bundles can be identified with the moduli space of solutions to the parabolic version of the Hitchin equations:
\begin{align}\label{Hit_eq_par}
{{F}_{A}^{\perp}}+\left[ \Phi , \Phi^{*} \right] &= 0  \\
{{\bar{\partial }}_{A}}\varphi &= 0, 
\end{align}
where ${{F}_{A}^{\perp}}$ is the trace-free part of the curvature of a connection $A$ which is a singular connection unitary with respect to a singular hermitian metric on a parabolic bundle adapted to the parabolic structure (see \S 3.5 of \cite{Mochizuki} for a detailed explanation). In the Corlette-Donaldson part of the non-abelian Hodge correspondence, the local monodromy of the associated flat connection $A+\Phi + \Phi^{*}$ around a point $x$ in the divisor $D$ is determined up to conjugacy by the parabolic weights and the eigenvalues of the residues of the Higgs field. 

Let $E$ be a smooth rank 2 parabolic bundle equipped with a full flag
\[{{E}_{x}}={{E}_{x,1}}\supset {{E}_{x,2}}\supset \left\{ 0 \right\}\]
and a pair of weights $\left( \alpha_{1}, \alpha_{2} \right)$. Fix $U$ a neighborhood of the parabolic point $x$ and let $z$ be a holomorphic local coordinate on $U$ with $z \left( x \right) = 0$. Let $ \left\{ e_{1}, e_{2} \right\}$ be a smooth frame on $E{\big|}_U$. Then, the singular hermitian metric 
\[h=\left( \begin{matrix}
   {{\left| z  \right|}^{2\alpha_{1}}}  & 0  \\
   0 & {{\left| z  \right|}^{2\alpha_{2}}}  \\
\end{matrix} \right)\]
is adapted to the parabolic structure. Moreover, with respect to the unitary frame 
\[\left\{ {{\left| z  \right|}^{-\alpha_{1}}} e_{1}, {{\left| z  \right|}^{-\alpha_{2}}} e_{2} \right\}\]
 the associated singular Chern connection is given by 
\[D_{h} = d+ \left( \begin{matrix}
   \alpha_{1} & 0  \\
   0 & \alpha_{2}  \\
\end{matrix} \right) i d\theta,\] 
where we write $z=re^{i\theta}$; note that $d\theta$ has a pole at the origin.

For fixed constants $\alpha \in \mathbb{R}$ and $C \in \mathbb{C}$, the pair
\[{{A}}= \left( \begin{matrix}
   \alpha & 0  \\
   0 & -\alpha  \\
\end{matrix} \right) \left( \frac{dz}{z} - \frac{d\bar{z}}{\bar{z}} \right)    ,\text{   }{{\Phi }}=\left( \begin{matrix}
   C & 0  \\
   0 & -C  \\
\end{matrix} \right)\frac{dz}{z}\]
describes a solution of the equations (4.1), (4.2) on $\mathbb{C}^{*}$ (cf. \S 2.3 in \cite{Swoboda}, where this model is used to study the moduli space of solutions to the Hitchin equations under a degeneration of a smooth Riemann surface to a nodal Riemann surface). For our purposes in this article, we will rather use the model pair for constants $\alpha = 0$ and $C \in \mathbb{R}^{+}$ providing a model solution, for which the local monodromy of the associated flat connection around the point $x \in D$ lies in $\text{SL} \left( 2, \mathbb{R} \right)$ (cf. \S 2.5 in \cite{Swoboda} and the references therein, where such a model is obtained by studying the behavior of the harmonic map between a surface $X$ with a given complex structure and the surface $X$ with the corresponding Riemannian metric of constant curvature -4, under degeneration of the domain Riemann surface $X$ to a nodal surface).

Thus, the \emph{model solution} to the $\text{SL(2}\text{,}\mathbb{R}\text{)}$-Hitchin equations we will be considering is described by
	\[{{A}^{\bmod }}=0,\text{   }{{\Phi }^{\bmod }}=\left( \begin{matrix}
   C & 0  \\
   0 & -C  \\
\end{matrix} \right)\frac{dz}{z}\]
over a punctured disk with $z$-coordinates around the puncture with the condition that $C\in \mathbb{R}$ with $C\ne 0$. Note that this pair is described only by the meromorphic quadratic differential $q:= \text{det} \Phi^{\text{mod}} = -C^{2} z^{-2}dz^{2}$ having a double pole at the point $x \in D$; we assume that $q$ has at least one simple zero-that this is indeed the generic case, is discussed in \cite{MSWW}. Lastly, we refer the reader to \S 3 of \cite{Lauraetal} for more examples of model solutions in this parabolic setting. 

\subsection{Weighted Sobolev spaces}\label{Sobolev}
In order to develop the necessary analytic arguments for the gluing construction later on, we need to introduce appropriate Sobolev spaces. Let $X$ be a compact Riemann surface and $D:=\left\{ {{p}_{1}},\ldots ,{{p}_{s}} \right\}$ be a collection of $s$-many distinct points on $X$. Moreover, let $\left( E,h \right)$ be a hermitian vector bundle on $E$. Choose an initial pair $\left( {{A}^{\bmod }},{{\Phi }^{\bmod }} \right)$ on $E$, such that in some unitary trivialization of $E$ around each point $p\in D$, the pair coincides with the local model from \S \ref{section_local_model}. Of course, on the interior of each region $X\backslash \left\{ {{p}} \right\}$ the pair $\left( {{A}^{\bmod }},{{\Phi }^{\bmod }} \right)$ need not satisfy the Hitchin equations.

For fixed local coordinates $z$ centered at each $p\in D$, let $r = \left| z \right|$ be the distance function from $p$. Using the measure ${{r}}drd\theta $ and a fixed weight $\delta >0$ define \emph{weighted ${{L}^{2}}$-based Sobolev spaces}
	\[L_{\delta }^{2}:=\left\{ f\in {{L}^{2}}\left( rdrd\theta  \right)\left| r^{- \delta -1}f \in {{L}^{2}}\left( rdrd\theta  \right) \right. \right\}\] and
\[H_{\delta }^{k}:=\left\{ u,{{\nabla }^{j}}u\in L_{\delta }^{2}\left( rdrd\theta \right),0\le j\le k \right\}.\]
We are interested in deformations of a connection $A$ and a Higgs field $\Phi$ such that the curvature of the connection $D=A+\Phi +{{\Phi }^{*}}$ remains $O\left( {{r}^{-2+\delta }} \right)$, that is, slightly better than ${{L}^{1}}$. We can then define \emph{global Sobolev spaces on $X$} as the spaces of admissible deformations of the model unitary connection ${{A}^{\bmod }}$ and the model Higgs field ${{\Phi }^{\bmod }}$ as
	\[\mathsf{\mathcal{A}}:=\mathsf{}\left\{ {{A}^{\bmod }}+\alpha \left| \alpha \in H_{-2+\delta }^{1}\left( {{\Omega }^{1}}\otimes \mathfrak{su}\left( E \right) \right) \right. \right\}\]
	and \[\mathsf{\mathcal{B}}:=\left\{ {{\Phi }^{\bmod }}+\varphi \left| \varphi \in H_{-2+\delta }^{1}\left( {{\Omega }^{1,0}}\otimes \text{End}\left( E \right) \right) \right. \right\}.\]
The space of unitary gauge transformations \[\mathsf{\mathcal{G}}=\left\{ g\in \text{U}\left( E \right),\,{{g}^{-1}}dg\in H_{-2+\delta }^{1}\left( {{\Omega }^{1}}\otimes \mathfrak{su}\left( E \right) \right) \right\}\] acts smoothly on $\mathsf{\mathcal{A}}$ and $\mathsf{\mathcal{B}}$ by
 \[{{g}^{*}}\left( A,\Phi  \right)=\left( {{g}^{-1}}Ag+{{g}^{-1}}dg,{{g}^{-1}}\Phi g \right),\] for a pair $\left( A,\Phi  \right)\in \mathsf{\mathcal{A}}\times \mathsf{\mathcal{B}}$.

These considerations allow us to introduce the moduli space of solutions which are close to the model solution over a punctured Riemann surface ${{X}^{\times }}:=X-D$ for some fixed parameter $C\in \mathbb{R}$:
	\[\mathsf{\mathcal{M}}\left( {{X}^{\times }} \right):=\frac{\left\{ \left( A,\Phi  \right)\in \mathsf{\mathcal{A}}\times \mathsf{\mathcal{B}}\left| \left( A,\Phi  \right)\text{ satisfies (4.1) and (4.2)} \right. \right\}}{\mathsf{\mathcal{G}}}.\]
This moduli space was explicitly constructed by H. Konno in \cite{Konno} as a hyperk{\"a}hler quotient, and is identified with the moduli space of stable parabolic Higgs bundles (see \cite{Simpson-noncompact}).

\subsection{Approximate solutions of the $\text{SL(2}\text{,}\mathbb{R}\text{)}$-Hitchin equations}\label{approximate_SL2}
In \S \ref{Sobolev} we saw that a point in the moduli space $\mathsf{\mathcal{M}}\left( {{X}^{\times }} \right)$ differs from a model pair $\left( {{A}^{\bmod }},{{\Phi }^{\bmod }} \right)$ by some element in $H_{-2+\delta }^{1}$. The following result by O. Biquard and P. Boalch shows that $\left( A,\Phi  \right)$ is asymptotically close to the model in a much stronger sense:

\begin{lemma}[Lemma 5.3 in \cite{BiBo}]\label{Biq_Bo_lemma}
For each point $p\in D$, let $\left( {{A}_{p}^{\bmod }},{{\Phi }_{p}^{\bmod }} \right)$ be a model pair as was defined in \S \ref{section_local_model}. If $\left( A,\Phi  \right)\in \mathsf{\mathcal{M}}\left( {{X}^{\times }} \right)$, then there exists a unitary gauge transformation $g\in \mathsf{\mathcal{G}}$ such that in a neighborhood of each point $p\in D$ one has
	\[{{g}^{*}}\left( A,\Phi  \right)=\left( {{A}^{\bmod }_{p}},{{\Phi }^{\bmod }_{p}} \right)+O\left( {{r}^{-1+\delta }} \right),\]
for a positive constant $\delta$.
\end{lemma}

The decay described in this lemma can be further improved by showing that in a suitable complex gauge transformation the point $\left( A,\Phi  \right)$ coincides precisely with the model near each puncture in $D$. With respect to the \emph{singular measure} ${{r}^{-1}}drd\theta $ on $\mathbb{C}$, we first introduce the Hilbert spaces
	\[L_{-1+\delta }^{2}\left( {{r}^{-1}}drd\theta  \right):=\left\{ u\in {{L}^{2}}\left( \mathbb{D} \right)\left| {{r}^{-\delta }}u\in {{L}^{2}}\left( {{r}^{-1}}drd\theta  \right) \right. \right\},\]
	\[H_{-1+\delta }^{k}\left( {{r}^{-1}}drd\theta  \right):=\left\{ u\in {{L}^{2}}\left( \mathbb{D} \right)\left| {{\left( r\partial r \right)}^{j}}\partial _{\theta }^{l}u\in L_{-1+\delta }^{2}\left( {{r}^{-1}}drd\theta  \right),0\le j+l\le k \right. \right\}\]
for $\mathbb{D}=\left\{ z\in \mathbb{C}\left| 0<\left| z \right|<1 \right. \right\}$ the punctured unit disk.
We then have the following result by J. Swoboda:

\begin{lemma}[Lemma 3.2 in \cite{Swoboda}]\label{complex_gauge_to_model}
Let $\left( A,\Phi  \right)\in \mathsf{\mathcal{M}}\left( {{X}^{\times }} \right)$ and let $\delta$ be the constant provided by Lemma \ref{Biq_Bo_lemma}. Fix another constant $0<{\delta }'<\min \left\{ \frac{1}{2},\delta  \right\}$. Then there is a complex gauge transformation $g=\exp \left( \gamma  \right)\in {{\mathsf{\mathcal{G}}}^{\mathbb{C}}}$ with $\gamma \in H_{-1+{\delta }'}^{2}\left( {{r}^{-1}}drd\theta  \right)$, such that ${{g}^{*}}\left( A,\Phi  \right)$ coincides with $\left( {{A}_{p}^{\bmod }},{{\Phi }_{p}^{\bmod }} \right)$ on a sufficiently small neighborhood of the point $p$, for each $p\in D$.
\end{lemma}

We shall now use this complex gauge transformation as well as a smooth cut-off function to obtain an approximate solution to the $\text{SL(2}\text{,}\mathbb{R}\text{)}$-Hitchin equations. For the fixed local coordinates $z$ around each puncture $p$ and the positive function $r = \left| z \right|$ around the puncture, fix a constant $0<R<1$ and choose a smooth cut-off function ${{\chi }_{R}}:\left[ 0,\infty  \right)\to \left[ 0,1 \right]$ with $\text{supp}\chi \subseteq \left[ 0,R \right]$ and ${{\chi }_{R}}\left( r \right)=1$ for $r\le \frac{3R}{4}$. We impose the further requirement on the growth rate of this cut-off function:
\begin{equation}\label{growth_rate}
  \left| r{{\partial }_{r}}{{\chi }_{R}} \right|+\left| {{\left( r\partial r \right)}^{2}}{{\chi }_{R}} \right|\le k,
\end{equation}
for some constant $k$ not depending on $R$.

The map $x\mapsto {{\chi }_{R}}\left( r\left( x \right) \right):{{X}^{\times }}\to \mathbb{R}$ gives rise to a smooth cut-off function on the punctured surface ${{X}^{\times }}$ which by a slight abuse of notation we shall still denote by ${{\chi }_{R}}$. We may use this function ${{\chi }_{R}}$ to glue the two pairs $\left( A,\Phi  \right)$ and $\left( A_{p}^{\bmod },\Phi _{p}^{\bmod } \right)$ into an \emph{approximate solution}
	\[\left( A_{R}^{app},\Phi _{R}^{app} \right):=\exp {{\left( {{\chi }_{R}}\gamma  \right)}^{*}}\left( A,\Phi  \right).\]
The pair $\left( A_{R}^{app},\Phi _{R}^{app} \right)$ is a smooth pair and is by construction an exact solution of the Hitchin equations away from each punctured neighborhood ${{\mathsf{\mathcal{U}}}_{p}}$, while it coincides with the model pair $\left( A_{p}^{\bmod },\Phi _{p}^{\bmod } \right)$ near each puncture. More precisely, we have:
\[\left( A_{R}^{app},\Phi _{R}^{app} \right)=\left\{ \begin{matrix}
   \left( A,\Phi  \right),\text{ over }X\backslash \bigcup\limits_{p\in D}{\left\{ z\in {{\mathsf{\mathcal{U}}}_{p}}\left| \frac{3R}{4}\le \left| z \right|\le R \right. \right\}}  \\
   \left( A_{p}^{\bmod },\Phi _{p}^{\bmod } \right),\text{ over }\left\{ z\in {{\mathsf{\mathcal{U}}}_{p}}\left| 0<\left| z \right|\le \frac{3R}{4} \right. \right\}\text{, for each }p\in D.  \\
\end{matrix} \right.\]

\vspace{5mm}
\begin{center}
 \includegraphics[width=0.6\linewidth,height=0.6\textheight,keepaspectratio]{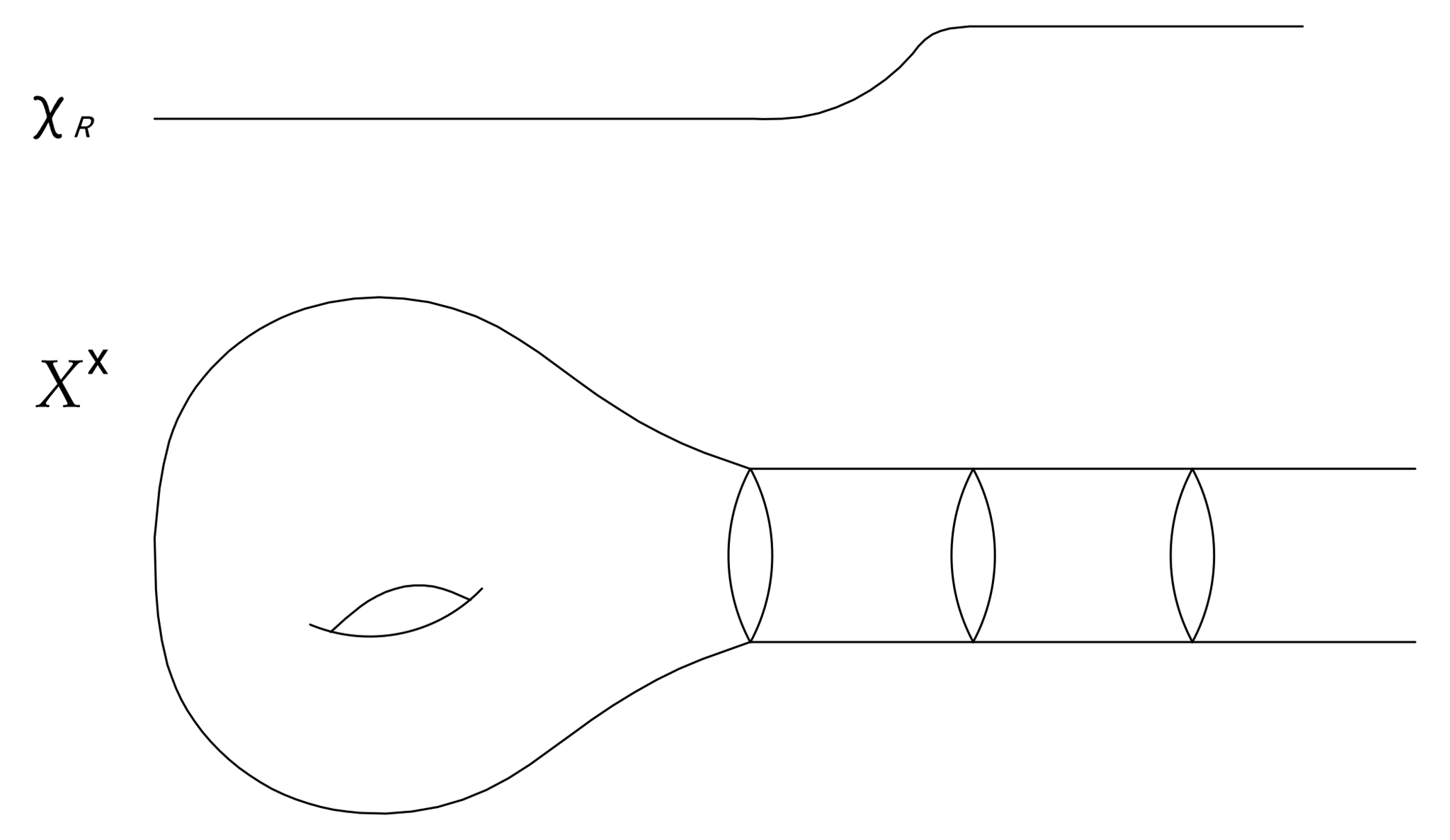}
    \captionof{figure}{Constructing an approximate solution over the punctured surface ${{X}^{\times }}$.}
\end{center}
\vspace{5mm}

Since $\left( A_{R}^{app},\Phi _{R}^{app} \right)$ is complex gauge equivalent to an exact solution $\left( A,\Phi  \right)$ of the Hitchin equations, it does still \emph{satisfy the second equation}, in other words, it holds that ${{\bar{\partial }}_{A_{R}^{app}}}\Phi _{R}^{app}=0$.\\
Indeed, for $\tilde{g}:=\exp \left( {{\chi }_{R}}\gamma  \right)$, we defined $\left( A_{R}^{app},\Phi _{R}^{app} \right)={{\tilde{g}}^{*}}\left( A,\Phi  \right)=\left( {{{\tilde{g}}}^{-1}}A\tilde{g}+{{{\tilde{g}}}^{-1}}d\tilde{g},{{{\tilde{g}}}^{-1}}\Phi \tilde{g} \right)$ and $\left( A,\Phi  \right)$ is an exact solution, thus in particular \[0={{\bar{\partial }}_{A}}\Phi =\bar{\partial }\Phi +\left[ {{A}^{0,1}}, \Phi  \right].\]

Moreover, Lemma \ref{complex_gauge_to_model} and assumption (\ref{growth_rate}) on the growth rate of the bump function ${{\chi }_{R}}$ provide us with a good estimate of the error up to which $\left( A_{R}^{app},\Phi _{R}^{app} \right)$ satisfies the first equation:

\begin{lemma}\label{curvature_estimate}Let ${\delta }'>0$ be as in Lemma \ref{complex_gauge_to_model} and fix some further constant $0<{\delta }''<{\delta }'$. The approximate solution $\left( A_{R}^{app},\Phi _{R}^{app} \right)$ to the parameter $0<R<1$ satisfies the inequality
	\[{{\left\| *F_{A_{R}^{app}}+*\left[ \Phi _{R}^{app}, {{\left( \Phi _{R}^{app} \right)}^{*}} \right] \right\|}_{{{C}^{0}}\left( {{X}^{\times }} \right)}}\le k{{R}^{{{\delta }''}}}\]
for some constant $k=k\left( {\delta }',{\delta }'' \right)$ which does not depend on $R$.
\end{lemma}

\begin{proof}
See \cite{Swoboda}, Lemma 3.5.
\end{proof}

In the subsections that follow we use the approximate solutions defined above, in order to obtain an approximate solution by gluing parabolic Higgs bundles over a complex connected sum of Riemann surfaces.

\subsection{Extending $\text{SL(2}\text{,}\mathbb{R}\text{)}$-pairs into $\text{Sp(4}\text{,}\mathbb{R}\text{)}$}\label{pairs_on_both_sides}
Let ${{X}_{1}}$ be a closed Riemann surface of genus ${{g}_{1}}$ and ${{D}_{1}}=\left\{ {{p}_{1}},\ldots ,{{p}_{s}} \right\}$ a collection of distinct points on ${{X}_{1}}$. Let $\left( {{E}_{1}},{{\Phi }_{1}} \right)\to {{X}_{1}}$ be a parabolic stable $\text{SL(2}\text{,}\mathbb{R}\text{)}$-Higgs bundle. Then there exists an adapted Hermitian metric ${{h}_{1}}$, such that $\left( {{A}_{{{h}_{1}}}},{{\Phi }_{1}} \right)$ is a solution to the equations with ${{A}_{{{h}_{1}}}}=\nabla \left( {{{\bar{\partial }}}_{1}},{{h}_{1}} \right)$ the associated Chern connection.

Let ${{g}_{1}}=\exp \left( {{\gamma }_{1}} \right)$ be the complex gauge transformation from \S \ref{approximate_SL2}, such that ${{g}^{*}_{1}}\left( {{A}_{{{h}_{1}}}},{{\Phi }_{1}} \right)$ is asymptotically close to a model solution $\left( A_{1,p}^{\bmod },\Phi _{1,p}^{\bmod } \right)$ near the puncture $p$, for each $p\in {{D}_{1}}$. Choose a trivialization $\tau $ over a neighborhood ${{\mathsf{\mathcal{U}}}_{p}}\subset X_{1}$ so that ${{\left( {{A}_{{{h}_{1}}}} \right)}^{\tau }}$ denotes the connection matrix and let ${{\chi }_{1}}$ be a smooth bump function on ${{\mathsf{\mathcal{U}}}_{p}}$ with the assumptions made in \S \ref{approximate_SL2}, so that we may define ${{\tilde{g}}_{1}}=\exp \left( {{\chi }_{1}}{{\gamma }_{1}} \right)$ and take the approximate solution over ${{X}_{1}}$:
 \[\left( A_{1}^{app},\Phi _{1}^{app} \right)={{\tilde{g}}_{1}}^{*}\left( {{A}_{{{h}_{1}}}},{{\Phi }_{1}} \right)=\left\{ \begin{matrix}
 \left( {{A}_{{{h}_{1}}}},{{\Phi }_{1}} \right),\text{  away from the points in the divisor } {{D}_{1}}  \\
 \left( A_{1,p}^{\bmod },\Phi _{1,p}^{\bmod } \right),\text{  near the point } p, \text{ for each } p \in {{D}_{1}}.\\
\end{matrix} \right. \]
The connection $A_{1}^{app}$ is given, in that same trivialization, by the connection matrix ${{\chi }_{1}}{{\left( {{A}_{{{h}_{1}}}} \right)}^{\tau }}$. The fact that ${{\tilde{g}}_{1}}$ is a complex gauge transformation may cause this $\text{SL(2}\text{,}\mathbb{R}\text{)}$-data to no longer be an exact solution of the equations over the bump region.

We wish to obtain an approximate $\text{Sp(4}\text{,}\mathbb{R}\text{)}$-pair by extending the $\text{SL(2}\text{,}\mathbb{R}\text{)}$-data via an embedding
\[\phi :\text{SL(2}\text{,}\mathbb{R}\text{)}\hookrightarrow \text{Sp(4}\text{,}\mathbb{R}\text{)}\]
and its extension $\phi :\text{SL(2}\text{,}\mathbb{C}\text{)}\hookrightarrow \text{Sp(4}\text{,}\mathbb{C}\text{)}$.
For the Cartan decompositions
\begin{align*}
\mathfrak{sl}\left( 2,\mathbb{R} \right) & = \mathfrak{so}\left( 2 \right)\oplus \mathfrak{m}\left( \text{SL(2}\text{,}\mathbb{R}\text{)} \right)\\
\mathfrak{sp}\left( 4,\mathbb{R} \right) & = \mathfrak{u}\left( 2 \right)\oplus \mathfrak{m}\left( \text{Sp(4}\text{,}\mathbb{R}\text{)} \right),
\end{align*}
their complexifications respectively read
\begin{align*}
\mathfrak{sl}\left( 2,\mathbb{C} \right) & = \mathfrak{so}\left( 2,\mathbb{C} \right)\oplus {{\mathfrak{m}}^{\mathbb{C}}}\left( \text{SL(2}\text{,}\mathbb{R}\text{)} \right)\\
\mathfrak{sp}\left( 4,\mathbb{C} \right) & = \mathfrak{gl}\left( 2,\mathbb{C} \right)\oplus {{\mathfrak{m}}^{\mathbb{C}}}\left( \text{Sp(4}\text{,}\mathbb{R}\text{)} \right).
\end{align*}
Assume now that copies of a maximal compact subgroup of $\text{SL(2}\text{,}\mathbb{R}\text{)}$ are mapped via $\phi$ into copies of a maximal compact subgroup of $\text{Sp(4}\text{,}\mathbb{R}\text{)}$. Then, since $\text{SO(2}{{\text{)}}^{\mathbb{C}}}=\text{SO(2}\text{,}\mathbb{C}\text{)}$ and $\text{U(2}{{\text{)}}^{\mathbb{C}}}=\text{GL}\left( 2,\mathbb{C} \right)$, the embedding $\phi $ describes an embedding $\text{SO(2}\text{,}\mathbb{C}\text{)}\hookrightarrow \text{GL}\left( 2,\mathbb{C} \right)$ and so we may use its infinitesimal deformation ${{\phi }_{*}}:\mathfrak{sl}\text{(2}\text{,}\mathbb{C}\text{)}\to \mathfrak{sp}\text{(4}\text{,}\mathbb{C}\text{)}$ to extend $\text{SL(2}\text{,}\mathbb{C}\text{)}$-data to $\text{Sp(4}\text{,}\mathbb{C}\text{)}$-data (see \cite{Ramanan}, \S 5.4, 5.5 for details).

By a slight abuse of notation, we shall still denote the $\text{Sp(4}\text{,}\mathbb{R}\text{)}$-pair obtained by extension through ${{\phi }}$ by $\left( {{A}_{1}},{{\Phi }_{1}} \right)$, with the curvature of the connection denoted by \[{{F}_{{{A}_{1}}}}\in {{\Omega }^{2}}\left( {{\mathbb{R}}^{2}};\text{ad}\left( Q \right) \right),\]
where $Q$ is the bundle obtained by extension of structure group and with the Higgs field ${{\Phi }_{1}}$ given by \[{{\Phi }_{1}}={{\phi }_{*}}\left| _{{{\mathfrak{m}}^{\mathbb{C}}}\left( \text{SL(2}\text{,}\mathbb{C}\text{)} \right)} \right.\left( \Phi _{1}^{app} \right).\]

Assume, moreover, that the norm of the infinitesimal deformation ${{\phi }}_{*}$ satisfies a Lipschitz condition, in other words, it holds that
\[{{\left\| {{\phi }_{*}}\left( M \right) \right\|}_{\mathfrak{sp}\left( 4,\mathbb{C} \right)}}\le m{{\left\| M \right\|}_{\mathfrak{sl}\left( 2,\mathbb{C} \right)}}\]
for $M\in \mathfrak{sl}\left( 2,\mathbb{C} \right)$ and a real constant $m$. In fact, the norms considered above are equivalent to the ${{C}_{0}}$-norm. Restricting these norms to $\mathfrak{so}\left( 2,\mathbb{C} \right)$ and ${{\mathfrak{m}}^{\mathbb{C}}}\left( \text{SL(2}\text{,}\mathbb{R}\text{)} \right)$ respectively, we may deduce that the error in curvature is still described by the inequality
\[{{\left\| *F_{A_{1}^{app}}+*\left[ \Phi _{1}^{app}, {{\left( \Phi _{1}^{app} \right)}^{*}} \right] \right\|}_{{{C}^{0}}}}\le {{k}_{1}}{{R}^{{{\delta }''}}}\] 
for a (different) real constant ${{k}_{1}}$, which still does not depend on the parameter $R>0$.

In summary, using an embedding
$\phi :\text{SL(2}\text{,}\mathbb{R}\text{)}\hookrightarrow \text{Sp(4}\text{,}\mathbb{R}\text{)}$ with the properties described above, we may extend the approximate solution $\left( A_{1}^{app},\Phi _{1}^{app} \right)$ to take an approximate $\text{Sp(4}\text{,}\mathbb{R}\text{)}$-pair $\left( {{A}_{1}},{{\Phi }_{1}} \right)$ over ${{X}_{1}}$, which agrees with a model solution over an annulus $\Omega _{1}^{p}$ around each puncture $p\in {{D}_{1}}$. This model solution is the extension via $\phi$ of the model $\left( A_{1, p}^{\bmod },\Phi _{1, p}^{\bmod } \right)$ in $\text{SL(2}\text{,}\mathbb{R}\text{)}$; by a slight abuse of notation it shall still be denoted by $\left( A_{1, p}^{\bmod },\Phi _{1, p}^{\bmod } \right)$. The pair $\left( {{A}_{1}},{{\Phi }_{1}} \right)$ lives in the holomorphic principal $\text{GL(2}\text{,}\mathbb{C}\text{)}$-bundle $Q$ obtained by extension of structure group via $\phi$, which we shall keep denoting as $\left( {{E}_{1}},{{h}_{1}} \right)$ to ease notation.

Repeating the above considerations for another closed Riemann surface ${{X}_{2}}$ of genus ${{g}_{2}}$ and ${{D}_{2}}=\left\{ {{q}_{1}},\ldots ,{{q}_{s}} \right\}$ a collection of $s$-many distinct points of ${{X}_{2}}$, we obtain  an approximate $\text{Sp(4}\text{,}\mathbb{R}\text{)}$-pair $\left( {{A}_{2}},{{\Phi }_{2}} \right)$ over ${{X}_{2}}$, which agrees with a model solution $\left( A_{2, q}^{\bmod },\Phi _{2, q}^{\bmod } \right)$ over an annulus $\Omega _{2}^{q}$ around each puncture $q\in {{D}_{2}}$. This pair lives on the holomorphic principal $\text{GL(2}\text{,}\mathbb{C}\text{)}$-bundle obtained by extension of structure group via another appropriate embedding $\text{SL(2}\text{,}\mathbb{R}\text{)}\hookrightarrow \text{Sp(4}\text{,}\mathbb{R}\text{)}$; let this hermitian bundle be denoted by $\left( {{E}_{2}},{{h}_{2}} \right)$.

\subsection{Complex connected sum of Riemann surfaces}\label{ccs}
In order to describe how two parabolic Higgs bundles can be glued to a (non-parabolic) Higgs bundle, the first step is to glue their underlying surfaces with boundary; we summarize this construction below and more details can be found in \cite{Kydonakis} for instance.

Take annuli ${{\mathbb{A}}_{1}}=\left\{ z\in \mathbb{C}\left| {{r}_{1}}<\left| z \right|<{{R}_{1}} \right. \right\}$  and ${{\mathbb{A}}_{2}}=\left\{ z\in \mathbb{C}\left| {{r}_{2}}<\left| z \right|<{{R}_{2}} \right. \right\}$  on the complex plane, and consider the M{\"o}bius transformation ${{f}_{\lambda }}:{{\mathbb{A}}_{1}}\to {{\mathbb{A}}_{2}}$ with ${{f}_{\lambda }}\left( z \right)=\frac{\lambda }{z}$, where $\lambda \in \mathbb{C}$ with $\left| \lambda  \right|={{r}_{2}}{{R}_{1}}={{r}_{1}}{{R}_{2}}$, which defines a conformal biholomorphism between the annuli. 

Let now two compact Riemann surfaces $X_{1}, X_{2}$ of respective genera $g_{1}, g_{2}$. Choose points $p\in {{X}_{1}}$, $q\in {{X}_{2}}$ and local charts ${{\psi }_{i}}:{{U}_{i}}\to \Delta \left( 0,{{\varepsilon }_{i}} \right)$ around these points, for $i=1,2$. The biholomorphism ${{f}_{\lambda }}:{{\mathbb{A}}_{1}}\to {{\mathbb{A}}_{2}}$ can be used to glue the two Riemann surfaces $X_{1}, X_{2}$ along the inverse image of the annuli ${{\mathbb{A}}_{1}},{{\mathbb{A}}_{2}}$ on the surfaces, via the biholomorphism
\[{{g}_{\lambda }}:{{\Omega }_{1}}=\psi _{1}^{-1}\left( {{\mathbb{A}}_{1}} \right)\to {{\Omega }_{2}}=\psi _{2}^{-1}\left( {{\mathbb{A}}_{2}} \right)\] with ${{g}_{\lambda }}=\psi _{2}^{-1}\circ {{f}_{\lambda }}\circ {{\psi }_{1}}$.
For collections of $s$-many distinct points $D_{1}$ on $X_{1}$ and $D_{2}$ on $X_{2}$, this procedure is assumed to be taking place for annuli around each pair of points $\left( p,q \right)$ for $p\in {{D}_{1}}$ and $q\in {{D}_{2}}$.

If ${{X}_{1}},{{X}_{2}}$ are orientable and orientations are chosen for both, since ${{f}_{\lambda }}$ is orientation preserving we obtain a natural orientation on the connected sum ${{X}_{1}}\#{{X}_{2}}$ which coincides with the given ones on $X_{1}$ and $X_{2}$. Therefore, ${{X}_{\#}}={{X}_{1}}\#{{X}_{2}}$ is a Riemann surface of genus ${{g}_{1}}+{{g}_{2}}+s-1$, the \emph{complex connected sum}, where $g_{i}$ is the genus of the $X_{i}$ and $s$ is the number of points in $D_{1}$ and $D_{2}$. Its complex structure however is heavily dependent on the parameters ${{p}_{i}},{{q}_{i}},\lambda $.

\subsection{Gluing cylindrical hermitian vector bundles}\label{gluing_cylindrical_hermitian}

A punctured neighborhood on a Riemann surface can be also thought of, using a cylindrical coordinate transformation, as a half cylinder attached to the surface, and also an annulus in the real parameter $R$ can be thought of as a finite tube of length $\sim T=\left| \log R \right|$. Thus, the gluing of two punctured Riemann surfaces can be thought of as the gluing of two Riemann surfaces with cylindrical ends to get a smooth surface with a finite number of long Euclidean cylinders of length $2\left| \log R \right|$, one for each puncture.

For the Riemann surfaces $X_{i}$  with neck regions $\Omega_{i}$ as defined in \S \ref{ccs} for $i=1,2$, consider the Riemannian metrics 
\[{{g}^{1}}=\frac{{{\left| dz \right|}^{2}}}{{{\left| z \right|}^{2}}}\text{   and   }{{g}^{2}}=\frac{{{\left| dw \right|}^{2}}}{{{\left| w \right|}^{2}}}\]
on ${{\Omega }_{1}}$ and ${{\Omega }_{2}}$ respectively, and endow these with cylindrical coordinates $\left( {{\tau }^{i}},{{\theta }^{i}} \right)$, for $i=1,2$. This way, the punctured Riemann surfaces $X_{1}^{\times}$ and $X_{2}^{\times}$  are viewed as Riemannian manifolds with cylindrical ends. The metrics $g^{1}$ and $g^{2}$ induce a smooth metric on the flat cylinder $\Omega ={{\Omega }_{1}}\sim {{\Omega }_{2}}$ obtained by gluing $\Omega_{1}$ and $\Omega_{2}$ via the orientation reversing isometry $g_{\lambda}$ from \S \ref{ccs}. This metric can be extended smoothly over ${{X}_{\#}}$ to a metric compatible with the complex structure. 

Let now ${{\hat{E}}_{i}}\to X_{i}^{\times }$ be a ${{\mathbb{Z}}_{2}}$-graded cylindrical hermitian vector bundle, for $i=1,2$. This means that there exists a vector bundle ${{E}_{i}}\to {{X}_{i}}$ and a bundle isomorphism ${{\pi }^{*}}{{E}_{i}}\simeq {{\hat{E}}_{i}}\left| _{{{\mathbb{R}}^{+}}\times {{X}_{i}}} \right.$, and that the hermitian metric ${{\hat{H}}_{i}}$ on ${{\hat{E}}_{i}}$ is along the cylindrical end of the form ${{\hat{H}}_{i}}={{\pi }^{*}}{{H}_{i}}$ for some hermitian metric $H_{i}$ on $E_{i}$ and ${{\hat{E}}_{i}}$ splits into an orthogonal sum ${{\hat{E}}_{i}}=\hat{E}_{i}^{+}\oplus \hat{E}_{i}^{-}$ of cylindrical vector bundles, for $i=1,2$. Assuming that there exists an isometry ${{E}_{1}}\to {{E}_{2}}$ of hermitian vector bundles covering ${{g}_{\lambda }}$ and respecting the gradings, the ${{\mathbb{Z}}_{2}}$-graded cylindrical hermitian vector bundles ${{\hat{E}}_{i}}$ can be glued together to get a ${{\mathbb{Z}}_{2}}$-graded hermitian vector bundle ${{E}_{\#}}=E_{\#}^{+}\oplus E_{\#}^{-}$ over ${{X}_{\#}}$.

\subsection{Gluing the data $\left( A, \Phi \right)$}\label{gluing_the_data}

In \S \ref{pairs_on_both_sides} we described the construction of two pairs $\left( A_{1}, \Phi_{1} \right)$ and $\left( A_{2}, \Phi_{2} \right)$ on hermitian vector bundles with cylindrical ends. The pair $\left( A_{1}, \Phi_{1} \right)$ agrees with the model solution $\left( A_{1, p}^{\bmod },\Phi _{1, p}^{\bmod } \right)$ over an annulus $\Omega_{1}$ for each puncture $p$. Similarly, for the pair $\left( A_{2}, \Phi_{2} \right)$. Remember that $A_{1}^{\text{mod}}=A_{2}^{\text{mod}}=0$ and 
\[\Phi _{1}^{\text{mod}}={{\phi }_{1,*}}\left| _{{{\mathfrak{m}}^{\mathbb{C}}}\left( \text{SL}\left( 2,\mathbb{R} \right) \right)} \right.\left( \begin{matrix}
   {{C}_{1}} & 0  \\
   0 & -{{C}_{1}}  \\
\end{matrix} \right)\frac{dz}{z},\]
\[\Phi _{2}^{\text{mod}}={{\phi }_{2,*}}\left| _{{{\mathfrak{m}}^{\mathbb{C}}}\left( \text{SL}\left( 2,\mathbb{R} \right) \right)} \right.\left( \begin{matrix}
   {{C}_{2}} & 0  \\
   0 & -{{C}_{2}}  \\
\end{matrix} \right)\frac{dw}{w},\]
for embeddings ${{\phi }_{1}},{{\phi }_{2}}:\mathfrak{sl}\left( 2,\mathbb{C} \right)\to \mathfrak{sp}\left( 4,\mathbb{C} \right)$ and real constants ${{C}_{1}},{{C}_{2}}$, as in \S \ref{pairs_on_both_sides}. 

The gluing of the Riemann surfaces is realized along the curve $zw=\lambda $, thus we have $\frac{dz}{z}=-\frac{dw}{w}$ over the annuli ${{\Omega }_{1}}$ and ${{\Omega }_{2}}$ for each point $p$, respectively $q$. Therefore, assuming that the constants ${{C}_{1}}$ and ${{C}_{2}}$ are such so that the Higgs fields  $\Phi _{1}^{\text{mod}}$ and $\Phi _{2}^{\text{mod}}$ match-up, then we are permitted to construct from this pair of singular model solutions on the cylinder, a smooth model solution, which we shall denote by $\left( A_{p,q}^{\bmod },\Phi _{p,q}^{\bmod } \right)$ for each pair of points $p,q$ around which the annuli are glued together; in \S \ref{pairs_model_represenations} later on we will see explicit examples of embeddings for which the Higgs fields match-up.  

We thus glue the pairs $\left( {{A}_{1}},{{\Phi }_{1}} \right),\left( {{A}_{2}},{{\Phi }_{2}} \right)$ together to get an \emph{approximate solution} of the $\text{Sp(4}\text{,}\mathbb{R}\text{)}$-Hitchin equations:

\[\left( A_{R}^{app},\Phi _{R}^{app} \right):=\left\{ \begin{matrix}
   \left( {{A}_{1}},{{\Phi }_{1}} \right), & {} & \text{over }{{X}_{1}}\backslash {{X}_{2}}  \\
   \left( A_{p,q}^{\bmod },\Phi _{p,q}^{\bmod } \right), & {} & \text{         over } \Omega \text{ around each pair of points } \left( p,q \right)  \\
   \left( {{A}_{2}},{{\Phi }_{2}} \right) & {} & \text{over }{{X}_{2}}\backslash {{X}_{1}}  \\
\end{matrix} \right.,\]
considered on the bundle $\left( E_{\#},{{h}_{\#}} \right)$ over the complex connected sum ${{X}_{\#}}:=X_{{1}}\#X_{{2}}$.
\vspace{5mm}
\begin{center}
  \includegraphics[width=0.6\linewidth,height=0.6\textheight,keepaspectratio]{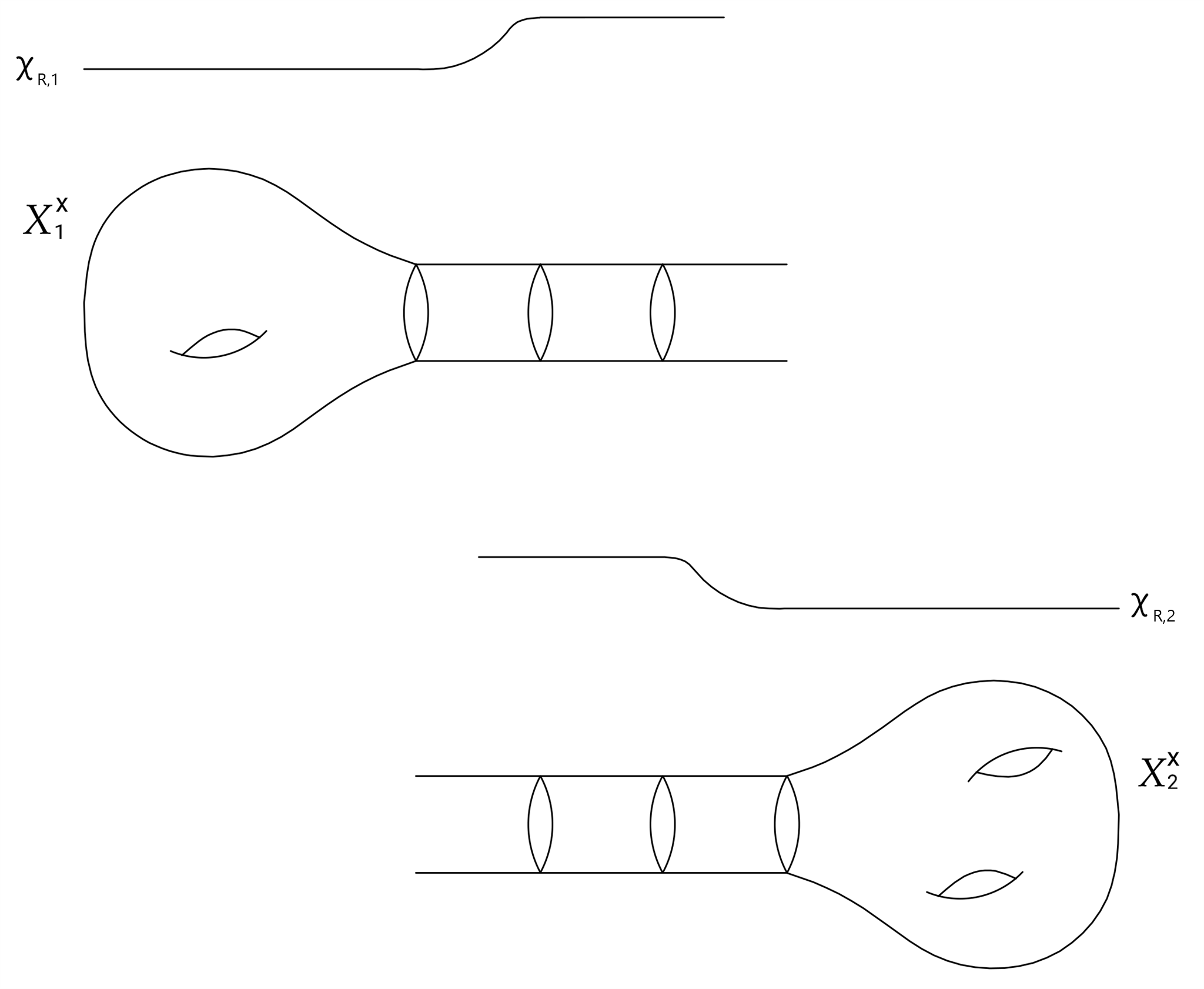}
  \captionof{figure}{Constructing approximate solutions over $X_{1}^{\times }$ and $X_{2}^{\times }$.}
\end{center}
\vspace{5mm}
\begin{center}
  \includegraphics[width=0.6\linewidth,height=0.6\textheight,keepaspectratio]{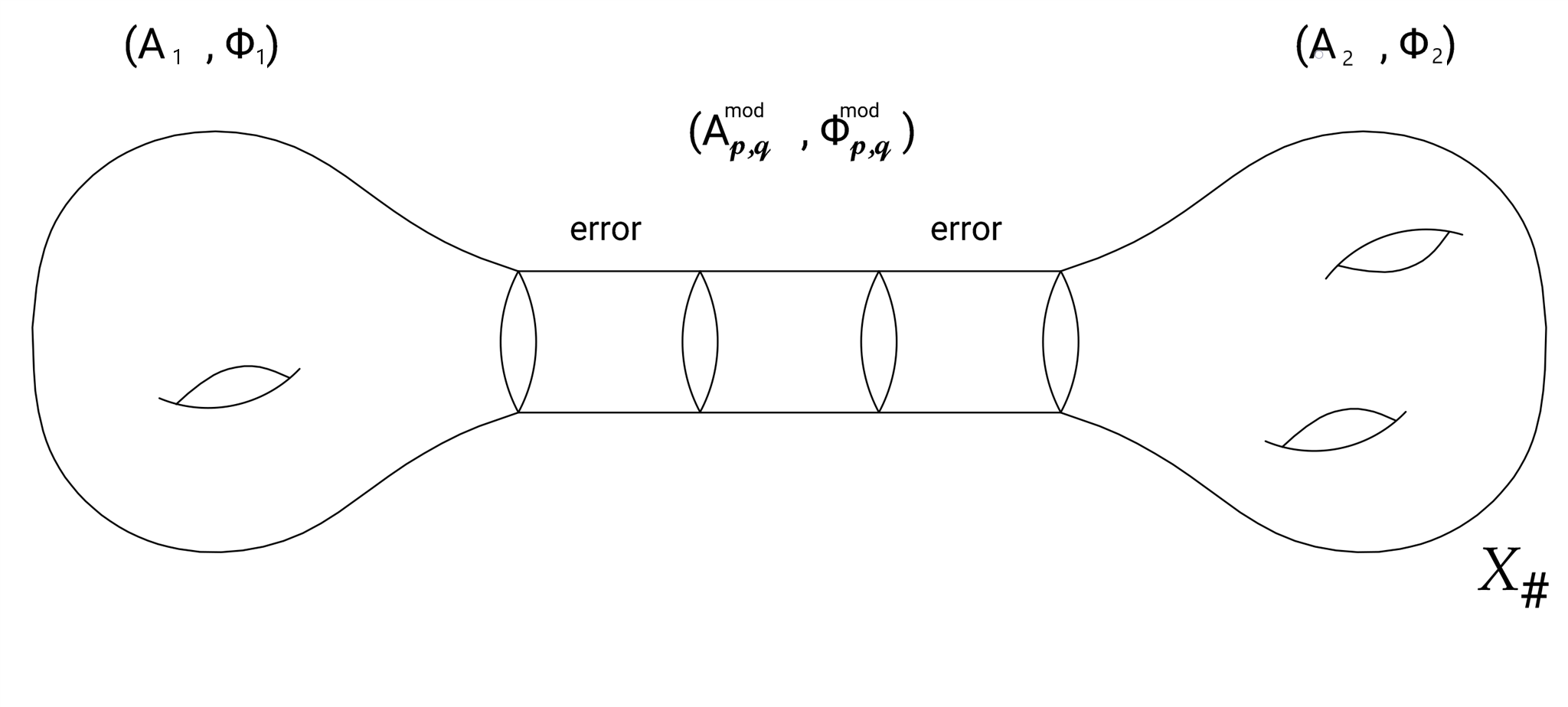}
   \captionof{figure}{$\left( A_{R}^{app},\Phi _{R}^{app} \right)$ over the complex connected sum ${{X}_{\#}}$.}
\end{center}
\vspace{5mm}
By construction, $\left( A_{R}^{app},\Phi _{R}^{app} \right)$ is a smooth pair on ${{X}_{\#}}$, complex gauge equivalent to an exact solution of the Hitchin equations by a smooth gauge transformation defined over all of ${{X}_{\#}}$. It satisfies the second equation, while the first equation is satisfied up to an error which we have good control of:

\begin{lemma}\label{approximate_estimate} The approximate solution $\left( A_{R}^{app},\Phi _{R}^{app} \right)$ to the parameter $0<R<1$ satisfies
	\[{{\left\| *F_{A_{R}^{app}}+*\left[ \Phi _{R}^{app},-\tau \left( \Phi _{R}^{app} \right) \right] \right\|}_{{{C}^{0}}\left( {{X}^{\times }} \right)}}\le k{{R}^{{{\delta }''}}},\]
for some constants ${\delta }''>0$ and $k=k\left( {{\delta }''} \right)$, which do not depend on $R$.
\end{lemma}

\begin{proof}
Follows from Lemma \ref{curvature_estimate}; take $k:=\max \left\{ {{k}_{1}},{{k}_{2}} \right\}$, for ${{k}_{1}},{{k}_{2}}$ the constants appearing in the bound of the error for the approximate solutions constructed over each of the Riemann surfaces $X_{1}$ and $X_{2}$.
\end{proof}

\subsection{The representations ${{\phi }_{irr}}$ and $\psi$}\label{pairs_model_represenations}
In this subsection, we provide examples when the model Higgs fields can match-up using particular representations from $\text{SL(2}\text{,}\mathbb{R}\text{)}$ into $\text{Sp(4}\text{,}\mathbb{R}\text{)}$.
\vspace{2mm}
\paragraph{\emph{The irreducible representation ${{\phi }_{irr}}:\mathrm{SL(2}\text{,}\mathbb{R}\text{)}\hookrightarrow \mathrm{Sp(4}\text{,}\mathbb{R}\text{)}$.}} Let $\left( A_{1}^{app},\Phi _{1}^{app} \right)$ over $X_{1}$ be the approximate $\text{SL(2}\text{,}\mathbb{C}\text{)}$-pair in parameter $R>0$, as was considered in \S \ref{approximate_SL2}, which agrees with the model pair \[A_{1}^{\bmod }=0,\,\,\Phi _{1}^{\bmod }=\left( \begin{matrix}
   C & 0  \\
   0 & -C  \\
\end{matrix} \right)\frac{dz}{z}\] for nonzero $C\in \mathbb{R}$, over an annulus in $z$-coordinates around a point $p\in {{D}_{1}}$.\\
The embedding ${{\phi }_{irr}}$ defined in (\ref{irreducible_rep}) extends to give an embedding ${{\phi }_{irr}}:\text{SL(2}\text{,}\mathbb{C}\text{)}\hookrightarrow \text{Sp(4}\text{,}\mathbb{C}\text{)}$. For the Lie algebra of $\text{SL(2}\text{,}\mathbb{C}\text{)}$, $\mathfrak{sl}\left( 2,\mathbb{C} \right)=\left\{ \left( \begin{matrix}
   a & b  \\
   c & -a  \\
\end{matrix} \right)\left| \,\,a,b,c\in \mathbb{C} \right. \right\}$, we may use a Cartan basis for the Lie algebra to determine the infinitesimal deformation, ${{\phi }_{irr}}_{*}:\mathfrak{sl}\left( 2,\mathbb{C} \right)\to \mathfrak{sp}\left( 4,\mathbb{C} \right)$ with
	\[{{\phi }_{irr}}_{*}\left( \left( \begin{matrix}
   a & b  \\
   c & -a  \\
\end{matrix} \right) \right)=\left( \begin{matrix}
   3a & -\sqrt{3}b & 0 & 0  \\
   -\sqrt{3}c & a & 0 & 2b  \\
   0 & 0 & -3a & \sqrt{3}c  \\
   0 & 2c & \sqrt{3}b & -a  \\
\end{matrix} \right).\]
We now notice that ${{\phi }_{irr}}\left( \text{SO(2)} \right)$ lies in a copy of $\text{U(2)}\hookrightarrow \text{Sp(4}\text{,}\mathbb{R}\text{)}$, that is,
\[\text{U(2)}\cong \left\{ \left( \begin{matrix}
   A & B  \\
   -B & A  \\
\end{matrix} \right)\left| \,{{A}^{T}}A+{{B}^{T}}B={{I}_{2}},\,\,{{A}^{T}}B-{{B}^{T}}A=0 \right. \right\}.\]
In other words, copies of a maximal compact subgroup of $\text{SL(2}\text{,}\mathbb{R}\text{)}$ are mapped into copies of a maximal compact subgroup of $\text{Sp(4}\text{,}\mathbb{R}\text{)}$. Furthermore, one can check that for $A\in \mathfrak{sl}\left( 2,\mathbb{C} \right)$: \[{{\left\| {{\phi }_{irr}}_{*}\left( A \right) \right\|}_{\mathfrak{sp}\left( 4,\mathbb{C} \right)}}=10{{\left\| A \right\|}_{\mathfrak{sl}\left( 2,\mathbb{C} \right)}}.\]
As was described in \S \ref{pairs_on_both_sides}, ${{\phi }_{irr}}$ can be now used to extend $\text{SL(2}\text{,}\mathbb{R}\text{)}$-data to $\text{Sp(4}\text{,}\mathbb{R}\text{)}$-data $\left( {{A}_{1}},{{\Phi }_{1}} \right)$, where in this case, we have ${{A}_{1}}=0$ and
 \[{{\Phi }_{1}}={{\phi }_{irr}}_{*}\left| _{{{\mathfrak{m}}^{\mathbb{C}}}\left( \text{SL(2}\text{,}\mathbb{C}\text{)} \right)} \right.\left( \Phi _{1}^{app} \right)=\left( \begin{matrix}
   3C & 0 & 0 & 0  \\
   0 & C & 0 & 0  \\
   0 & 0 & -3C & 0  \\
   0 & 0 & 0 & -C  \\
\end{matrix} \right)\frac{dz}{z}\]
over the annulus on $X_{1}$ in $z$-coordinates around the point $p$.
\vspace{2mm}
\paragraph{\emph{The representation $\psi :\mathrm{SL(2}\text{,}\mathbb{R}\text{)}\times \mathrm{SL(2}\text{,}\mathbb{R}\text{)}\hookrightarrow \mathrm{Sp(4}\text{,}\mathbb{R}\text{)}$.}}
Let $\left( A_{2,1}^{app},\Phi _{2,1}^{app} \right),\,\left( A_{2,2}^{app},\Phi _{2,2}^{app} \right)$ over $X_{2}$ be two approximate $\text{SL(2}\text{,}\mathbb{C}\text{)}$-pairs in parameter $R>0$, which agree respectively with the model pairs
	\[A_{2,1}^{\bmod }=0,\,\,\Phi _{2,1}^{\bmod }=\left( \begin{matrix}
   -3C & 0  \\
   0 & 3C  \\
\end{matrix} \right)\frac{dz}{z} \,\, \text{   and   } \,\, A_{2,2}^{\bmod }=0,\,\,\Phi _{2,2}^{\bmod }=\left( \begin{matrix}
   -C & 0  \\
   0 & C  \\
\end{matrix} \right)\frac{dz}{z},\]
for the same nonzero real parameter $C\in \mathbb{R}$ considered in defining the pair $\left( A_{1}^{app},\Phi _{1}^{app} \right)$ over $X_{1}$ above, over an annulus in $w$-coordinates around a point $q\in {{D}_{2}}$.\\
We extend $\text{SL(2}\text{,}\mathbb{C}\text{)}\times \text{SL(2}\text{,}\mathbb{C}\text{)}$-data into $\text{Sp(4}\text{,}\mathbb{C}\text{)}$ using the homomorphism $\psi$ defined in (\ref{diagonal_Fuchsian}). Take the extension of the embedding $\psi$ into $\text{SL(2}\text{,}\mathbb{C}\text{)}\times \text{SL(2}\text{,}\mathbb{C}\text{)}$, and now the
infinitesimal deformation of this homomorphism is given by ${{\psi }_{*}}:\mathfrak{sl}\left( 2,\mathbb{C} \right)\times \mathfrak{sl}\left( 2,\mathbb{C} \right)\hookrightarrow \mathfrak{sp}\left( 4,\mathbb{C} \right)$ with
 \[{{\psi }_{*}} \left( \left( \begin{matrix}
   a & b  \\
   c & -a  \\
\end{matrix} \right),\left( \begin{matrix}
   e  & f   \\
   g  & -e   \\
\end{matrix} \right) \right)=\left( \begin{matrix}
   a & 0 & b & 0  \\
   0 & e  & 0 & f   \\
   c & 0 & -a & 0  \\
   0 & g  & 0 & -e   \\
\end{matrix} \right).\]
We may again check that $\psi \left( \text{SO(2)}\times \text{SO(2)} \right)$ is a copy of $\text{U(2)}$. On the other hand, a norm on the space
$\mathfrak{sl}\left( 2,\mathbb{C} \right)\times \mathfrak{sl}\left( 2,\mathbb{C} \right)$ is given by \[\psi \left( A,B \right)=\left\| A \right\|+\left\| B \right\|\] and we check that
	\[{{\left\| {{\psi }_{*}}\left( A,B \right) \right\|}_{\mathfrak{sp}\left( 4,\mathbb{C} \right)}}={{\left\| \left( A,B \right) \right\|}_{\mathfrak{sl}\left( 2,\mathbb{C} \right)\times \mathfrak{sl}\left( 2,\mathbb{C} \right)}}={{\left\| A \right\|}_{\mathfrak{sl}\left( 2,\mathbb{C} \right)}}+{{\left\| B \right\|}_{\mathfrak{sl}\left( 2,\mathbb{C} \right)}}\]
and so the map ${{\psi }_{*}}$ at the level of Lie algebras is an isometry. Therefore, $\psi $ extends to give an embedding $\psi :\text{SO(2}\text{,}\mathbb{C}\text{)}\times \text{SO(2}\text{,}\mathbb{C}\text{)}\hookrightarrow \text{GL}\left( 2,\mathbb{C} \right)$, and so we may use the infinitesimal deformation ${{\psi }_{*}}$ to extend the $\text{SL(2}\text{,}\mathbb{C}\text{)}\times \text{SL(2}\text{,}\mathbb{C}\text{)}$-data $\left( \left( A_{2,1}^{app},\Phi _{2,1}^{app} \right),\left( A_{2,2}^{app},\Phi _{2,2}^{app} \right) \right)$ to an $\text{Sp(4}\text{,}\mathbb{R}\text{)}$-pair $\left( {{A}_{2}},{{\Phi }_{2}} \right)$, with ${{A}_{2}}=0$ and Higgs field ${{\Phi }_{2}}$ given by
	\[{{\Phi }_{2}}={{\psi }_{*}}\left| _{{{\mathfrak{m}}^{\mathbb{C}}}\left( \text{SL(2}\text{,}\mathbb{R}\text{)} \right)\,\times \,{{\mathfrak{m}}^{\mathbb{C}}}\left( \text{SL(2}\text{,}\mathbb{R}\text{)} \right)} \right.\left( \Phi _{2,1}^{app},\Phi _{2,2}^{app} \right)=\left( \begin{matrix}
   -3C & 0 & 0 & 0  \\
   0 & -C & 0 & 0  \\
   0 & 0 & 3C & 0  \\
   0 & 0 & 0 & C  \\
\end{matrix} \right)\frac{dw}{w}\]
over the annulus on $X_{2}$ in $w$-coordinates around the point $q$.

\section{Correcting an approximate solution to an exact solution}

\subsection{The contraction mapping argument}\label{contraction_mapping}
A standard strategy, due largely to C. Taubes \cite{Taubes}, for correcting an approximate solution to an exact solution of gauge-theoretic equations involves studying the linearization of a relevant elliptic operator. In the Higgs bundle setting, the linearization of the Hitchin operator was first described in \cite{MSWW} and furthermore in \cite{Swoboda} for solutions to the $\text{SL(2}\text{,}\mathbb{C}\text{)}$-self-duality equations over a nodal surface. We are going to use this analytic machinery to correct our approximate solution to an exact solution over the complex connected sum of Riemann surfaces. We next summarize the strategy to be followed in the forthcoming two sections.

Let $G$ be a connected, semisimple Lie group. For the complex connected sum ${{X}_{\#}}$ consider the nonlinear $G$-Hitchin operator at a pair $\left( A,\Phi  \right)\in {{\Omega }^{1}}\left( {{X}_{\#}},{{E}_{H}}\left( {{\mathfrak{h}}^{\mathbb{C}}} \right) \right)\oplus {{\Omega }^{1,0}}\left( {{X}_{\#}},{{E}_{H}}\left( {{\mathfrak{m}}^{\mathbb{C}}} \right) \right)$

\[\mathsf{\mathcal{H}}\left( A,\Phi  \right)=\left( F\left( A \right)-\left[ \Phi ,\tau \left( \Phi  \right) \right],{{{\bar{\partial }}}_{A}}\Phi  \right).\]
Moreover, consider the orbit map
\[\gamma \mapsto {{\mathsf{\mathcal{O}}}_{\left( A,\Phi  \right)}}\left( \gamma  \right)={{g}^{*}}\left( A,\Phi  \right)=\left( {{g}^{*}}A,{{g}^{-1}}\Phi g \right)\]
for $g=\exp \left( \gamma  \right)$ and $\gamma \in {{\Omega }^{0}}\left( {{X}_{\#}},{{E}_{H}}\left( {{\mathfrak{h}}^{\mathbb{C}}} \right)  \right)$, where $H\subset G$ is a maximal compact subgroup.

Therefore, correcting the approximate solution $\left( A_{R}^{app},\Phi _{R}^{app} \right)$ to an exact solution of the $G$-Hitchin equations accounts to finding a point $\gamma$ in the complex gauge orbit of $\left( A_{R}^{app},\Phi _{R}^{app} \right)$, for which $\mathsf{\mathcal{H}}\left( {{g}^{*}}\left( A_{R}^{app},\Phi _{R}^{app} \right) \right)=0$. However, since we have seen that the second equation is satisfied by the pair $\left( A_{R}^{app},\Phi _{R}^{app} \right)$ and since the condition ${{\bar{\partial }}_{A}}\Phi =0$ is preserved under the action of the complex gauge group ${{\mathsf{\mathcal{G}}}_{H}^{\mathbb{C}}}$, we actually seek a solution $\gamma $ to the following equation 
\[{{\mathsf{\mathcal{F}}}_{R}}\left( \gamma  \right):= \mathsf{\mathcal{H}}\circ {{\mathsf{\mathcal{O}}}_{\left( A_{R}^{app},\Phi _{R}^{app} \right)}}\left( \exp (\gamma ) \right)=0.\]
For a Taylor series expansion of this operator
	\[{{\mathsf{\mathcal{F}}}_{R}}\left( \gamma  \right)=\mathsf{\mathcal{H}}\left( A_{R}^{app},\Phi _{R}^{app} \right)+{{L}_{\left( A_{R}^{app},\Phi _{R}^{app} \right)}}\left( \gamma  \right)+{{Q}_{R}}\left( \gamma  \right),\]
where ${{Q}_{R}}$ includes the quadratic and higher order terms in $\gamma $, we can then see that ${{\mathsf{\mathcal{F}}}_{R}}\left( \gamma  \right)=0$ if and only if $\gamma $ is a fixed point of the map:
\begin{align*}
   T: H_{B}^{2}\left( {{X}_{\#}} \right) & \to H_{B}^{2}\left( {{X}_{\#}} \right) \\
   \gamma & \mapsto -{{G}_{R}}\left( \mathsf{\mathcal{H}}\left( A_{R}^{app},\Phi _{R}^{app} \right)+{{Q}_{R}}(\gamma ) \right)
\end{align*}
where we denoted ${{G}_{R}}:= L_{\left( A_{R}^{app},\Phi _{R}^{app} \right)}^{-1}$ and $H_{B}^{2}\left( {{X}_{\#}} \right)$ is the Hilbert space defined by
\[H_{B}^{2}\left( {{X}_{\#}} \right):=\left\{ \gamma \in {{L}^{2}}\left( {{X}_{\#}} \right)\left| {{\nabla }_{B}}\gamma ,\nabla _{B}^{2}\gamma  \right.\in {{L}^{2}}\left( {{X}_{\#}} \right) \right\},\]
for a fixed background connection $\nabla_{B}$ defined as a smooth extension to ${X}_{\#}$ of the model connection $A^{\text{mod}}_{p,q}$ over the cylinder for each pair of points $(p,q)$. 
The problem then reduces to showing that the mapping $T$ is a contraction of an open ball ${{B}_{{{\rho }_{R}}}}$ of radius ${{\rho }_{R}}$ in $H_{B}^{2}\left( {{X}_{\#}} \right)$, since then from Banach's fixed point theorem there will exist a unique $\gamma $ such that $T\left( \gamma  \right)=\gamma $, in other words, such that ${{\mathsf{\mathcal{F}}}_{R}}\left( \gamma  \right)=0$. In particular, one needs to show that:
\begin{enumerate}
  \item $T$ is a contraction defined on ${{B}_{{{\rho }_{R}}}}$ for some ${{\rho }_{R}}$, and
  \item $T$ maps ${{B}_{{{\rho }_{R}}}}$ to ${{B}_{{{\rho }_{R}}}}$.
\end{enumerate}

In order to complete the above described contraction mapping argument, we need to show the following:

\begin{description}
  \item[i] The linearized operator at the approximate solution ${{L}_{\left( A_{R}^{app},\Phi _{R}^{app} \right)}}$ is invertible.
  \item[ii] There is an upper bound for the inverse operator ${{G}_{R}}=L_{\left( A_{R}^{app},\Phi _{R}^{app} \right)}^{-1}$ as an operator ${{L}^{2}}\left( {{r}}drd\theta \right)\to {{L}^{2}}\left( {{r}}drd\theta \right)$.
  \item[iii] There is an upper bound for the inverse operator ${{G}_{R}}=L_{\left( A_{R}^{app},\Phi _{R}^{app} \right)}^{-1}$ also when viewed as an operator ${{L}^{2}}\left( {{r}}drd\theta \right)\to H_{B}^{2}\left( {{X}_{\#}},{{r}}drd\theta \right)$.
  \item[iv] We can control a Lipschitz constant for ${{Q}_{R}}$, that is, there exists a constant $C>0$ such that
	\[{{\left\| {{Q}_{R}}\left( {{\gamma }_{1}} \right)-{{Q}_{R}}\left( {{\gamma }_{0}} \right) \right\|}_{{{L}^{2}}}}\le C\rho {{\left\| {{\gamma }_{1}}-{{\gamma }_{0}} \right\|}_{H_{B}^{2}}}\]
for all $0<\rho \le 1$ and ${{\gamma }_{0}},{{\gamma }_{1}}\in {{B}_{\rho }}$, the closed ball of radius $\rho$ around $0$ in $H_{B}^{2}\left( {{X}_{\#}} \right)$.
\end{description}

\subsection{The Linearization operator ${{L}_{\left( A,\Phi  \right)}}$}\label{Hitchin_operator}

We first need to characterize the linearization operator ${{L}_{\left( A,\Phi  \right)}}$ in general before considering this for the particular approximate pair $\left( A_{R}^{app},\Phi _{R}^{app} \right)$ constructed. The differential of the $G$-Hitchin operator at a pair $\left( A,\Phi  \right)\in {{\Omega }^{1}}\left( {{X}_{\#}},{{E}_{H}}\left( {{\mathfrak{h}}^{\mathbb{C}}} \right) \right)\oplus {{\Omega }^{1,0}}\left( {{X}_{\#}},{{E}_{H}}\left( {{\mathfrak{m}}^{\mathbb{C}}} \right) \right)$ is described by \[D\mathsf{\mathcal{H}}\left( \begin{matrix}
   {\dot{A}}  \\
   {\dot{\Phi }}  \\
\end{matrix} \right)=\left( \begin{matrix}
   {{d}_{A}} & \left[ \Phi ,-\tau \left( \cdot  \right)\right]+\left[ \cdot ,-\tau \left( \Phi  \right) \right]  \\
   \left[ \cdot ,\Phi  \right] & {{{\bar{\partial }}}_{A}}  \\
\end{matrix} \right)\left( \begin{matrix}
   {\dot{A}}  \\
   {\dot{\Phi }}  \\
\end{matrix} \right).\]
Moreover, the differential at $g=Id$ of the orbit map ${{\mathsf{\mathcal{O}}}_{\left( A,\Phi  \right)}}$ is
	\[{{\Lambda }_{\left( A,\Phi  \right)}}\gamma =\left( {{\Lambda }_{A}}\left( \gamma  \right),{{\Lambda }_{\Phi }}\left( \gamma  \right) \right) = \left( {{{\bar{\partial }}}_{A}}\gamma -{{\partial }_{A}}{{\gamma }^{*}},\left[ \Phi ,\gamma  \right] \right).\]
Therefore, \[\left( D\mathsf{\mathcal{H}}\circ {{\Lambda }_{\left( A,\Phi  \right)}} \right)\left( \gamma  \right)=\left( \begin{matrix}
   {{\partial }_{A}}{{{\bar{\partial }}}_{A}} \gamma - {{{\bar{\partial }}}_{A}}{{\partial }_{A}} \gamma^{*} +\left[ \Phi ,-\tau \left( \left[ \Phi ,\gamma  \right] \right) \right]+\left[ \left[ \Phi ,\gamma  \right],-\tau \left( \Phi  \right) \right] \\
   \left[ {{{\bar{\partial }}}_{A}}\gamma -{{\partial }_{A}}\gamma^{*} ,\Phi  \right]+{{{\bar{\partial }}}_{A}}\left[ \Phi ,\gamma  \right]  \\
\end{matrix} \right).\]
Now the differential  $D\mathsf{\mathcal{F}}\left( \gamma  \right)$ is the first entry of $\left( D\mathsf{\mathcal{H}}\circ {{\Lambda }_{\left( A,\Phi  \right)}} \right)\left( \gamma  \right)$,
\begin{align*}
D\mathsf{\mathcal{F}}\left( \gamma  \right) &=D\left(  \mathsf{\mathcal{H}}\circ {{\mathsf{\mathcal{O}}}_{\left( A,\Phi  \right)}} \right)\left( \gamma  \right)\\
& = {{\partial }_{A}}{{{\bar{\partial }}}_{A}} \gamma - {{{\bar{\partial }}}_{A}}{{\partial }_{A}} \gamma^{*} +\left[ \Phi ,-\tau \left( \left[ \Phi ,\gamma  \right] \right) \right]+\left[ \left[ \Phi ,\gamma  \right],-\tau \left( \Phi  \right) \right].
\end{align*}
Note that ${{\Lambda }_{\left( A,\Phi  \right)}}:{{\Omega }^{0}}\left( {{X}_{\#}},{{E}_{H}}\left( {{\mathfrak{h}}^{\mathbb{C}}} \right) \right)\to {{\Omega }^{1}}\left( {{X}_{\#}},{{E}_{H}}\left( {{\mathfrak{h}}^{\mathbb{C}}} \right) \right)\oplus {{\Omega }^{1,0}}\left( {{X}_{\#}},{{E}_{H}}\left( {{\mathfrak{g}}^{\mathbb{C}}} \right) \right)$, and 
$$D\mathsf{\mathcal{H}}\circ {{\Lambda }_{\left( A,\Phi  \right)}}:{{\Omega }^{0}}\left( {{X}_{\#}},{{E}_{H}}\left( {{\mathfrak{h}}^{\mathbb{C}}} \right) \right)\to {{\Omega }^{2}}\left( {{X}_{\#}},{{E}_{H}}\left( {{\mathfrak{h}}^{\mathbb{C}}} \right) \right)\oplus {{\Omega }^{1,1}}\left( {{X}_{\#}},{{E}_{H}}\left( {{\mathfrak{g}}^{\mathbb{C}}} \right) \right).$$
We finally apply the operator $-i*:{{\Omega }^{2}}\left( {{X}_{\#}},{{E}_{H}}\left( {{\mathfrak{h}}^{\mathbb{C}}} \right) \right)\to {{\Omega }^{0}}\left( {{X}_{\#}},{{E}_{H}}\left( {{\mathfrak{h}}^{\mathbb{C}}} \right) \right)$ to define the \emph{linearization operator}:
 \[{{L}_{\left( A,\Phi  \right)}}:=-i*D\mathsf{\mathcal{F}}\left( \gamma  \right):{{\Omega }^{0}}\left( {{X}_{\#}},{{E}_{H}}\left( {{\mathfrak{h}}^{\mathbb{C}}} \right) \right)\to {{\Omega }^{0}}\left( {{X}_{\#}},{{E}_{H}}\left( {{\mathfrak{h}}^{\mathbb{C}}} \right) \right).\]

\begin{lemma}\label{linearization}
For $\gamma \in {{\Omega }^{0}}\left( {{X}_{\#}},{{E}_{H}}\left( \mathfrak{h} \right) \right)$, \[{{\left\langle {{L}_{\left( A,\Phi  \right)}}\gamma ,\gamma  \right\rangle }_{{{L}^{2}}}}=\left\| {{d}_{A}}\gamma  \right\|_{{{L}^{2}}}^{2}+2\left\| \left[ \Phi ,\gamma  \right] \right\|_{{{L}^{2}}}^{2}\ge 0.\]
In particular, ${{L}_{\left( A,\Phi  \right)}}\gamma =0$ if and only if ${{d}_{A}}\gamma =\left[ \Phi ,\gamma  \right]=0$.
\end{lemma}

\begin{proof}
That the linearization operator is nonnegative was first observed by C. Simpson in \cite{Simpson-variations}. The argument that follows next generalizes the one from Proposition 5.1 in \cite{MSWW} for $\text{SL(2}\text{,}\mathbb{C}\text{)}$:\\
The compact real form $\tau :{{\mathfrak{g}}^{\mathbb{C}}}\to {{\mathfrak{g}}^{\mathbb{C}}}$ induces an ad-invariant inner product on ${{\mathfrak{g}}^{\mathbb{C}}}$. For operators ${D}':={{\partial }_{A}}+\tau \left( \Phi  \right)$ and ${D}'':={{\bar{\partial }}_{A}}+\Phi $, set  $D={D}'+{D}''$. Then, similarly to  \cite{Simpson-variations} and \cite{Simpson-Higgs} these operators satisfy analogues of the K\"ahler identities, and the calculations of C. Simpson show that 
\[{{L}_{\left( A,\Phi  \right)}}={{D}^{*}}D=2{{\left( {{D}'} \right)}^{*}}{D}'=2{{\left( {{D}''} \right)}^{*}}{D}''.\]
This implies the statement of the proposition.
\end{proof}

\section{Analytic estimates for an approximate solution}

\subsection{The Cappell-Lee-Miller gluing theorem}
A very useful method when dealing with surgery problems in gauge theory over manifolds with very long necks involves the study of the space of eigenfunctions corresponding to small eigenvalues (low eigensolutions) of a self-adjoint Dirac type operator on such a manifold. In this section we explain how the gluing construction for Higgs bundles over a complex connected sum of Riemann surfaces fits into the framework of gluing cylindrical $\mathbb{Z}_{2}$-graded Dirac-type operators. For precise definitions and details on the general method used below we refer the reader to \cite{CLM}, \cite{Nicolaescuarticle}, \cite{Swoboda} and \cite{Yoshida}. We shall  review next only parts of this general framework. 

\begin{definition} Let $\hat{E}\to \hat{N}$ be a ${{\mathbb{Z}}_{2}}$-graded cylindrical hermitian vector bundle over a cylindrical Riemannian manifold $\hat{N}$. A first order partial differential operator $\mathfrak{D}:{{C}^{\infty }}\left( {\hat{E}} \right)\to {{C}^{\infty }}\left( {\hat{E}} \right)$ is called a \emph{${{\mathbb{Z}}_{2}}$-graded cylindrical Dirac-type operator} if with respect to the ${{\mathbb{Z}}_{2}}$-grading of $\hat{E}$, it takes the form
	\[\mathfrak{D} = \left( \begin{matrix}
   0 & {{{\mathsf{\mathcal{D}}}}^{*}}  \\
   {\mathsf{\mathcal{D}}} & 0  \\
\end{matrix} \right),\]
such that along the cylindrical end, $\mathsf{\mathcal{D}}$ is of the form $\mathsf{\mathcal{D}}=G\left( d\tau -D \right)$, for a self-adjoint Dirac-type operator $D:{{C}^{\infty }}\left( {{E}^{+}} \right)\to {{C}^{\infty }}\left( {{E}^{+}} \right)$ and for $G: E^{+}\to E^{-}$ the bundle isomorphism given by the Cliford multiplication by $d\tau$, where $\tau$ is the longitudinal coordinate along the neck.
\end{definition}
Recall that the Dirac-type condition asserts that the square $D^{2}$ has the same principal symbols as a Laplacian and that $D$ is independent of the longitudinal coordinate $\tau$ along the necks.

For our purposes, we will rather need to use the perturbed operator
\[\mathfrak{D} + \mathfrak{B} = \left( \begin{matrix}
   0 & \mathsf{\mathcal{D}}+B  \\
   {{{\mathsf{\mathcal{D}}}}^{*}}+{{B}^{*}} & 0  \\
\end{matrix} \right),\] 
where $B$ is an \emph{exponentially decaying operator} of order 0;  in other words,  there exists a pair of constants $C,\lambda >0$ for which
	\[\sup \left\{ \left| B\left( x \right) \right|\left| x\in \left[ \tau ,\tau +1 \right]\times N \right. \right\}\le C{{e}^{-\lambda \left| \tau  \right|}},\]
for all $\tau \in {{\mathbb{R}}^{+}}$.

A pair of ${{\mathbb{Z}}_{2}}$-graded cylindrical hermitian vector bundles ${{\hat{E}}_{i}}$ is glued together to provide a ${{\mathbb{Z}}_{2}}$-graded hermitian vector bundle ${{E}_{T}}=E_{T}^{+}\oplus E_{T}^{-}$ over the manifold ${{N}_{T}}$, where $T=\left| \log R \right|$ for a real parameter $R$ (cf. \S \ref{gluing_cylindrical_hermitian}). Moreover, a pair of cylindrical operators ${{\mathfrak{D}}_{i}}$ combine to give a ${{\mathbb{Z}}_{2}}$-graded Dirac-type operator ${{\mathfrak{D}}_{T}}$ on the bundle ${{E}_{T}}$. For a pair of perturbed operators, we can also obtain a perturbed Dirac-type operator defined on the bundle ${{E}_{T}}$; let us still denote this by ${{\mathfrak{D}}_{T}}$ and write such an operator as
	\[{{\mathfrak{D}}_{T}}=\left( \begin{matrix}
   0 & \mathsf{\mathcal{D}}_{T}^{*}  \\
   {{{\mathsf{\mathcal{D}}}}_{T}} & 0  \\
\end{matrix} \right).\]
Consider also ${{\mathfrak{D}}_{i,\infty }}:={{\mathfrak{D}}_{i}}+{{\mathfrak{B}}_{i}}$ for $i=1,2$ and write
	\[{{\mathfrak{D}}_{i,\infty }}=\left( \begin{matrix}
   0 & \mathsf{\mathcal{D}}_{i,\infty }^{*}  \\
   {{{\mathsf{\mathcal{D}}}}_{i,\infty }} & 0  \\
\end{matrix} \right).\]

We are going to need one last piece of notation to introduce:
\begin{definition}
Let $\hat{E}$ be a cylindrical vector bundle over the cylindrical manifold $\hat{N}$. We define the \emph{extended ${{L}^{2}}$ space $L_{\text{ext}}^{2}\left( \hat{N},\hat{E} \right)$} as the space of all sections $\hat{u}$ of $\hat{E}$, such that there exists an ${{L}^{2}}$ section ${{u}_{\infty }}$ of $E$ satisfying
	\[\hat{u}-{{\pi }^{*}}{{u}_{\infty }}\in {{L}^{2}}\left( N,E \right).\]
The section ${{u}_{\infty }}$ is uniquely determined by $\hat{u}$, thus the so-called \emph{asymptotic trace map} is well-defined
\begin{align*}
{{\partial }_{\infty }}:L_{\text{ext}}^{2}\left( \hat{N},\hat{E} \right) & \to {{L}^{2}}\left( N,E \right)\\
\hat{u} & \mapsto {{u}_{\infty }}.
\end{align*}
\end{definition}

The following theorem is the version of the Cappell-Lee-Miller gluing theorem, which we are going to apply. For a proof see \cite{Nicolaescuarticle}, \S 5.B:

\begin{theorem}[S. Cappell-R. Lee-E. Miller \cite{CLM}, L. Nicolaescu \cite{Nicolaescuarticle}]\label{CLM}
Let ${{\mathfrak{D}}_{i,\infty }}$ be a pair of ${{\mathbb{Z}}_{2}}$-graded Dirac-type operators on the cylindrical vector bundles ${{\hat{E}}_{i}}\to {{\hat{N}}_{i}}$ for $i=1,2$ as was defined above. Suppose that the kernel  $K_{i}^{+}\subseteq L_{\text{ext}}^{2}\left( {{{\hat{N}}}_{i}},{{{\hat{E}}}_{i}} \right)$ of the operator ${{\mathsf{\mathcal{D}}}_{i,\infty }}$ is trivial for $i=1,2$. Then there exist a ${{T}_{0}}>0$ and a constant $C>0$ such that the operator $\mathsf{\mathcal{D}}_{T}^{*}{{\mathsf{\mathcal{D}}}_{T}}$ is bijective for all $T>{{T}_{0}}$ and admits a bounded inverse ${{\left( \mathsf{\mathcal{D}}_{T}^{*}{{{\mathsf{\mathcal{D}}}}_{T}} \right)}^{-1}}:{{L}^{2}}\left( {{N}_{T}},E_{T}^{+} \right)\to {{L}^{2}}\left( {{N}_{T}},E_{T}^{+} \right)$ with 	
	\[{{\left\| {{\left( \mathsf{\mathcal{D}}_{T}^{*}{{{\mathsf{\mathcal{D}}}}_{T}} \right)}^{-1}} \right\|}_{\mathsf{\mathcal{L}}\left( {{L}^{2}},{{L}^{2}} \right)}}\le C{{T}^{2}}.\]
\end{theorem}

\subsection{The elliptic complex over the complex connected sum}
For our approximate solution $\left( A_{R}^{app},\Phi _{R}^{app} \right)$ constructed over ${{X}_{\#}}$ with $0<R<1$ and $T=-\log R$, consider the elliptic complex:
\begin{align*}
0\xrightarrow{{}}{{\Omega }^{0}}\left( {{X}_{\#}},{{E}_{H}}\left( {{\mathfrak{h}}^{\mathbb{C}}} \right) \right) & \xrightarrow{{{L}_{1,T}}}{{\Omega }^{1}}\left( {{X}_{\#}},{{E}_{H}}\left( {{\mathfrak{h}}^{\mathbb{C}}} \right) \right)\oplus {{\Omega }^{1,0}}\left( {{X}_{\#}},{{E}_{H}}\left( {{\mathfrak{g}}^{\mathbb{C}}} \right) \right)\\
& \xrightarrow{{{L}_{2,T}}}{{\Omega }^{2}}\left( {{X}_{\#}},{{E}_{H}}\left( {{\mathfrak{h}}^{\mathbb{C}}} \right) \right)\oplus {{\Omega }^{2}}\left( {{X}_{\#}},{{E}_{H}}\left( {{\mathfrak{g}}^{\mathbb{C}}} \right) \right)\xrightarrow{{}}0,
\end{align*}
where \[{{L}_{1,T}}\gamma =\left( {{d}_{A_{R}^{app}}}\gamma ,\left[ \Phi _{R}^{app},\gamma  \right] \right)\]
is the linearization of the complex gauge group action and \[{{L}_{2,T}}\left( \alpha ,\varphi  \right)=D\mathsf{\mathcal{H}}\left( \alpha ,\varphi  \right)=\left( \begin{matrix}
   {{d}_{A_{R}^{app}}}\alpha +\left[ \Phi _{R}^{app},-\tau \left( \varphi  \right)\right]+\left[ \varphi ,-\tau \left( \Phi _{R}^{app} \right) \right]   \\
   {{{\bar{\partial }}}_{A_{R}^{app}}}\varphi +\left[ \alpha ,\Phi _{R}^{app} \right]  \\
\end{matrix} \right)\]
is the differential of the Hitchin operator considered in \S \ref{Hitchin_operator}.

Note that in general it does not hold that ${{L}_{2,T}}{{L}_{1,T}}=\left[ {{F}_{A_{R}^{app}}},\gamma  \right]+\left[ \left[ \Phi _{R}^{app},-\tau \left( \Phi _{R}^{app} \right) \right],\gamma  \right]=0$, since $\left( A_{R}^{app},\Phi _{R}^{app} \right)$ need not be an exact solution. Decomposing ${{\Omega }^{*}}\left( {{X}_{\#}},{{E}_{H}}\left( {{\mathfrak{g}}^{\mathbb{C}}} \right) \right)$ into forms of even, respectively odd total degree, we may introduce the  ${{\mathbb{Z}}_{2}}$-graded Dirac-type operator
	\[{{\mathfrak{D}}_{T}}:=\left( \begin{matrix}
   0 & L_{1,T}^{*}+{{L}_{2,T}}  \\
   {{L}_{1,T}}+L_{2,T}^{*} & 0  \\
\end{matrix} \right)\]
on the closed surface ${{X}_{\#}}$.

As $R\searrow 0$, the curve ${{X}_{\#}}$ degenerates to a nodal surface $X_{\#}^{\times }$ (equivalently, the cylindrical neck of ${{X}_{\#}}$ extends infinitely). For the cut-off functions ${{\chi }_{R}}$ that we considered in obtaining the approximate pair $\left( A_{R}^{app},\Phi _{R}^{app} \right)$, their support will tend to be empty as $R\searrow 0$, therefore the ``error regions'' disappear along with the neck $\Omega$, thus $\left( A_{R}^{app},\Phi _{R}^{app} \right)\to \left( {{A}_{0}},{{\Phi }_{0}} \right)$ uniformly on compact subsets with
	\[\left( A_{0}^{app},\Phi _{0}^{app} \right)=\left\{ \begin{matrix}
   \left( {{A}_{1}},{{\Phi }_{1}} \right),\,\,\,{{X}_{1}}\backslash \Omega  \\
   \left( {{A}_{2}},{{\Phi }_{2}} \right),\,\,\,{{X}_{2}}\backslash \Omega  \\
\end{matrix} \right.\]
an exact solution with the holonomy of the associated flat connection in $G$.

For $T=\infty $ the elliptic complex for the exact solution $\left( A_{0}^{app},\Phi _{0}^{app} \right)$ gives rise to the Dirac-type operator
\[{{\mathfrak{D}}_{\infty }}=\left( \begin{matrix}
   0 & L_{1}^{*}+{{L}_{2}}  \\
   {{L}_{1}}+L_{2}^{*} & 0  \\
\end{matrix} \right).\]
We now describe the map ${{L}_{1}}+L_{2}^{*}$ more closely. Using the Hodge $*$-operator we can identify
\[{{\Omega }^{2}}\left( X_{\#}^{\times },{{E}_{H}}\left( {{\mathfrak{h}}^{\mathbb{C}}} \right) \right)\cong {{\Omega }^{0}}\left( X_{\#}^{\times },{{E}_{H}}\left( {{\mathfrak{h}}^{\mathbb{C}}} \right) \right) \text{ and } {{\Omega }^{2}}\left( X_{\#}^{\times },{{E}_{H}}\left( {{\mathfrak{g}}^{\mathbb{C}}} \right) \right)\cong {{\Omega }^{0}}\left( X_{\#}^{\times },{{E}_{H}}\left( {{\mathfrak{g}}^{\mathbb{C}}} \right) \right)\]
as well as ${{\Omega }^{1}}\left( X_{\#}^{\times },{{E}_{H}}\left( {{\mathfrak{h}}^{\mathbb{C}}} \right) \right)\cong {{\Omega }^{0,1}}\left( {{X}_{\#}},{{E}_{H}}\left( {{\mathfrak{g}}^{\mathbb{C}}} \right) \right)$ via the projection $A\mapsto {{\pi }^{0,1}}A$. We further identify \[\left( {{\gamma }_{1}},{{\gamma }_{2}} \right)\in {{\Omega }^{0}}\left( X_{\#}^{\times },{{E}_{H}}\left( {{\mathfrak{h}}^{\mathbb{C}}} \right) \right)\oplus {{\Omega }^{0}}\left( X_{\#}^{\times },{{E}_{H}}\left( {{\mathfrak{h}}^{\mathbb{C}}} \right) \right)\] with ${{\psi }_{1}}={{\gamma }_{1}}+i{{\gamma }_{2}}\in {{\Omega }^{0}}\left( X_{\#}^{\times },{{E}_{H}}\left( {{\mathfrak{g}}^{\mathbb{C}}} \right) \right)$. The operator ${{L}_{1}}+L_{2}^{*}$ can be now expressed as the map
\begin{align*}
{{\Omega }^{0}}\left( X_{\#}^{\times },{{E}_{H}}\left( {{\mathfrak{g}}^{\mathbb{C}}} \right) \right)\oplus {{\Omega }^{0}}\left( X_{\#}^{\times },{{E}_{H}}\left( {{\mathfrak{g}}^{\mathbb{C}}} \right) \right) & \to {{\Omega }^{0,1}}\left( X_{\#}^{\times },{{E}_{H}}\left( {{\mathfrak{g}}^{\mathbb{C}}} \right) \right)\oplus {{\Omega }^{1,0}}\left( X_{\#}^{\times },{{E}_{H}}\left( {{\mathfrak{g}}^{\mathbb{C}}} \right) \right)\\
\left( {{\psi }_{1}},{{\psi }_{2}} \right) & \mapsto \left( \begin{matrix}
   {{{\bar{\partial }}}_{A_{0}^{app}}}{{\psi }_{1}}+\left[ {{\psi }_{2}},-\tau \left( \Phi _{0}^{app} \right) \right]  \\
   {{\partial }_{A_{0}^{app}}}{{\psi }_{2}}+\left[ {{\psi }_{1}},\Phi _{0}^{app} \right]  \\
\end{matrix} \right).
\end{align*}

\subsection{${{\mathfrak{D}}_{\infty }}$ is an exponentially small perturbation of a cylindrical operator}

Consider the operator ${{\mathfrak{\hat{D}}}_{\infty }}:=\left( \begin{matrix}
   0 & \hat{L}_{1}^{*}+{{{\hat{L}}}_{2}}  \\
   {{{\hat{L}}}_{1}}+\hat{L}_{2}^{*} & 0  \\
\end{matrix} \right)$
arising similarly from the elliptic complex for some model solution $\left( {{A}^{\bmod }},{{\Phi }^{\bmod }} \right)$ replacing $\left( A_{0}^{app},\Phi _{0}^{app} \right)$, and for which
	\[\left( {{A}^{\bmod }},{{\Phi }^{\bmod }} \right)=\left( 0,\varphi \frac{dz}{z} \right)\]
along each cylindrical neck. The operator ${{\mathfrak{\hat{D}}}_{\infty }}$ is in fact cylindrical. Indeed, introducing the complex coordinate $\zeta =\tau +i\theta $, we have the identities $d\tau =-\frac{dr}{r}$, $d\theta =-d\theta $, $\frac{dz}{z}=-d\zeta $, and $\frac{d\bar{z}}{{\bar{z}}}=-d\bar{\zeta }$. Hence the operator ${{\hat{L}}_{1}}+\hat{L}_{2}^{*}$ (as well as the operator $\hat{L}_{1}^{*}+{{\hat{L}}_{2}}$ similarly) can be written as a cylindrical differential operator ${{\hat{L}}_{1}}+\hat{L}_{2}^{*}=\frac{\sqrt{2}}{2}G\left( {{\partial }_{\tau }}-D \right)$ with
\[{{\hat{L}}_{1}}+\hat{L}_{2}^{*}: \left( {{\psi }_{1}},{{\psi }_{2}} \right)\mapsto \frac{1}{2}\left( \begin{matrix}
   {{\partial }_{\tau }}{{\psi }_{1}}d\bar{\zeta }  \\
   {{\partial }_{\tau }}{{\psi }_{2}}d\zeta   \\
\end{matrix} \right)-\left( \begin{matrix}
   \left( \frac{i}{2}{{\partial }_{\theta }}{{\psi }_{1}}+\left[ {{\psi }_{2}},\tau \left( \varphi  \right) \right] \right)d\bar{\zeta }  \\
   \left( -\frac{i}{2}{{\partial }_{\theta }}{{\psi }_{2}}-\left[ {{\psi }_{1}},\varphi  \right] \right)d\zeta   \\
\end{matrix} \right),\]
where
\begin{equation}\label{form_of_D}
  D\left( {{\psi }_{1}},{{\psi }_{2}} \right):=2\left( \begin{matrix}
   {\frac{i}{2}{{\partial }_{\theta }}{{\psi }_{1}}+\left[ {{\psi }_{2}},\tau \left( \varphi  \right) \right]}  \\
   {-\frac{i}{2}{{\partial }_{\theta }}{{\psi }_{2}}-\left[ {{\psi }_{1}},\varphi  \right]}  \\
\end{matrix} \right)
\end{equation}
and $G:\left( {{\psi }_{1}},{{\psi }_{2}} \right) \mapsto \frac{\sqrt{2}}{2}\left( {{\psi }_{1}}d\bar{\zeta },{{\psi }_{2}}d\zeta  \right)$ denotes the Clifford multiplication by $d\tau $.

By construction of the approximate solution  $\left( A_{R}^{app},\Phi _{R}^{app} \right)$ and the decay described in Lemma \ref{Biq_Bo_lemma}, one sees that the operator ${{\mathfrak{D}}_{\infty }}$ is indeed an exponentially small perturbation of ${{\mathfrak{\hat{D}}}_{\infty }}$; see \cite{Swoboda}, p. 667 for more details.

\subsection{The space $\ker \left( {{L}_{1}}+L_{2}^{*} \right)\cap L_{\text{ext}}^{2}\left( X_{\#}^{\times } \right)$ is trivial}\label{invertibility}

We now restrict to the case $G=\text{Sp(4}\text{,}\mathbb{R}\text{)}$ in order to study the space $\ker \left( {{L}_{1}}+L_{2}^{*} \right)\cap L_{\text{ext}}^{2}\left( X_{\#}^{\times } \right)$ for the operator ${{\mathfrak{D}}_{\infty }}$ more closely. We are also taking here into consideration the particular model Higgs field we picked for the $G=\text{Sp(4}\text{,}\mathbb{R}\text{)}$-Hitchin equations coming from the particular embeddings ${{\phi }_{irr}}$ and $\psi$ from (\ref{irreducible_rep}) and (\ref{diagonal_Fuchsian}). In other words, we fix
	\[\varphi \equiv {{\varphi }^{\bmod }}=\left( \begin{matrix}
   3C & 0 & 0 & 0  \\
   0 & C & 0 & 0  \\
   0 & 0 & -3C & 0  \\
   0 & 0 & 0 & -C  \\
\end{matrix} \right),\]
for a nonzero real constant $C$. Moreover, the compact real form on $\varphi$ in this case is $\tau \left( \varphi  \right)=-{{\varphi }^{*}}$. We have the following:
\begin{proposition}\label{asympt_trace}
Let $\left( {{\psi }_{1}},{{\psi }_{2}} \right)\in \ker \left( {{L}_{1}}+L_{2}^{*} \right)\cap L_{\text{ext}}^{2}\left( X_{\#}^{\times } \right)$. Then its asymptotic trace is described by
	\[{{\partial }_{\infty }}\left( {{\psi }_{1}},{{\psi }_{2}} \right)=\left( \left( \begin{matrix}
   {{a}_{1}} & 0 & 0 & 0  \\
   0 & {{d}_{1}} & 0 & 0  \\
   0 & 0 & -{{a}_{1}} & 0  \\
   0 & 0 & 0 & -{{d}_{1}}  \\
\end{matrix} \right),\left( \begin{matrix}
   {{a}_{2}} & 0 & 0 & 0  \\
   0 & {{d}_{2}} & 0 & 0  \\
   0 & 0 & -{{a}_{2}} & 0  \\
   0 & 0 & 0 & -{{d}_{2}}  \\
\end{matrix} \right) \right)\]
for constants ${{a}_{i}},{{d}_{i}}\in \mathbb{C}$, for $i=1,2$.
\end{proposition}

\begin{proof}
By \cite{Nicolaescuarticle}, p. 169, the space of asymptotic traces of $\ker \left( {{L}_{1}}+L_{2}^{*} \right)$ is a subspace of $\ker D$ with $D$ as defined in (\ref{form_of_D}). We will check that the elements of the latter have the asserted form. Consider the Fourier decomposition $\left( {{\psi }_{1}},{{\psi }_{2}} \right)=\left( \sum\nolimits_{j\in \mathbb{Z}}{{{\psi }_{1,j}}{{e}^{ij\vartheta }},\sum\nolimits_{j\in \mathbb{Z}}{{{\psi }_{2,j}}{{e}^{ij\vartheta }}}} \right)$, where
	\[{{\psi }_{i,j}}\in \mathfrak{sp}\left( 4,\mathbb{C} \right)=\left\{ \left( \begin{matrix}
   A & B  \\
   C & -{{A}^{T}}  \\
\end{matrix} \right)\left| A,B,C\in {{M}_{2\times 2}}\left( \mathbb{C} \right);\,\,{{B}^{T}}=B,{{C}^{T}}=C \right. \right\}.\]
Then the equation $D\left( {{\psi }_{1}},{{\psi }_{2}} \right)=0$ is equivalent to the system of linear equations
\begin{equation}\label{linear_system}
  \left( \begin{matrix}
   -\frac{j}{2}{{\psi }_{1,j}}+\left[ {{\varphi }^{*}},{{\psi }_{2,j}} \right]  \\
   \frac{j}{2}{{\psi }_{2,j}}+\left[ \varphi ,{{\psi }_{1,j}} \right]  \\
\end{matrix} \right)=0
\end{equation}
for $j\in \mathbb{Z}$. Since the Higgs field $\varphi$ is diagonal, the operator $D$ acts invariantly on diagonal, respectively off-diagonal endomorphisms. It therefore suffices to consider these two cases separately. \\
\textbf{Case 1.} Let $\left( {{\psi }_{1,j}},{{\psi }_{2,j}} \right)=\left( \left( \begin{matrix}
   {{a}_{1,j}} & 0 & 0 & 0  \\
   0 & {{d}_{1,j}} & 0 & 0  \\
   0 & 0 & -{{a}_{1,j}} & 0  \\
   0 & 0 & 0 & -{{d}_{1,j}}  \\
\end{matrix} \right),\left( \begin{matrix}
   {{a}_{2,j}} & 0 & 0 & 0  \\
   0 & {{d}_{2,j}} & 0 & 0  \\
   0 & 0 & -{{a}_{2,j}} & 0  \\
   0 & 0 & 0 & -{{d}_{2,j}}  \\
\end{matrix} \right) \right)$, with ${{a}_{i,j}},{{d}_{i,j}}\in \mathbb{C}$ for $i=1,2$. Then Equation (\ref{linear_system}) is equivalent to the pair of equations \[\frac{j}{2}\left( \begin{matrix}
   {{a}_{i,j}} & 0 & 0 & 0  \\
   0 & {{d}_{i,j}} & 0 & 0  \\
   0 & 0 & -{{a}_{i,j}} & 0  \\
   0 & 0 & 0 & -{{d}_{i,j}}  \\
\end{matrix} \right)=\mathbb{O},\text{ for }i=1,2,\]
thus the system has a non-trivial solution if and only if $j=0$. In other words,
${{\psi }_{1}}={{\psi }_{1,0}}$ and ${{\psi }_{2}}={{\psi }_{2,0}}$ are of the asserted form. \\
\textbf{Case 2.} Let now $\left( {{\psi }_{1,j}},{{\psi }_{2,j}} \right)=\left( \left( \begin{matrix}
   0 & {{b}_{1,j}} & {{e}_{1,j}} & {{f}_{1,j}}  \\
   {{c}_{1,j}} & 0 & {{f}_{1,j}} & {{g}_{1,j}}  \\
   {{k}_{1,j}} & {{l}_{1,j}} & 0 & -{{c}_{1,j}}  \\
   {{l}_{1,j}} & {{m}_{1,j}} & -{{b}_{1,j}} & 0  \\
\end{matrix} \right),\left( \begin{matrix}
   0 & {{b}_{2,j}} & {{e}_{2,j}} & {{f}_{2,j}}  \\
   {{c}_{2,j}} & 0 & {{f}_{2,j}} & {{g}_{2,j}}  \\
   {{k}_{2,j}} & {{l}_{2,j}} & 0 & -{{c}_{2,j}}  \\
   {{l}_{2,j}} & {{m}_{2,j}} & -{{b}_{2,j}} & 0  \\
\end{matrix} \right) \right)$ with all entries in $\mathbb{C}$. Then Equation (\ref{linear_system}) reads as the pair of equations
\[-\frac{j}{2}\left( \begin{matrix}
   0 & {{b}_{1,j}} & {{e}_{1,j}} & {{f}_{1,j}}  \\
   {{c}_{1,j}} & 0 & {{f}_{1,j}} & {{g}_{1,j}}  \\
   {{k}_{1,j}} & {{l}_{1,j}} & 0 & -{{c}_{1,j}}  \\
   {{l}_{1,j}} & {{m}_{1,j}} & -{{b}_{1,j}} & 0  \\
\end{matrix} \right)=\left( \begin{matrix}
   0 & -2{{b}_{2,j}}C & -6{{e}_{2,j}}C & -4{{f}_{2,j}}C  \\
   2{{c}_{2,j}}C & 0 & -4{{f}_{2,j}}C & -2{{g}_{2,j}}C  \\
   6{{k}_{2,j}}C & 4{{l}_{2,j}}C & 0 & -2{{c}_{2,j}}C  \\
   4{{l}_{2,j}}C & 2{{m}_{2,j}}C & 2{{b}_{2,j}}C & 0  \\
\end{matrix} \right)\] and
	\[\frac{j}{2}\left( \begin{matrix}
   0 & {{b}_{2,j}} & {{e}_{2,j}} & {{f}_{2,j}}  \\
   {{c}_{2,j}} & 0 & {{f}_{2,j}} & {{g}_{2,j}}  \\
   {{k}_{2,j}} & {{l}_{2,j}} & 0 & -{{c}_{2,j}}  \\
   {{l}_{2,j}} & {{m}_{2,j}} & -{{b}_{2,j}} & 0  \\
\end{matrix} \right)=\left( \begin{matrix}
   0 & -2{{b}_{1,j}}C & -6{{e}_{1,j}}C & -4{{f}_{1,j}}C  \\
   2{{c}_{1,j}}C & 0 & -4{{f}_{1,j}}C & -2{{g}_{1,j}}C  \\
   6{{k}_{1,j}}C & 4{{l}_{1,j}}C & 0 & -2{{c}_{1,j}}C  \\
   4{{l}_{1,j}}C & 2{{m}_{1,j}}C & 2{{b}_{1,j}}C & 0  \\
\end{matrix} \right).\]
This pair of equations is then equivalent to the equation
\begin{equation}\label{basic_2by2}
  \left( \begin{matrix}
   \frac{j}{2} & -2C  \\
   2C & \frac{j}{2}  \\
\end{matrix} \right)\left( \begin{matrix}
   {{b}_{1,j}}  \\
   {{b}_{2,j}}  \\
\end{matrix} \right)=\left( \begin{matrix}
   0  \\
   0  \\
\end{matrix} \right)
\end{equation}
and seven more similar equations involving the ${{c}_{i,j}},{{e}_{i,j}},{{f}_{i,j}},{{g}_{i,j}},{{k}_{i,j}},{{l}_{i,j}},{{m}_{i,j}}$, for $i=1,2$ and $j\in \mathbb{Z}$. Since $C\ne 0$, the determinant of the $2\times 2$ matrix in Equation (\ref{basic_2by2}) is ${{\left( \frac{j}{2} \right)}^{2}}+4C^2>0$, and so this system has no non-trivial solution for $\left({{b}_{1,j}}, {{b}_{2,j}}\right)$; the same is true for the remaining seven equations. Therefore, there are no non-trivial off-diagonal elements in $\ker D$ and so the only non-trivial elements are of the asserted form in the proposition.
\end{proof}

\begin{lemma}\label{dpsi}
Suppose $\left( {{\psi }_{1}},{{\psi }_{2}} \right)\in \ker \left( {{L}_{1}}+L_{2}^{*} \right)\cap L_{\text{ext}}^{2}\left( X_{\#}^{\times } \right)$. Then 
\[{{d}_{A_{0}^{app}}}{{\psi }_{i}}=\left[ {{\psi }_{i}},\Phi _{0}^{app} \right]=\left[ {{\psi }_{i}},{{\left( \Phi _{0}^{app} \right)}^{*}} \right]=0,\] 
for $i=1,2$.
\end{lemma}

\begin{proof}
By definition of the operator $\left( {{L}_{1}}+L_{2}^{*} \right)$, an element $\left( {{\psi }_{1}},{{\psi }_{2}} \right)$ lies in the kernel of this operator if and only if it is a solution to the system
\begin{equation}\label{system}
  \left\{ \begin{matrix}
   0={{{\bar{\partial }}}_{A_{0}^{app}}}{{\psi }_{1}}+\left[ {{\psi }_{2}},{{\left( \Phi _{0}^{app} \right)}^{*}} \right]  \\
   0={{\partial }_{A_{0}^{app}}}{{\psi }_{2}}+\left[ {{\psi }_{1}},\Phi _{0}^{app} \right].  \\
\end{matrix} \right.
\end{equation}
Differentiate the first equation and use that ${{\partial }_{A_{0}^{app}}}{{\left( \Phi _{0}^{app} \right)}^{*}}=0$ to imply that
\begin{align*}
0 & = {{\partial }_{A_{0}^{app}}}{{\bar{\partial }}_{A_{0}^{app}}}{{\psi }_{1}}-\left[ {{\partial }_{A_{0}^{app}}}{{\psi }_{2}},{{\left( \Phi _{0}^{app} \right)}^{*}} \right]\\
& = {{\partial }_{A_{0}^{app}}}{{\bar{\partial }}_{A_{0}^{app}}}{{\psi }_{1}}+\left[ \left[ {{\psi }_{1}},\Phi _{0}^{app} \right],{{\left( \Phi _{0}^{app} \right)}^{*}} \right].
\end{align*}
From this it follows that
\begin{align*}
\partial \left\langle {{{\bar{\partial }}}_{A_{0}^{app}}}{{\psi }_{1}},{{\psi }_{1}} \right\rangle & = \left\langle {{\partial }_{A_{0}^{app}}}{{{\bar{\partial }}}_{A_{0}^{app}}}{{\psi }_{1}},{{\psi }_{1}} \right\rangle -\left\langle {{{\bar{\partial }}}_{A_{0}^{app}}}{{\psi }_{1}},{{{\bar{\partial }}}_{A_{0}^{app}}}{{\psi }_{1}} \right\rangle \\
& = -{{\left| \left[ {{\psi }_{1}},\Phi _{0}^{app} \right] \right|}^{2}}-{{\left| {{{\bar{\partial }}}_{A_{0}^{app}}}{{\psi }_{1}} \right|}^{2}}
\end{align*}
and similarly 
\[\bar{\partial }\left\langle {{\partial }_{A_{0}^{app}}}{{\psi }_{1}},{{\psi }_{1}} \right\rangle =-{{\left| \left[ {{\psi }_{1}},{{\left( \Phi _{0}^{app} \right)}^{*}} \right] \right|}^{2}}-{{\left| {{\partial }_{A_{0}^{app}}}{{\psi }_{1}} \right|}^{2}}.\]
Now let ${{X}_{S}}:=X_{\#}^{\times }\backslash \bigcup\limits_{p\in \mathfrak{p}}{{{\text{C}}_{p}}}\left( S \right)$, where for $S>0$ we denote by ${{\text{C}}_{p}}\left( S \right)$ the subcylinders of points $\left( \tau ,\vartheta  \right)\in {{\text{C}}_{p}}\left( 0 \right)$ with $\tau \ge S$. From Stokes' theorem one has
\[\int\limits_{{{X}_{S}}}{\partial \left\langle {{{\bar{\partial }}}_{A_{0}^{app}}}{{\psi }_{1}},{{\psi }_{1}} \right\rangle }+\bar{\partial }\left\langle {{\partial }_{A_{0}^{app}}}{{\psi }_{1}},{{\psi }_{1}} \right\rangle =\int\limits_{\partial {{X}_{S}}}{\left\langle {{d}_{A_{0}^{app}}}{{\psi }_{1}},{{\psi }_{1}} \right\rangle }.\]
Letting $S\to \infty $, ${{\psi }_{1}}\left| _{\tau =S} \right.$ ${{L}^{2}}$-converges to its asymptotic trace ${{\partial }_{\infty }}{{\psi }_{1}}\in {{\Omega }^{0}}\left( {{S}^{1}},\mathfrak{sp}\left( 4,\mathbb{C} \right) \right)$, which by Proposition \ref{asympt_trace} is of the form
	\[{{\psi }_{1}}\left( \infty  \right)=\left( \begin{matrix}
   {{a}_{1}} & 0 & 0 & 0  \\
   0 & {{d}_{1}} & 0 & 0  \\
   0 & 0 & -{{a}_{1}} & 0  \\
   0 & 0 & 0 & -{{d}_{1}}  \\
\end{matrix} \right),\]
for ${{a}_{1}},{{d}_{1}}\in \mathbb{C}$. Therefore, ${{d}_{A_{0}^{app}}}\left( {{\partial }_{\infty }}{{\psi }_{1}}\left( \infty  \right) \right)=0$ and so
	\[\int\limits_{X_{\#}^{\times }}{\partial \left\langle {{{\bar{\partial }}}_{A_{0}^{app}}}{{\psi }_{1}},{{\psi }_{1}} \right\rangle }+\bar{\partial }\left\langle {{\partial }_{A_{0}^{app}}}{{\psi }_{1}},{{\psi }_{1}} \right\rangle =\underset{S\to \infty }{\mathop{\lim }}\,\int\limits_{\partial {{X}_{S}}}{\left\langle {{d}_{A_{0}^{app}}}{{\psi }_{1}},{{\psi }_{1}} \right\rangle }=0.\]
This implies that
${{\bar{\partial }}_{A_{0}^{app}}}{{\psi }_{1}}={{\partial}_{A_{0}^{app}}}{{\psi }_{1}}=\left[ {{\psi }_{1}},\Phi _{0}^{app} \right]=\left[ {{\psi }_{1}},{{\left( \Phi _{0}^{app} \right)}^{*}} \right]=0$.\\
We may as well derive that ${{\bar{\partial }}_{A_{0}^{app}}}{{\psi }_{2}}={{\partial}_{A_{0}^{app}}}{{\psi }_{2}}=\left[ {{\psi }_{2}},\Phi _{0}^{app} \right]=\left[ {{\psi }_{2}},{{\left( \Phi _{0}^{app} \right)}^{*}} \right]=0$ by taking the hermitian adjoint of Equation (\ref{system}) and repeating the same arguments for the solution $\left( A_{0}^{app},-\Phi _{0}^{app} \right)$.
\end{proof}

\begin{proposition}
The operator ${{L}_{1}}+L_{2}^{*}$ considered as a densely defined operator on $L_{\text{ext}}^{2}\left( X_{\#}^{\times } \right)$ has trivial kernel.
\end{proposition}

\begin{proof}
Let $\left( {{\psi }_{1}},{{\psi }_{2}} \right)\in \ker \left( {{L}_{1}}+L_{2}^{*} \right)\cap L_{\text{ext}}^{2}\left( X_{\#}^{\times } \right)$. From Lemma \ref{dpsi} we have
 \[{{d}_{A_{0}^{app}}}{{\psi }_{i}}=\left[ {{\psi }_{i}},\Phi _{0}^{app} \right]=\left[ {{\psi }_{i}},{{\left( \Phi _{0}^{app} \right)}^{*}} \right]=0,\]
 for $i=1,2$. We show that ${{\psi }_{1}}=0$ by showing that $\gamma :={{\psi }_{1}}+\psi _{1}^{*}\in {{\Omega }^{0}}\left( X_{\#}^{\times },\mathfrak{u}\left( 2 \right) \right)$ and $\delta :=i\left( {{\psi }_{1}}-\psi _{1}^{*} \right)\in {{\Omega }^{0}}\left( X_{\#}^{\times },\mathfrak{u}\left( 2 \right) \right)$ both vanish. Choosing a holomorphic coordinate $z$ centered at the node of $X_{\#}^{\times }$, the Higgs field $\Phi _{0}^{app}$ in our exact solution is written
	\[\Phi _{0}^{app}=\varphi \frac{dz}{z}\]
with $\varphi \in {{\mathfrak{m}}^{\mathbb{C}}}\left( \text{Sp(4}\text{,}\mathbb{R}\text{)} \right)=\left\{ \left( \begin{matrix}
   A & B  \\
   B & -A  \\
\end{matrix} \right)\left| A,B\in {{\mathsf{\mathcal{M}}}_{2}}\left( \mathbb{C} \right)\text{ with }{{A}^{T}}=A,\,\,{{B}^{T}}=B \right. \right\}$. We get that $d{{\left| \gamma  \right|}^{2}}=2\left\langle {{d}_{A_{0}^{app}}}\gamma ,\gamma  \right\rangle =0$, in other words, $\left| \gamma  \right|$ is constant on $X_{\#}^{\times }$, as well as that $\gamma \left( x \right)$ lies in the kernel of the linearization operator.\\
Now, this $\gamma \left( x \right) \in \mathfrak{u}\left( 2 \right)$ is hermitian. It has orthogonal eigenvectors for distinct eigenvalues, but even if there are degenerate eigenvalues, it is still possible to find an orthonormal basis of ${{\mathbb{C}}^{4}}$ consisting of four eigenvectors of $\gamma \left( x \right)$, thus ${{\mathbb{C}}^{4}}={{E}_{{{\lambda }_{1}}}}\oplus \ldots \oplus {{E}_{{{\lambda }_{4}}}}$, where ${{\lambda }_{i}}$ the eigenvalues of $\gamma \left( x \right)$. Assuming that $\gamma \left( x \right)$ is non-zero, since $\left[ \varphi \left( x \right),\gamma \left( x \right) \right]=0$ it follows that $\varphi \left( x \right)$ preserves the eigenspaces of $\gamma \left( x \right)$ for all $x\in X_{\#}^{\times }$ and so $\left\langle \varphi \left( x \right)v,\varphi \left( x \right)w \right\rangle =\left\langle v,w \right\rangle $ for $v,w\in {{\mathbb{C}}^{4}}$. In other words, $\varphi \left( x \right)$ ought to be an isometry with respect to the usual norm in ${{\mathbb{C}}^{4}}$. Equivalently, $\varphi \left( x \right)$ is unitary for all $x\in X_{\#}^{\times }$. However, for a zero ${{x}_{0}}$ of $\det \Phi =\det \tilde{\varphi }\left( {{x}_{0}} \right)\frac{d{{z}^{2}}}{{{z}^{2}}}$ chosen on the left hand side surface ${{X}_{l}}$ of $X_{\#}^{\times }$ we see that
	\[\varphi \left( {{x}_{0}} \right)={{\phi }_{irr*}}\left( \begin{matrix}
   0 & 1  \\
   z & 0  \\
\end{matrix} \right)=\left( \begin{matrix}
   0 & -\sqrt{3} & 0 & 0  \\
   -\sqrt{3}z & 0 & 0 & 2  \\
   0 & 0 & 0 & \sqrt{3}z  \\
   0 & 2z & \sqrt{3} & 0  \\
\end{matrix} \right),\]
which is not unitary. Therefore, $\gamma=0$.\\
That $\delta$ vanishes, as well as ${{\psi }_{2}}=0$, is proven similarly.
\end{proof}

\begin{remark}
The method described in this subsection for showing that the linearization operator in the case $G = \text{Sp(4}\text{,}\mathbb{R}\text{)}$ is invertible can be adapted to study this problem for other split real Lie groups accordingly.
\end{remark}

\subsection{Upper bound for ${{L}_{\left( A_{R}^{app},\Phi _{R}^{app} \right)}}$ in ${{H}^{2}}\left( X_{\#}^{\times } \right)$}

Define the operator
\[{{\mathsf{\mathcal{D}}}_{T}}:={{L}_{1,T}}+L_{2,T}^{*}.\]
The following proposition is an immediate consequence of the Cappell-Lee-Miller theorem (Theorem \ref{CLM}) for this operator ${{\mathsf{\mathcal{D}}}_{T}}$ using the fact that the kernel of the limiting operator ${{L}_{1}}+L_{2}^{*}$ is trivial on $L_{\text{ext}}^{2}\left( X_{\#}^{\times } \right)$:

\begin{proposition} There exist constants ${{T}_{0}}>0$ and $C>0$ such that the operator $\mathsf{\mathcal{D}}_{T}^{*}{{\mathsf{\mathcal{D}}}_{T}}$ is bijective for all $T>{{T}_{0}}$ and its inverse ${{\left( \mathsf{\mathcal{D}}_{T}^{*}{{\mathsf{\mathcal{D}}}_{T}} \right)}^{-1}}:{{L}^{2}}\left( {{X}_{\#}} \right)\to {{L}^{2}}\left( {{X}_{\#}} \right)$ satisfies
	\[{{\left\| {{\left( \mathsf{\mathcal{D}}_{T}^{*}{{\mathsf{\mathcal{D}}}_{T}} \right)}^{-1}} \right\|}_{\mathsf{\mathcal{L}}\left( {{L}^{2}},{{L}^{2}} \right)}}\le C{{T}^{2}}.\]
\end{proposition}

We are finally in position to imply the existence of the inverse operator ${{G}_{R}}=L_{\left( A_{R}^{app},\Phi _{R}^{app} \right)}^{-1}:{{L}^{2}}\left( {{X}_{\#}} \right)\to {{L}^{2}}\left( {{X}_{\#}} \right)$ and provide an upper bound for its norm, by adapting the analogous proof from \cite{Swoboda} into our case. We first need the following:
\begin{corollary}\label{basic_corollary}
There exist constants ${{T}_{0}}>0$ and $C>0$ such that for all $T>{{T}_{0}}$ and $\gamma \in {{\Omega }^{0}}\left( {{X}_{\#}},{{E}_{H}}\left( \mathfrak{h}^{\mathbb{C}} \right) \right)$ it holds that
	\[{{\left\| L_{1,T}^{*}{{L}_{1,T}}\gamma  \right\|}_{{{L}^{2}}\left( {{X}_{\#}} \right)}}\ge C{{T}^{-2}}{{\left\| \gamma  \right\|}_{{{L}^{2}}\left( {{X}_{\#}} \right)}}.\]
\end{corollary}
\begin{proof}
The previous proposition provides the existence of constants ${{T}_{0}}>0$ and $C>0$ such that for all $T>{{T}_{0}}$ and $\gamma \in {{\Omega }^{0}}\left( {{X}_{\#}},{{E}_{H}}\left( \mathfrak{h}^{\mathbb{C}} \right) \right)$:
\[{{\left\| {{\left( \mathsf{\mathcal{D}}_{T}^{*}{{\mathsf{\mathcal{D}}}_{T}} \right)}^{-1}}\gamma  \right\|}_{{{L}^{2}}\left( {{X}_{\#}} \right)}}\le C{{T}^{2}}{{\left\| \gamma  \right\|}_{{{L}^{2}}\left( {{X}_{\#}} \right)}}\]
and thus \[{{\left\| \mathsf{\mathcal{D}}_{T}^{*}{{\mathsf{\mathcal{D}}}_{T}}\gamma  \right\|}_{{{L}^{2}}\left( {{X}_{\#}} \right)}}\ge C{{T}^{-2}}{{\left\| \gamma  \right\|}_{{{L}^{2}}\left( {{X}_{\#}} \right)}}.\]
According to the definition of ${{\mathsf{\mathcal{D}}}_{T}}$ we have
\begin{align*}
\mathsf{\mathcal{D}}_{T}^{*}{{\mathsf{\mathcal{D}}}_{T}} & = {{\left( {{L}_{1,T}}+L_{2,T}^{*} \right)}^{*}}\left( {{L}_{1,T}}+L_{2,T}^{*} \right)\\
& = L_{1,T}^{*}{{L}_{1,T}}+{{L}_{2,T}}{{L}_{1,T}}+L_{1,T}^{*}L_{2,T}^{*}+{{L}_{2,T}}L_{2,T}^{*},
\end{align*}
as well as ${{L}_{2,T}}{{L}_{1,T}}\gamma =\left[ {{F}_{A_{R}^{app}}},\gamma  \right]+\left[ \left[ \Phi _{R}^{app},-\tau \left( \Phi _{R}^{app} \right) \right],\gamma  \right]$, for sections $\gamma \in {{\Omega }^{0}}\left( {{X}_{\#}},{{E}_{H}}\left( \mathfrak{h}^{\mathbb{C}} \right) \right)$. For the parameter $T=-\log R$, Lemma \ref{approximate_estimate} provides the estimate
\begin{align*}
{{\left\| {{L}_{2,T}}{{L}_{1,T}}\gamma  \right\|}_{{{L}^{2}}\left( {{X}_{\#}} \right)}} & \le {{C}_{1}}{{R}^{{{\delta }''}}}{{\left\| \gamma  \right\|}_{{{L}^{2}}\left( {{X}_{\#}} \right)}}\\
& = {{C}_{1}}{{e}^{-{\delta }''T}}{{\left\| \gamma  \right\|}_{{{L}^{2}}\left( {{X}_{\#}} \right)}},
\end{align*}
for $T$-independent constants ${{C}_{1}},{\delta }''>0$.\\
Remember that the operator $\mathsf{\mathcal{D}}_{T}^{*}{{\mathsf{\mathcal{D}}}_{T}}$ acts on forms of even total degree. Now, decomposing forms of even total degree into forms of degree zero and degree two, for a 0-form $\gamma$ we may write $\gamma=\gamma+0$ and thus is
\[L_{1,T}^{*}{{L}_{1,T}}\gamma =\mathsf{\mathcal{D}}_{T}^{*}{{\mathsf{\mathcal{D}}}_{T}}\gamma -{{L}_{2,T}}{{L}_{1,T}}\gamma. \]
The triangle inequality now provides that
\begin{align*}
{{\left\| L_{1,T}^{*}{{L}_{1,T}}\gamma  \right\|}_{{{L}^{2}}\left( {{X}_{\#}} \right)}} & \ge {{\left\| \mathsf{\mathcal{D}}_{T}^{*}{{\mathsf{\mathcal{D}}}_{T}}\gamma  \right\|}_{{{L}^{2}}\left( {{X}_{\#}} \right)}}-{{\left\| {{L}_{2,T}}{{L}_{1,T}}\gamma  \right\|}_{{{L}^{2}}\left( {{X}_{\#}} \right)}}\\
& \ge C{{T}^{-2}}{{\left\| \gamma  \right\|}_{{{L}^{2}}\left( {{X}_{\#}} \right)}}-{{C}_{1}}{{e}^{-{\delta }''T}}{{\left\| \gamma  \right\|}_{{{L}^{2}}\left( {{X}_{\#}} \right)}},
\end{align*}
which in turn for sufficiently large $T$ implies the desired inequality.
\end{proof}

\begin{proposition}
There exist constants ${{R}_{0}}>0$ and $C>0$, such that for all sufficiently small $0<R<{{R}_{0}}$ the operator ${{L}_{\left( A_{R}^{app},\Phi _{R}^{app} \right)}}$ is invertible and its inverse ${{G}_{R}}=L_{\left( A_{R}^{app},\Phi _{R}^{app} \right)}^{-1}$ satisfies the estimate
	\[{{\left\| {{G}_{R}}\gamma  \right\|}_{{{L}^{2}}\left( {{X}_{\#}} \right)}}\le C{{\left| \log R \right|}^{2}}{{\left\| \gamma  \right\|}_{{{L}^{2}}\left( {{X}_{\#}} \right)}},\]
for all $\gamma \in {{L}^{2}}\left( {{X}_{\#}} \right)$.
\end{proposition}

\begin{proof} It suffices to show the statement for the unitarily equivalent operator (which we shall still denote by ${{L}_{\left( A_{R}^{app},\Phi _{R}^{app} \right)}}$) acting on the space ${{\Omega }^{0}}\left( {{X}_{\#}},{{E}_{H}}\left( \mathfrak{h}^{\mathbb{C}} \right) \right)$ defined after conjugation by the map $\gamma \mapsto i\gamma $. From Lemma \ref{linearization} it follows for all $\gamma \in {{\Omega }^{0}}\left( {{X}_{\#}},{{E}_{H}}\left( \mathfrak{h}^{\mathbb{C}} \right) \right)$ that
\[\left\langle \left( {{L}_{\left( A_{R}^{app},\Phi _{R}^{app} \right)}}-L_{1,T}^{*}{{L}_{1,T}} \right)\gamma ,\gamma  \right\rangle ={{\left\| \left[ \Phi _{R}^{app},\gamma  \right] \right\|}^{2}}\ge 0.\]
Consequently, ${{L}_{\left( A_{R}^{app},\Phi _{R}^{app} \right)}}-L_{1,T}^{*}{{L}_{1,T}}$ is a nonnegative operator. Furthermore, from Corollary \ref{basic_corollary} we obtain
\[{{\left\| {{L}_{\left( A_{R}^{app},\Phi _{R}^{app} \right)}}\gamma  \right\|}_{{{L}^{2}}\left( {{X}_{\#}} \right)}}\ge {{\left\| L_{1,T}^{*}{{L}_{1,T}}\gamma  \right\|}_{{{L}^{2}}\left( {{X}_{\#}} \right)}}\ge C{{T}^{-2}}{{\left\| \gamma  \right\|}_{{{L}^{2}}\left( {{X}_{\#}} \right)}}.\]
Therefore, the operator ${{L}_{\left( A_{R}^{app},\Phi _{R}^{app} \right)}}$ is strictly positive, and so invertible, and the norm of its inverse is bounded above by the inverse of the smallest eigenvalue of ${{L}_{\left( A_{R}^{app},\Phi _{R}^{app} \right)}}$, thus providing the statement of the proposition.
\end{proof}

This upper bound for the inverse operator $G_{R}$ is valid also when $G_{R}$ is viewed as an operator ${{L}^{2}}\left( {{X}_{\#}},{{r}}drd\theta  \right)\to H_{B}^{2}\left( {{X}_{\#}},{{r}}drd\theta  \right)$:
\begin{proposition}\label{estimate_H2}
There exist constants ${{R}_{0}}>0$ and $C>0$, such that for all sufficiently small $0<R<{{R}_{0}}$ there holds the estimate 
\[{{\left\| {{G}_{R}}\gamma  \right\|}_{H_{B}^{2}\left( {{X}_{\#}} \right)}}\le C{{\left| \log R \right|}^{2}}{{\left\| \gamma  \right\|}_{{{L}^{2}}\left( {{X}_{\#}} \right)}},\]
for all $\gamma \in {{L}^{2}}\left( {{X}_{\#}} \right)$.
\end{proposition}

\begin{proof}
The proof of this statement readily adapts from the proof of Proposition 3.14 and Corollary 3.15 in \cite{Swoboda}; we refer the interested reader to this article for details.
\end{proof}

\subsection{Lipschitz constants for ${{Q}_{R}}$}

The orbit map for any Higgs pair $\left( A,\Phi  \right)$ and any $g=\exp \left( \gamma  \right)$ with $\gamma \in {{\Omega }^{0}}\left( {{X}_{\#}},{{E}_{H}}\left( {{\mathfrak{h}}^{\mathbb{C}}} \right) \right)$ is given by
	\[{{\mathsf{\mathcal{O}}}_{\left( A,\Phi  \right)}}\left( \gamma  \right)={{g}^{*}}\left( A,\Phi  \right)=\left( A+{{g}^{-1}}\left( {{{\bar{\partial }}}_{A}}g \right)-\left( {{\partial }_{A}}g \right){{g}^{-1}},{{g}^{-1}}\Phi g \right),\]
thus
\begin{align*}
\exp {{\left( \gamma  \right)}^{*}}A & = A+\left( {{{\bar{\partial }}}_{A}}-{{\partial }_{A}} \right)\gamma +{{R}_{A}}\left( \gamma  \right)\\
\exp \left( -\gamma  \right)\Phi \exp \left( \gamma  \right) & = \Phi +\left[ \Phi ,\gamma  \right]+{{R}_{\Phi }}\left( \gamma  \right),
\end{align*}
where these reminder terms are
$${{R}_{A}}\left( \gamma  \right)=\exp \left( -\gamma  \right)\left( {{{\bar{\partial }}}_{A}}\exp \left( \gamma  \right) \right)-\left( {{\partial }_{A}}\exp \left( \gamma  \right) \right)\exp \left( -\gamma  \right)-\left( {{{\bar{\partial }}}_{A}}-{{\partial }_{A}} \right)\gamma $$
$${{R}_{\Phi }}\left( \gamma  \right)=\exp \left( -\gamma  \right)\Phi \exp \left( \gamma  \right)-\left[ \Phi ,\gamma  \right]-\Phi. $$
The Taylor series expansion of the operator ${{\mathsf{\mathcal{F}}}_{R}}$ is then
 \[{{\mathsf{\mathcal{F}}}_{R}}\left( \exp \left( \gamma  \right) \right)=\text{p}{{\text{r}}_{1}}\left( {{\mathsf{\mathcal{H}}}_{R}}\left( A,\Phi  \right) \right)+{{L}_{R}}\gamma +{{Q}_{R}}\gamma, \]
 with
 \begin{align*}
 {{Q}_{R}}\left( \gamma  \right) & := {{d}_{A}}\left( {{R}_{A}}\left( \gamma  \right) \right)+\left[ {{\Phi }^{*}},{{R}_{\Phi }}\left( \gamma  \right) \right] + \left[ \Phi ,{{R}_{\Phi }}{{\left( \gamma  \right)}^{*}} \right]\\
 & + \frac{1}{2}\left[ \left( \left( {{{\bar{\partial }}}_{A}}-{{\partial }_{A}} \right)\gamma +{{R}_{A}}\left( \gamma  \right) \right),\left( \left( {{{\bar{\partial }}}_{A}}-{{\partial }_{A}} \right)\gamma +{{R}_{A}}\left( \gamma  \right) \right) \right]\\
 & + \left[ \left( \left[ \Phi ,\gamma  \right]+{{R}_{\Phi }}\left( \gamma  \right) \right),{{\left( \left[ \Phi ,\gamma  \right]+{{R}_{\Phi }}\left( \gamma  \right) \right)}^{*}} \right].
 \end{align*}

\begin{lemma}\label{Q_estimate} Consider the pair $\left( A_{R}^{app},\Phi _{R}^{app} \right)$ in place of $\left( A,\Phi  \right)$ above. Then there exists a constant $C>0$ such that
	\[{{\left\| {{Q}_{R}}\left( {{\gamma }_{1}} \right)-{{Q}_{R}}\left( {{\gamma }_{0}} \right) \right\|}_{{{L}^{2}}\left( {{X}_{\#}} \right)}}\le Cr{{\left\| {{\gamma }_{1}}-{{\gamma }_{0}} \right\|}_{H_{B}^{2}\left( {{X}_{\#}} \right)}},\]
for all $0<r\le 1$ and ${{\gamma }_{0}},{{\gamma }_{1}}\in {{B}_{r}}$, the closed ball of radius $r$ around 0 in $H_{B}^{2}\left( {{X}_{\#}} \right)$.
\end{lemma}

\begin{proof}
See \cite{Swoboda}, Lemma 4.1.
\end{proof}

\section{Gluing theorems}\label{hybrid_Higgs}

The necessary prerequisites are now in place in order to apply the contraction mapping argument described in
\S \ref{contraction_mapping} and correct the approximate solution constructed into an exact solution of the $\text{Sp(4}\text{,}\mathbb{R}\text{)}$-Hitchin equations.

\begin{theorem}\label{exact_solution}
There exists a constant $0<{{R}_{0}}<1$, and for every $0<R<{{R}_{0}}$ there exist a constant ${{\sigma }_{R}}>0$ and a unique section $\gamma \in H_{B}^{2}\left( {{X}_{\#}},\mathfrak{gl}\left( 2 \right) \right)$ satisfying ${{\left\| \gamma  \right\|}_{H_{B}^{2}\left( {{X}_{\#}} \right)}}\le {{\sigma }_{R}}$, so that, for $g=\exp \left( \gamma  \right)$,
	\[\left( {{A}_{\#}},{{\Phi }_{\#}} \right)={{g}^{*}}\left( A_{R}^{app},\Phi _{R}^{app} \right)\]
is an exact solution of the $\mathrm{Sp(4}\text{,}\mathbb{R}\text{)}$-Hitchin equations over the closed surface ${{X}_{\#}}$.
\end{theorem}

\begin{proof}
We show that for ${{\sigma }_{R}}>0$ sufficiently small, the operator $T$ from \S \ref{contraction_mapping} defined by $T\left( \gamma  \right)=-{{G}_{R}}\left( \mathsf{\mathcal{H}}\left( \left( A_{R}^{app},\Phi _{R}^{app} \right) \right)+{{Q}_{R}}\left( \gamma  \right) \right)$ is a contraction of ${{B}_{{{\sigma }_{R}}}}$, the open ball of radius ${{\sigma }_{R}}$. From Proposition \ref{estimate_H2} and Lemma \ref{Q_estimate} we get
 \begin{align*}{{\left\| T\left( {{\gamma }_{1}}-{{\gamma }_{0}} \right) \right\|}_{H_{B}^{2}\left( {{X}_{\#}} \right)}} & = {{\left\| {{G}_{R}}\left( {{Q}_{R}}\left( {{\gamma }_{1}} \right)-{{Q}_{R}}\left( {{\gamma }_{0}} \right) \right) \right\|}_{H_{B}^{2}\left( {{X}_{\#}} \right)}}\\
 & \le C{{\left( \log R \right)}^{2}}{{\left\| {{Q}_{R}}\left( {{\gamma }_{1}} \right)-{{Q}_{R}}\left( {{\gamma }_{0}} \right) \right\|}_{{{L}^{2}}\left( {{X}_{\#}} \right)}}\\
 & \le C{{\left( \log R \right)}^{2}}{{\sigma }_{R}}{{\left\| {{\gamma }_{1}}-{{\gamma }_{0}} \right\|}_{H_{B}^{2}\left( {{X}_{\#}} \right)}}.
 \end{align*}
 Let $\varepsilon >0$ and set ${{\sigma }_{R}}:={{C}^{-1}}{{\left| \log R \right|}^{-2-\varepsilon }}$. Then for all $0<R<{{e}^{-1}}$ it follows that $C{{\left( \log R \right)}^{2}}{{\sigma }_{R}}<1$ and therefore $T$ is a contraction on the ball of radius ${{\sigma }_{R}}$.\\
 Furthermore, since ${{Q}_{R}}\left( 0 \right)=0$, using again Proposition \ref{estimate_H2} and Lemma \ref{Q_estimate} we have
 \begin{align*}
 {{\left\| T\left( 0 \right) \right\|}_{H_{B}^{2}\left( {{X}_{\#}} \right)}} & = {{\left\| {{G}_{R}}\left( \text{p}{{\text{r}}_{1}}\left( {{\mathsf{\mathcal{H}}}_{R}}\left( A_{R}^{app},\Phi _{R}^{app} \right) \right) \right) \right\|}_{H_{B}^{2}\left( {{X}_{\#}} \right)}}\\
 & \le C{{\left( \log R \right)}^{2}}{{\left\| \text{p}{{\text{r}}_{1}}\left( {{\mathsf{\mathcal{H}}}_{R}}\left( A_{R}^{app},\Phi _{R}^{app} \right) \right) \right\|}_{{{L}^{2}}\left( {{X}_{\#}} \right)}}\\
 & \le C{{\left( \log R \right)}^{2}}{{R}^{{{\delta }''}}}.
 \end{align*}
 Thus, when ${{R}_{0}}$ is chosen to be sufficiently small, then ${{\left\| T\left( 0 \right) \right\|}_{H_{B}^{2}\left( {{X}_{\#}} \right)}}<\frac{1}{10}{{\sigma }_{R}}$, for all $0<R<{{R}_{0}}$ and for the above choice of ${{\sigma }_{R}}$; thus the ball ${{B}_{{{\sigma }_{R}}}}$ is mapped to itself by $T$.
 \end{proof}

\begin{remark}The analytic arguments developed in the preceding sections provide also that the Main Theorem 1.1 in \cite{Swoboda} also holds for solutions to the $\text{Sp(4}\text{,}\mathbb{R}\text{)}$-Hitchin equations. In particular, we have the following:
\end{remark}
\begin{corollary}
Let $\left( \Sigma ,{{J}_{0}} \right)$ be a Riemann surface with nodes at a finite collection of points $D \subset \Sigma $. Let $\left( {{A}_{0}},{{\Phi }_{0}} \right)$ be a solution to the $\mathrm{Sp(4}\text{,}\mathbb{R}\text{)}$-Hitchin equations with logarithmic singularities at $D$, which is obtained from a solution to the $\mathrm{SL(2}\text{,}\mathbb{R}\text{)}$-Hitchin equations via an embedding $\rho :\mathrm{SL(2}\text{,}\mathbb{R}\text{)}\hookrightarrow \mathrm{Sp(4}\text{,}\mathbb{R}\text{)}$ that maps a copy of a maximal compact subgroup of $\mathrm{SL(2}\text{,}\mathbb{R}\text{)}$ into a maximal compact subgroup of $\mathrm{Sp(4}\text{,}\mathbb{R}\text{)}$. Suppose that there is a model solution near those nodes which is of the form described in \S \ref{section_local_model}. Let $\left( \Sigma ,{{J}_{i}} \right)$ be a sequence of smooth Riemann surfaces converging uniformly to $\left( \Sigma ,{{J}_{0}} \right)$. Then, for every sufficiently large $i\in \mathbb{N}$, there exists a smooth solution $\left( {{A}_{i}},{{\Phi }_{i}} \right)$ on $\left( \Sigma ,{{J}_{i}} \right)$, such that $\left( {{A}_{i}},{{\Phi }_{i}} \right)\to \left( {{A}_{0}},{{\Phi }_{0}} \right)$ as $i\to \infty $, uniformly on compact subsets of $\Sigma \backslash D$.
\end{corollary}

Theorem \ref{exact_solution} now implies that for $\bar{\partial }:=A_{\#}^{0,1}$, the Higgs bundle $\left( {{E}_{\#}}:=\left( {{\mathbb{E}}_{\#}},\bar{\partial } \right),{{\Phi }_{\#}} \right)$ is a polystable $\text{Sp(4}\text{,}\mathbb{R}\text{)}$-Higgs bundle over the complex connected sum ${{X}_{\#}}$. Collecting the steps from the last three sections 4,5,6 we now have our main result:

\begin{theorem}\label{main_theorem}
Let $X_{1}$ be a closed Riemann surface of genus $g_{1}$ and ${{D}_{1}}=\left\{ {{p}_{1}},\ldots ,{{p}_{s}} \right\}$ be a collection of $s$-many distinct points on $X_{1}$. Consider respectively a closed Riemann surface $X_{2}$ of genus $g_{2}$ and a collection of also $s$-many distinct points ${{D}_{2}}=\left\{ {{q}_{1}},\ldots ,{{q}_{s}} \right\}$ on $X_{2}$. Let $\left( {{E}_{1}},{{\Phi }_{1}} \right)\to {{X}_{1}}$ and $\left( {{E}_{2}},{{\Phi }_{2}} \right)\to {{X}_{2}}$ be parabolic polystable $\mathrm{Sp(4}\text{,}\mathbb{R}\text{)}$-Higgs bundles with corresponding solutions to the Hitchin equations $\left( {{A}_{1}},{{\Phi }_{1}} \right)$ and $\left( {{A}_{2}},{{\Phi }_{2}} \right)$. Assume that these solutions agree with model solutions $\left( A_{1,{{p}_{i}}}^{\bmod },\Phi _{1,{{p}_{i}}}^{\bmod } \right)$ and $\left( A_{2,{{q}_{j}}}^{\bmod },\Phi _{2,{{q}_{j}}}^{\bmod } \right)$ near the points ${{p}_{i}}\in {{D}_{1}}$ and ${{q}_{j}}\in {{D}_{2}}$, and that the model solutions satisfy $\left( A_{1,{{p}_{i}}}^{\bmod },\Phi _{1,{{p}_{i}}}^{\bmod } \right)=-\left( A_{2,{{q}_{j}}}^{\bmod },\Phi _{2,{{q}_{j}}}^{\bmod } \right)$, for $s$-many possible pairs of points $\left( {{p}_{i}},{{q}_{j}} \right)$. Then there is a polystable $\mathrm{Sp(4}\text{,}\mathbb{R}\text{)}$-Higgs bundle $\left( E_{\#},\Phi_{\#}  \right)\to {{X}_{\#}}$, constructed over the complex connected sum of Riemann surfaces ${{X}_{\#}}={{X}_{1}}\#{{X}_{2}}$, which agrees with the initial data over ${{X}_{\#}}\backslash {{X}_{1}}$ and ${{X}_{\#}}\backslash {{X}_{2}}$.
\end{theorem}

\begin{remark}
In \S \ref{pairs_model_represenations} we checked that for the particular parabolic $\text{Sp(4}\text{,}\mathbb{R}\text{)}$-Higgs bundles arising from the representations ${{\phi }_{irr}}$ and $\psi $, the main assumption in the theorem does apply.
\end{remark}

\begin{definition}
We call an $\text{Sp(4}\text{,}\mathbb{R}\text{)}$-Higgs bundle constructed by the procedure developed in \S 4-7 a \emph{hybrid $\mathrm{Sp(4}\text{,}\mathbb{R}\text{)}$-Higgs bundle}.
\end{definition}

\section{Topological invariants}

In this final section, we identify the connected component of the moduli space ${{\mathsf{\mathcal{M}}}^{\max }}$ a hybrid Higgs bundle lies, given a choice of stable parabolic ingredients to glue. For this, we need to look at how the Higgs bundle topological invariants behave under the complex connected sum operation. As an application, we see that under the right initial choices for the gluing data, we can find model Higgs bundles in the exceptional components of the maximal $\text{Sp(4}\text{,}\mathbb{R}\text{)}$-Higgs bundle moduli space; these models are described by the hybrid Higgs bundles of \S \ref{hybrid_Higgs}.

\subsection{Degree of a connected sum bundle}

In \cite{BiquardGM} a general notion of parabolic degree was introduced for parabolic principal $H^{\mathbb{C}}$-bundles $E$ equipped with a parabolic structure $\alpha$. This degree is defined using Chern-Weil theory in terms of a holomorphic reduction $\sigma$ of the structure group of $E$ from $H^{\mathbb{C}}$ to a parabolic subgroup $P$, and any antidominant character $\chi$ of $\mathfrak{p}:=\text{Lie}\left( P \right)$.  The next proposition describes an additivity property for this parabolic degree over the complex connected sum; the reader is referred to \S 2 and Appendix B of \cite{BiquardGM} for the precise definitions in this principal $H^{\mathbb{C}}$-bundle setting; we next introduce some notation from this article:

Let $H^{\mathbb{C}}$ be a reductive complex Lie group. For any element $s \in \sqrt{-1} \mathfrak{h}$, where $\mathfrak{h} = \text{Lie} \left(H \right)$, let $P = P_{s} \subset H^{\mathbb{C}}$ be the corresponding parabolic subgroup and let $\chi: \mathfrak{p} \to \mathbb{C}$ be the antidominant character defined by $\chi \left( \alpha \right) = \left\langle \alpha, s \right\rangle$, where $\left\langle \cdot ,\cdot  \right\rangle :{{\mathfrak{h}}^{\mathbb{C}}}\times {{\mathfrak{h}}^{\mathbb{C}}}\to \mathbb{C}$ is the extension of an invariant scalar product on $\mathfrak{h}$ to a Hermitian pairing. Let also ${{E}_{\sigma }}$ denote the corresponding $P$-principal bundle to a holomorphic reduction $\sigma$ of the structure group of $E$ from ${{H}^{\mathbb{C}}}$ to $P$.  For $N\subset P$ the unipotent part of $P$, choose the Levi subgroup $L={{Z}_{{{H}^{\mathbb{C}}}}}\left( s \right)\subset P$. Then, there is a well-defined $L$-action on ${{{E}_{\sigma }}}/{N}\;$ which turns it into a principal $L$-bundle, which we denote by ${{E}_{\sigma ,L}}$. Moreover, since $L={{Z}_{{{H}^{\mathbb{C}}}}}\left( s \right)$, there is a section
	\[{{s}_{\sigma }}\in \Gamma \left( {{E}_{\sigma ,L}}\left( \mathfrak{l}\cap \sqrt{-1}\mathfrak{h} \right) \right),\] 
canonically defined by $s$, where $\mathfrak{l}=\text{Lie}\left( L \right)$. Considering now a smooth metric $h$ on $E$, we get from both $h$ and $\sigma$, a reduction of structure group of $E$ from  ${{H}^{\mathbb{C}}}$ to $P\cap H$. Denote this bundle by ${{E}_{\sigma ,L}}$, which comes equipped with the metric ${{h}_{\sigma ,L}}$ induced by $h$. Lastly, let ${{F}_{h,L}}$ be the curvature for the Chern connection on ${{E}_{\sigma ,L}}$ associated to ${{h}_{\sigma ,L}}$.

Let $X_{1}$ and $X_{2}$ be closed Riemann surfaces with divisors $D_{1}$ and $D_{2}$ of $s$-many distinct points on each, and let ${{V}_{1}},{{V}_{2}}$ be two parabolic principal ${{H}^{\mathbb{C}}}$-bundles with parabolic structures $\alpha_{1}$, $\alpha_{2}$ over $X_{1}$, $X_{2}$ respectively. Assume that the underlying smooth bundles ${{\mathbb{V}}_{1}},{{\mathbb{V}}_{2}}$ come equipped with adapted hermitian metrics ${{h}_{1}},{{h}_{2}}$. Let $\left( {{\mathbb{V}}_{1}}\#{{\mathbb{V}}_{2}},{{h}_{\#}} \right)$ be the smooth hermitian bundle over the complex connected sum ${{X}_{\#}}$ of $X_{1}$ and $X_{2}$. The hermitian metric ${{h}_{\#}}$ coincides with ${{h}_{1}}$ and ${{h}_{2}}$ in a neighborhood of ${{X}_{1}}\backslash \Omega $ and ${{X}_{2}}\backslash \Omega $ respectively, where $\Omega$ is the neck region in the connected sum construction. We have the following:

\begin{proposition}\label{additivity_pardeg}
Let ${{X}_{\#}}=X_{1}\#X_{2}$ be the complex connected sum of two closed Riemann surfaces $X_{1}$ and $X_{2}$ with divisors $D_{1}$ and $D_{2}$ of $s$-many distinct points on each surface, and let ${{V}_{1}},{{V}_{2}}$ be parabolic principal ${{H}^{\mathbb{C}}}$-bundles over $X_{1}$ and $X_{2}$ respectively. Fix an antidominant character $\chi$ of $\mathfrak{p}$ and let $\sigma_{1}$, $\sigma_{2}$ be holomorphic reductions of the structure group of $V_{1}$, $V_{2}$ respectively from $H^{\mathbb{C}}$ to $P$. Assuming that the parabolic bundles $V_{1}$ and $V_{2}$ are glued to a bundle 
${{V}_{1}}\#{{V}_{2}}$, denote by $\sigma_{\#}$ the holomorphic reduction of the structure group of ${{V}_{1}}\#{{V}_{2}} $  from ${{H}^{\mathbb{C}}}$ to $P$ induced by $\sigma_{1}$ and $\sigma_{2}$. Then, the following identity holds:
\[\deg \left( {{V}_{1}}\#{{V}_{2}} \right)\left( \sigma_{\#} ,\chi  \right)=par{{\deg }_{{{\alpha }_{1}}}}\left( {{V}_{1}} \right)\left( \sigma_{1} ,\chi  \right)+par{{\deg }_{{{\alpha }_{2}}}}\left( {{V}_{2}} \right)\left( \sigma_{2} ,\chi  \right).\]
\end{proposition}

\begin{proof}
Consider smooth metrics ${{\hbar }_{1}},{{\hbar }_{2}}$ on the principal ${{H}^{\mathbb{C}}}$-bundles ${{V}_{1}},{{V}_{2}}$ defined over $X_{1}$ and $X_{2}$, which coincide with the adapted metrics ${{h}_{1}},{{h}_{2}}$ on ${{X}_{1}}\backslash {{D}_{1}}$, ${{X}_{2}}\backslash {{D}_{2}}$ respectively.\\
For $v>0$, let ${{X}_{i,v}}:=\left\{ x\in {{X}_{{i}}}\left| d\left( x,D \right)\ge {{e}^{-v}} \right. \right\}$ and ${{B}_{i,v}}:={{X}_{{i}}}\backslash {{X}_{i,v}}$, for $i=1,2$. For holomorphic reductions $\sigma_{i} $ and an antidominant character $\chi $, the metrics ${{\hbar }_{i}},{{h}_{i}}$ induce metrics ${{\hbar }_{i,L}},{{h}_{i,L}}$ on ${{\left( {{V}_{i}} \right)}_{\sigma ,L}}$ with curvature ${{F}_{{{h}_{i}},L}}$ and ${{F}_{{{\hbar }_{i}},L}}$ respectively. Similarly, the smooth metric ${{h}_{\#}}$ on ${{V}_{1}}\#{{V}_{2}}$ induces a metric ${{h}_{\#,L}}$ on ${{\left( {{V}_{1}}\#{{V}_{2}} \right)}_{\sigma ,L}}$ with curvature ${{F}_{{{h}_{\#}},L}}$. We now have:
\begin{align*}
& \deg \left( {{V}_{1}}\#{{V}_{2}} \right)\left( \sigma ,\chi  \right)=\frac{\sqrt{-1}}{2\pi }\int\limits_{{{X}_{\#}}}{\left\langle {{F}_{{{h}_{\#}},L}},{{s}_{\sigma }} \right\rangle }\\
& =\frac{\sqrt{-1}}{2\pi }\int\limits_{{{X}_{1,v}}}{\left\langle {{F}_{{{h}_{1}},L}},{{s}_{\sigma }} \right\rangle }+\frac{\sqrt{-1}}{2\pi }\int\limits_{{{X}_{2,v}}}{\left\langle {{F}_{{{h}_{2}},L}},{{s}_{\sigma }} \right\rangle }+\frac{\sqrt{-1}}{2\pi }\int\limits_{{{X}_{\#}}\backslash \left( {{X}_{1,v}}\cup {{X}_{2,v}} \right)}{\left\langle {{F}_{{{h}_{\#}},L}},{{s}_{\sigma }} \right\rangle }.
\end{align*}
Now notice:
  \[\frac{\sqrt{-1}}{2\pi }\int\limits_{{{X}_{1,v}}}{\left\langle {{F}_{{{h}_{1}},L}},{{s}_{\sigma }} \right\rangle }=\frac{\sqrt{-1}}{2\pi }\int\limits_{{{X}_{{1}}}}{\left\langle {{F}_{{{\hbar }_{1}},L}},{{s}_{\sigma }} \right\rangle }-\frac{\sqrt{-1}}{2\pi }\int\limits_{{{B}_{1,v}}}{\left\langle {{F}_{{{h}_{1}},L}},{{s}_{\sigma }} \right\rangle }\]
and
  \[\frac{\sqrt{-1}}{2\pi }\int\limits_{{{X}_{{1}}}}{\left\langle {{F}_{{{\hbar }_{1}},L}},{{s}_{\sigma }} \right\rangle }=\deg \left( {{\mathbb{V}}_{1}} \right)\left( \sigma ,\chi  \right);\]
similarly for the integral over ${{X}_{2,v}}$.
Therefore, for every $v>0$,
\begin{align*}
\deg \left( {{V}_{1}}\#{{V}_{2}} \right)\left( \sigma ,\chi  \right) & = \deg \left( {{V}_{1}} \right)\left( \sigma ,\chi  \right) - \frac{\sqrt{-1}}{2\pi }\int\limits_{{{B}_{1,v}}}{\left\langle {{F}_{{{h}_{1}},L}},{{s}_{\sigma }} \right\rangle } + \deg \left( {{V}_{2}} \right)\left( \sigma ,\chi  \right) \\
& - \frac{\sqrt{-1}}{2\pi }\int\limits_{{{B}_{2,v}}}{\left\langle {{F}_{{{h}_{2}},L}},{{s}_{\sigma }} \right\rangle } + \frac{\sqrt{-1}}{2\pi }\int\limits_{{{X}_{\#}}\backslash \left( {{X}_{1,v}}\cup {{X}_{2,v}} \right)}{\left\langle {{F}_{{{h}_{\#}},L}},{{s}_{\sigma }} \right\rangle }.
\end{align*}
Passing to the limit as $v\to +\infty $, the last integral vanishes, while each integral over ${{B}_{i,v}}$ for $i=1,2$ converges to the local term measuring the contribution of the parabolic structure in the definition of the parabolic degree; the latter is the content of Equation (2.10) established within the proof of Lemma 2.13 in \cite{BiquardGM}. The desired identity now follows.
\end{proof}

Proposition \ref{additivity_pardeg} implies in particular that the complex connected sum of \emph{maximal} parabolic $\text{Sp(4}\text{,}\mathbb{R}\text{)}$-Higgs bundles is a \emph{maximal} (non-parabolic) $\text{Sp(4}\text{,}\mathbb{R}\text{)}$-Higgs bundle. For a parabolic principal $\text{Sp(4}\text{,}\mathbb{R}\text{)}$-bundle the general situation of Proposition \ref{additivity_pardeg} describes an additivity property for the parabolic degree of the underlying vector bundle in the data as in Definition \ref{par_toledo_Sp4} (cf. \S \ref{section_parabolic_Sp4} of the present article and Example A.25 in \cite{KSZ}). This is the analogue in the language of Higgs bundles of the additivity property for the Toledo invariant from the point of view of fundamental group representations (Proposition \ref{additivity_Toledo_rep}).

\subsection{Model maximal parabolic $\text{Sp}\left( 4,\mathbb{R} \right)$-Higgs bundles.}\label{Examples_of_maximal}

In the sequel, we describe the construction of specific parabolic $\text{Sp(4}\text{,}\mathbb{R}\text{)}$-Higgs bundle models, which will be used in order to provide model $\text{Sp(4}\text{,}\mathbb{R}\text{)}$-Higgs bundles lying in all the $2g-3$ exceptional components of ${{\mathsf{\mathcal{M}}}^{\max }}$. These are obtained by extending the parabolic $\text{SL(2}\text{,}\mathbb{R}\text{)}$-data from \S \ref{section_parabolic_SL2} using certain embeddings of $\text{SL(2}\text{,}\mathbb{R}\text{)}$ into $\text{Sp(4}\text{,}\mathbb{R}\text{)}$. Note that this idea was also used in \cite{BGGDeformations} in the non-parabolic case.

\begin{example}\label{example_Hitchin}
Let $\left( E_{1}^{\left( 2 \right)},\Phi _{1}^{\left( 2 \right)} \right)$ be a parabolic $\text{SL}\left( 2,\mathbb{R} \right)$-Higgs bundle over the pair $\left( {{X}_{1}},{{D}_{1}} \right)$ of a Riemann surface and a divisor of $s$-many distinct points as in \S \ref{section_parabolic_SL2}, that is, 
	\[E_{1}^{\left( 2 \right)}={{\left( {{L}_{1}}\otimes {{\iota }_{1}} \right)}^{*}}\oplus {{L}_{1}},\]
for ${{L}_{1}}\to {{X}_{1}}$ a line bundle with $L_{1}^{2}\cong {{K}_{{{X}_{1}}}}$, ${{\iota }_{1}}\cong {{\mathsf{\mathcal{O}}}_{{{X}_{1}}}}\left( {{D}_{1}} \right)$ and assuming that the bundle ${{L}_{1}}$ is equipped with the trivial flag ${{\left( {{L}_{1}} \right)}_{{{x}_{i}}}}\supset \left\{ 0 \right\}$ and weight $\frac{1}{2}$ for every point ${{x}_{i}}\in {{D}_{1}}$. Moreover, the Higgs field is of the form 
	\[\Phi _{1}^{\left( 2 \right)}\in {{H}^{0}}\left( {{X}_{1}},\text{End}\left( E_{1}^{\left( 2 \right)} \right)\otimes {{K}_{{{X}_{1}}}}\otimes {{\iota }_{1}} \right).\]
We use the embedding ${{\phi }_{irr}}$ from (\ref{irreducible_rep}) to extend this pair to parabolic $\text{Sp}\left( 4,\mathbb{R} \right)$-data as follows. Under an appropriate change of basis for $\text{Sy}{{\text{m}}^{3}}{{\mathbb{R}}^{2}}\otimes \mathbb{C}$ determined by a transformation matrix, say $S$, we may write the matrix for the irreducible representation with respect to the new basis as
	\[{{\tilde{\phi }}_{irr}}\left( A \right)=\left( \begin{matrix}
   {{\nu }^{3}} & 0 & 0 & 0  \\
   0 & {{\nu }^{-1}} & 0 & 0  \\
   0 & 0 & {{\nu }^{-3}} & 0  \\
   0 & 0 & 0 & \nu   \\
\end{matrix} \right),\]
with $\nu =a+ic$ for a matrix $A=\left( \begin{matrix}
   a & -c  \\
   c & a  \\
\end{matrix} \right)\in \text{SO}\left( 2 \right)$; indeed, such a basis can be chosen so that ${{\tilde{\phi }}_{irr}}={{S}^{-1}}{{\phi }_{irr}}S$ and $J={{S}^{T}}JS$ for the antisymmetric matrix $J=\left( \begin{matrix}
   0 & \text{Id}  \\
   -\text{Id} & 0  \\
\end{matrix} \right)$ giving the symplectic form (cf. \S 8.1 in \cite{BGGDeformations} for a more detailed description).\\
The embedding of the parabolic  $\text{SL}\left( 2,\mathbb{R} \right)$-data $\left( E_{1}^{\left( 2 \right)},\Phi _{1}^{\left( 2 \right)} \right)$ is now obtained by applying ${{\phi }_{irr}}$ to ${{T}^{-1}}\left( \begin{matrix}
   {{l}_{ij}} & 0  \\
   0 & l_{ij}^{-1}  \\
\end{matrix} \right)T$ and ${{T}^{-1}}\left( \begin{matrix}
   0 & {{{\tilde{\beta }}}_{i}}  \\
   {{{\tilde{\gamma }}}_{i}} & 0  \\
\end{matrix} \right)T$ for the bundle and the Higgs field respectively, where $T=\frac{1}{2}\left( \begin{matrix}
   1 & i  \\
   1 & -i  \\
\end{matrix} \right)$, $\left\{ {{l}_{ij}} \right\}$ is a cocycle defining ${{L}_{1}}$, and ${{\tilde{\beta }}_{i}}$, ${{\tilde{\gamma }}_{i}}$ are the local descriptions of the Higgs field 
$\left( \tilde{\beta },\tilde{\gamma } \right)$ in the bundle map 
$\Phi _{1}^{\left( 2 \right)}=\left( \begin{matrix}
   0 & {\tilde{\beta }}  \\
   {\tilde{\gamma }} & 0  \\
\end{matrix} \right):E_{1}^{\left( 2 \right)}\to E_{1}^{\left( 2 \right)}\otimes {{K}_{{{X}_{1}}}}\otimes {{\iota }_{1}}$. \\
The Higgs bundle so obtained is a parabolic  $\text{SL}\left( 4,\mathbb{C} \right)$-Higgs bundle which is compatible with the symplectic form defined by $J$, this means, it is of the form 
	\[\left( {{V}_{1}}\oplus V_{1}^{\vee },{{\phi }_{irr,*}}\left| _{{{\mathfrak{m}}^{\mathbb{C}}}\left( \text{SL}\left( 2,\mathbb{C} \right) \right)} \right.\left( \Phi _{1}^{\left( 2 \right)} \right) \right),\]
where $V_{1}^{\vee }$ denotes the parabolic dual of ${{V}_{1}}$, which is then given by 
\[{{V}_{1}}=\left( L^{3}_{1}\otimes \iota_{1}  \right)\oplus {{\left( L_{1}\otimes \iota_{1}  \right)}^{*}}\]
and it comes equipped with a parabolic structure defined by a trivial flag ${{\left( {{V}_{1}} \right)}_{{{x}_{i}}}}\supset \left\{ 0 \right\}$ and weight $\frac{1}{2}$ for every ${{x}_{i}}\in D$.\\
Moreover, ${{V}_{1}}$ can be expressed as ${{V}_{1}}={{N}_{1}}\oplus N_{1}^{\vee} \otimes K_{X_{1}} \otimes \iota_{1}$, where $ N_{1}^{\vee}$ denotes the parabolic dual of  $N_{1}$. Indeed, for ${{N}_{1}}=L_{1}^{3}\otimes \iota_{1} $ we see that
	\[N_{1}^{\vee}\otimes K_{X_{1}} \otimes \iota_{1}={{\left( L_{1}^{3}\otimes \iota_{1}  \right)}^{*}}\otimes {{\iota_{1} }^{*}} \otimes K_{X_{1}} \otimes \iota_{1}=L_{1}^{-3}\otimes {{\iota_{1} }^{*}} \otimes L_{1}^{2}={{\left( L_{1}\otimes \iota_{1}  \right)}^{*}}.\]
It can be checked that $\left( {{V}_{1}},{{\beta }_{1}},{{\gamma }_{1}} \right)$  is a parabolic \emph{stable} $\text{Sp(4}\text{,}\mathbb{R}\text{)}$-Higgs bundle.
Also notice that
\begin{align*}
par\deg {{V}_{1}} & = par\deg \left( L_{1}^{3}\otimes \iota_{1}  \right)+par\deg {{\left( L_{1}\otimes \iota_{1}  \right)}^{*}} \\
	& = 3g-3+s+\frac{s}{2}+1-g-s+\frac{s}{2}=2g-2+s.
\end{align*}
Therefore, the triple $\left( {{V}_{1}},{{\beta }_{1}},{{\gamma }_{1}} \right)$ defines a \emph{model maximal parabolic $\mathrm{Sp(4}\text{,}\mathbb{R}\text{)}$-Higgs bundle}.
\end{example}

\begin{example}\label{example_diagonal}
In a similar way we may construct a triple $\left( {{V}_{2}},{{\beta }_{2}},{{\gamma }_{2}} \right)$ defining a maximal parabolic $\text{Sp(4}\text{,}\mathbb{R}\text{)}$-Higgs bundle over the pair $\left( {{X}_{2}},{{D}_{2}} \right)$, induced by the diagonal embedding ${{\phi }_{\Delta }}$ described in \S \ref{maximal_ representations} of a model parabolic $\text{SL(2}\text{,}\mathbb{R}\text{)}$-Higgs bundle $\left( E_{2}^{\left( 2 \right)},\Phi _{2}^{\left( 2 \right)} \right)$. We thus may obtain a triple $\left( {{V}_{2}},{{\beta }_{2}},{{\gamma }_{2}} \right)$, with
\[{{V}_{2}}=L_{2}\oplus L_{2},\]
for ${{L}_{2}}\to {{X}_{2}}$ a line bundle with $L_{2}^{2}\cong {{K}_{{{X}_{2}}}}$ that comes equipped with a parabolic structure defined by a trivial flag ${{\left( {{L}_{2}} \right)}_{{{y}_{i}}}}\supset \left\{ 0 \right\}$ and weight $\frac{1}{2}$ for every ${{y}_{i}}\in D_{2}$.\\
Moreover, ${{V}_{2}}$ can be expressed as ${{V}_{2}}={{N}_{2}}\oplus N_{2}^{\vee} \otimes K_{X_{2}} \otimes \iota_{2}$, for ${{\iota }_{2}}\cong {{\mathsf{\mathcal{O}}}_{{{X}_{2}}}}\left( {{D}_{2}} \right)$ . Indeed, for ${{N}_{2}}=L_{2}$ we see that
	\[N_{2}^{\vee} \otimes K_{X_{2}} \otimes \iota_{2}=L_{2}^{-1}\otimes {{\iota_{2} }^{*}} \otimes K_{X_{2}} \otimes \iota_{2} =L_{2}.\]
It can be checked that $\left( {{V}_{2}},{{\beta }_{2}},{{\gamma }_{2}} \right)$ is a parabolic \emph{stable} $\text{Sp(4}\text{,}\mathbb{R}\text{)}$-Higgs bundle.
Also notice that \[par\deg {{V}_{2}}=2par\deg L=2\left( g-1+\frac{s}{2} \right)=2g-2+s.\]
Therefore, the triple $\left( {{V}_{2}},{{\beta }_{2}},{{\gamma }_{2}} \right)$ defines a \emph{model maximal parabolic $\mathrm{Sp(4}\text{,}\mathbb{R}\text{)}$-Higgs bundle}.
\end{example}

An application of Proposition \ref{additivity_pardeg} now provides directly  that the polystable hybrid $\text{Sp(4}\text{,}\mathbb{R}\text{)}$-Higgs bundle constructed, $\left( {{V}_{\#}},\Phi_{\#}, {{h}_{\#}} \right)$, is \emph{maximal}:

\begin{proposition}
The hybrid Higgs bundle $\left( {{V}_{\#}},\Phi_{\#} , {{h}_{\#}} \right)$ constructed by gluing the maximal parabolic Higgs bundles defined by the triples $\left( {{V}_{1}},{{\beta }_{1}},{{\gamma }_{1}} \right)$ and $\left( {{V}_{2}},{{\beta }_{2}},{{\gamma }_{2}} \right)$ described above is maximal, that is, $\deg \left( {{V}_{\#}} \right)=2\left( {{g}_{1}}+{{g}_{2}}+s-1 \right)-2=2g-2$, where $g$ is the genus of the Riemann surface ${{X}_{\#}}$, the connected sum of the $s$-punctured Riemann surfaces $X_{1}$ and $X_{2}$.
\end{proposition}

\subsection{Model Higgs bundles in the exceptional components of ${{\mathsf{\mathcal{M}}}^{\max }}$}
We may now describe model $\text{Sp(4}\text{,}\mathbb{R}\text{)}$-Higgs bundles that exhaust the exceptional components of ${{\mathsf{\mathcal{M}}}^{\max }}$.  Take the parabolic bundle $V_{1}$ from Example \ref{example_Hitchin} and fix a square root $M_{1}$ of the canonical line bundle ${{K}_{{{{X}}_{1}}}}$. Now, define:
\begin{align*}
{{W}_{1}} & := V_{1}^{*}\otimes {{M}_{1}}={{\left[ \left( L_{1}^{3}\otimes \iota  \right)\oplus {{\left( {{L}_{1}}\otimes \iota  \right)}^{*}} \right]}^{*}}\otimes {{M}_{1}} \\
& \cong \left[ \left( {{L}_{1}}\otimes \iota  \right)\oplus \left( L_{1}^{-3}\otimes {{\iota }^{*}} \right) \right]\otimes {{M}_{1}} \cong \left( {{L}_{1}}\otimes {{M}_{1}}\otimes \iota  \right)\oplus {{\left( {{L}_{1}}\otimes {{M}_{1}}\otimes \iota  \right)}^{*}},
\end{align*}
in other words, $W_{1}$ is of the form $\mathsf{\mathcal{L}}\oplus {{\mathsf{\mathcal{L}}}^{*}}$ for $\mathsf{\mathcal{L}}:= \left( {{L}_{1}}\otimes {{M}_{1}}\otimes \iota  \right)$ and also the map $\gamma_{1} \otimes {{I}_{M_{1}^{*}}}:W_{1}^{*}\to {{W}_{1}}$ is an isomorphism, which comes from the fact that $\gamma_{1}$ is; this follows from the proof of the Milnor-Wood inequality in the parabolic case.\\
Similarly, take the parabolic bundle $V_{2}$ from Example \ref{example_diagonal} and fix a square root $M_{2}$ of the canonical line bundle ${{K}_{{{{X}}_{2}}}}$. Define:
\begin{align*}
{{W}_{2}} & :=V_{2}^{*}\otimes {{M}_{2}} = {{\left( {{L}_{2}}\oplus L_{2}^{*}{{K}_{{{X}_{2}}}} \right)}^{*}}\otimes {{M}_{2}}\simeq \left( {{L}_{2}}M_{2}^{-2}\oplus L_{2}^{*} \right)\otimes {{M}_{2}}\\
& \simeq \left( {{L}_{2}}\otimes M_{2}^{-1} \right)\oplus {{\left( {{L}_{2}}\otimes M_{2}^{-1} \right)}^{*}},
\end{align*}
in other words, $W_{2}$ is of the form $\mathsf{\mathcal{L}}\oplus {{\mathsf{\mathcal{L}}}^{*}}$ for $\mathsf{\mathcal{L}}:\cong {{L}_{2}}\otimes M_{2}^{-1}$.\\
In fact, it is possible to choose appropriate square roots of the canonical bundles allowing to identify the restrictions of $W_{1}$ and $W_{2}$ to the annuli, thus the same square root can be used on both sides of the connected sum operation.\\
A notion of Stiefel-Whitney class $w_{1}$ as an appropriate topological invariant for parabolic  $\text{Sp(4}\text{,}\mathbb{R}\text{)}$-Higgs bundles was defined in \cite{KSZ} in terms of the corresponding Higgs $V$-bundles. The Cayley partners $W_{1}$, $W_{2}$ of the parabolic bundles $V_{1}$, $V_{2}$ respectively have vanishing $w_{1}$. From the way the connected sum operation is carried out and the fact that $w_{1}$ only depends on the restriction of the bundles to the 1-skeleton of the underlying Riemann surfaces, the first Stiefel-Whitney class $w_{1} \left( W_{\#} \right) $ for the Cayley partner $W_{\#}$ of the connected sum bundle $V_{\#}$ will also vanish. Therefore, there is a decomposition 
\[{{V}_{\#}}={{N}_{\#}}\oplus N_{\#}^{*} \otimes {{K}_{{{X}_{\#}}}},\]
with ${{N}_{\#}}={{N}_{1}}\#{{N}_{2}}$. Moreover, this provides that the Cayley partner ${{W}_{\#}}$ of ${{V}_{\#}}$ decomposes as ${{W}_{\#}}={{L}_{\#}}\oplus L_{\#}^{-1}$ for some line bundle ${{L}_{\#}}$. We thus have established the following:

\begin{proposition}
The hybrid Higgs bundle $\left( {{V}_{\#}},{{\Phi }_{\#}} \right)$ constructed by gluing the maximal parabolic Higgs bundles $\left( {{V}_{1}},{{\beta }_{1}},{{\gamma }_{1}} \right)$ and $\left( {{V}_{2}},{{\beta }_{2}},{{\gamma }_{2}} \right)$ of \S \ref{Examples_of_maximal} is maximal with a corresponding Cayley partner ${{W}_{\#}}$ for which ${{w}_{1}}\left( {{W}_{\#}} \right)=0$ and ${{W}_{\#}}={{L}_{\#}}\oplus L_{\#}^{-1}$, for some line bundle ${{L}_{\#}}$ over ${{X}_{\#}}$.
\end{proposition}

\begin{remark}
Compare this result to Proposition 5.9 in \cite{GW}, where an analogous property for the Stiefel-Whitney classes of a hybrid representation was established.
\end{remark}

The degree of this line bundle ${{L}_{\#}}$ fully determines the connected component in which a hybrid Higgs bundle will lie:

\begin{proposition} For the line bundle ${{L}_{\#}}$ appearing in the decomposition ${{W}_{\#}}={{L}_{\#}}\oplus L_{\#}^{-1}$ of the Cayley partner, it is
	\[\deg \left( {{L}_{\#}} \right)=par\deg {{K}_{{{{X}}_{1}}}}\otimes {{\iota }_{1}},\]
where ${{\iota }_{1}}={{\mathsf{\mathcal{O}}}_{{{{X}}_{1}}}}\left( {{D}_{1}} \right)$.
\end{proposition}

\begin{proof}
The identity of Proposition \ref{additivity_pardeg} applies to provide the computation of the degree for the bundle ${{N}_{\#}}$ appearing in the decomposition ${{V}_{\#}}={{N}_{\#}}\oplus N_{\#}^{*}{{K}_{{{X}_{\#}}}}$:
\begin{align*}
   \deg \left( {{N}_{\#}} \right) & = par\deg \left( L_{1}^{3}\otimes {{\iota }_{1}} \right)+par\deg \left( {{L}_{2}} \right)\\
   & = 3\left( {{g}_{1}}-1 \right)+s+\frac{s}{2}+{{g}_{2}}-1+\frac{s}{2}\\
   & = g+2{{g}_{1}}-3+s,
\end{align*}
where $g:={{g}_{1}}+{{g}_{2}}+s-1$ is the genus of ${{X}_{\#}}$.\\
Considering ${{N}_{\#}}\otimes L_{0}^{-1}$ for some square root ${{L}_{0}}:=K_{\#}^{\frac{1}{2}}$ now gives
\begin{align*}
   \deg \left( {{N}_{\#}}\otimes L_{0}^{-1} \right) & = g+2{{g}_{1}}-3+s+1-g\\
   & = 2{{g}_{1}}+s-2\\
   & = -\chi \left( {{\Sigma }_{1}} \right) = par\deg {{K}_{{{{X}}_{1}}}}\otimes {{\iota }_{1}},
\end{align*}
where ${{\iota }_{1}}={{\mathsf{\mathcal{O}}}_{{{{X}}_{1}}}}\left( {{D}_{1}} \right)$.
\end{proof}

\vspace{2mm}
\noindent\textbf{Final overview}.\\
We have constructed a holomorphic vector bundle ${{V}_{\#}}\to {{X}_{\#}}$ with $\deg \left( {{V}_{\#}} \right)=2g-2$ and ${{V}_{\#}}={{N}_{\#}}\oplus N_{\#}^{*}{{K}_{{{X}_{\#}}}}$ with $\deg \left( {{N}_{\#}}\otimes L_{0}^{-1} \right) = 2{{g}_{1}}-2+s$, which is \emph{odd} (respectively \emph{even}) whenever $s$ is odd (resp. even). The contraction mapping argument developed in \S 5-7 provides a holomorphic structure with respect to which $ {{V}_{\#}}$ is a polystable $\text{Sp(}4,\mathbb{R}\text{)}$-Higgs bundle. The numerical information we already have for the topological invariants of ${{V}_{\#}}$ is preserved and it identifies the connected component of the maximal moduli space in which the tuple $\left( {{V}_{\#}},\Phi ,{{h}_{\#}} \right)$ will lie. We thus derive the following conclusions:

\begin{enumerate}
  \item The method described in Sections 4,5 and 6 can be more generally applied for producing model $G$-Higgs bundles for any semisimple Lie group $G$ using appropriate embeddings $\text{SL}\left( 2,\mathbb{C} \right) \hookrightarrow G$. In each particular case for the group $G$, however, the computations isolated in \S \ref{invertibility} need to be explicitly checked for the invertibility of the linearization operator. 
  \item In the case $G = \text{Sp}\left( 4,\mathbb{R} \right)$ in particular, the component in which a hybrid Higgs bundle lies depends on the genera and the number of points in the divisors of the initial Riemann surfaces $X_{1}$ and $X_{2}$ in the construction; there are no extra parameters arising from the deformation of stable parabolic data to model data near these points, or the perturbation argument to correct the approximate solution to an exact solution.\\
Since $1\le {{g}_{1}}\le {{g}_{1}}+{{g}_{2}}-1$, it follows that
	\[s\le \deg \left( {{N}_{\#}}\otimes L_{0}^{-\frac{1}{2}} \right)\le 2g-s-2,\]
with $s$ an integer between 1 and $g-1$. Therefore, the hybrid Higgs bundles constructed are modeling \emph{all} exceptional $2g-3$ connected components of ${{\mathsf{\mathcal{M}}}^{\max }}\left( X,\text{Sp(4}\text{,}\mathbb{R}\text{)} \right)$. These components are fully distinguished by the degree of the line bundle ${{L}_{\#}}$ for the hybrid Higgs bundle constructed by gluing.
  \item The gluing of two parabolic Higgs bundles of the same type as the model $\left( {{V}_{1}},{{\beta }_{1}},{{\gamma }_{1}} \right)$ from Example \ref{example_Hitchin} implies that $\deg \left( {{N}_{\#}} \right)=3g-3$. On the other hand, the gluing of two parabolic Higgs bundles of the same type as $\left( {{V}_{2}},{{\beta }_{2}},{{\gamma }_{2}} \right)$ from Example \ref{example_diagonal} implies that $\deg \left( {{N}_{\#}} \right)=g-1$, as expected.
  \item For a hybrid representation $\rho :{{\pi }_{1}}\left( \Sigma  \right)\to \text{Sp(4}\text{,}\mathbb{R}\text{)}$, O. Guichard and A. Wienhard in \cite{GW} defined an Euler class $e$  with values $e=-\chi \left( {{\Sigma }_{l}} \right)\left[ \Sigma  \right]\in {{H}^{2}}\left( {{T}^{1}}\Sigma ,\mathbb{Z} \right)$, where ${{T}^{1}}\Sigma $ is the unit tangent bundle of the surface $\Sigma ={{\Sigma }_{l}}{{\cup }_{\gamma }}{{\Sigma }_{r}}$ and ${{\Sigma }_{l}}$ is considered to be a surface of genus $1\le {{g}_{l}}\le g-1$ and one boundary component, thus its Euler characteristic $\chi \left( {{\Sigma }_{l}} \right)=2-2{{g}_{l}}-1=1-2{{g}_{l}}$ is an odd integer within $-2g+3$ and $-1$. On the other hand,  a relation between the Stiefel-Whitney classes for maximal $\text{Sp(4}\text{,}\mathbb{R}\text{)}$-Higgs bundles and the Stiefel-Whitney classes for $\text{Sp(4}\text{,}\mathbb{R}\text{)}$-representations, was described in \cite{GW}:  
\begin{proposition}[Proposition 19 in \cite{GW}] Let $\rho :{{\pi }_{1}}\left(\Sigma \right)\to \mathrm{Sp}\left( 2n,\mathbb{R} \right)$ be a maximal fundamental group representation for a closed topological surface $\Sigma$. Then, for any choice of a spin structure $v$, the following formulas in ${{H}^{i}}\left( {{T}^{1}}\Sigma,{{\mathbb{Z}}_{2}} \right)$, $i=1, 2$ hold 
\begin{align*}
   \text{s}{{\text{w}}_{1}}\left( \rho  \right) & = {{w}_{1}}\left( \rho ,v \right)+nv,  \\
   \text{s}{{\text{w}}_{2}}\left( \rho  \right) & ={{w}_{2}}\left( \rho ,v \right)+\text{s}{{\text{w}}_{1}}\left( \rho  \right)\smile v+\left( g-1 \right)\bmod 2,
\end{align*}
where $w_{1}$ and $w_{2}$ are the Higgs bundle invariants (Stiefel-Whitney classes of the corresponding Cayley partner) and $  \text{s}{{\text{w}}_{1}}$, $  \text{s}{{\text{w}}_{2}}$ are the topological invariants of the Anosov representation $\rho$.
\end{proposition}
In view of this proposition, we furthermore deduce that in the case of Riemann surfaces with $s=1$ point in the divisors, the degree $\deg \left( {{L}_{\#}} \right)$ of the underlying bundle ${{L}_{\#}}$ in the decomposition of the Cayley partner ${{W}_{\#}}={{L}_{\#}}\oplus L_{\#}^{-1}$ of a hybrid $\text{Sp(4}\text{,}\mathbb{R}\text{)}$-Higgs bundle \emph{is equal to} the Euler class $e$ for the hybrid representation. This provides a comparison between the invariants for maximal Higgs bundles and the topological invariants for Anosov representations constructed by O. Guichard and A. Wienhard, although these invariants live naturally in different cohomology groups.
\end{enumerate}

\bigskip

\noindent\small{\textsc{Max-Planck-Institut f\"{u}r Mathematik} \\
Vivatsgasse 7, 53111 Bonn, Germany}\\
\emph{E-mail address}:  \texttt{kydonakis@mpim-bonn.mpg.de}

\vspace{2mm}

\begin{thebibliography}{99}

\bibitem{BiBo}
O. Biquard and P. Boalch, Wild non-abelian Hodge theory on curves. \emph{Compos. Math.} \textbf{140} (2004), no. 1, 179-204.

\bibitem{BiquardGM}
O. Biquard, O.  Garc\'{i}a-Prada and I. Mundet i Riera, Parabolic Higgs bundles and representations of the fundamental group of a punctured surface into a real group.  \emph{Adv. Math.} \textbf{372} (2020), 107305.

\bibitem{BiswasAresGovin}
I. Biswas, P. Ar\'{e}s-Gastesi and S. Govindarajan, Parabolic Higgs bundles and Teichm\"{u}ller spaces for punctured surfaces. \emph{Trans. Amer. Math. Soc.} \textbf{349} (1997), no. 4, 1551-1560.

\bibitem{BoYo}
H. U. Boden and K. Yokogawa, Moduli spaces of parabolic Higgs bundles and $K\left( D \right)$ pairs over smooth curves: I. \emph{Int. J. Math.} \textbf{7} (1996), 573-598.

\bibitem{BGGDeformations}
S. B. Bradlow, O. Garc{\'i}a-Prada and P. B. Gothen, Deformations of maximal representations in $\text{Sp}\left( 4,\mathbb{R} \right)$. \emph{Q. J. Math.} \textbf{63} (2012), no. 4, 795-843.

\bibitem{BIW}
M. Burger, A. Iozzi and A. Wienhard, Surface group representations with maximal Toledo invariant. \emph{Ann. of Math. (2)} \textbf{172} (2010), no. 1, 517-566.

\bibitem{CLM}
S. Cappell, R. Lee and E. Miller, Self-adjoint elliptic operators and manifold decompositions. Part I: Low eigenmodes and stretching. \emph{Comm. Pure Appl. Math.} \textbf{49} (1996), 825-866.

\bibitem{Collier}
B. Collier, Maximal $\text{Sp(}4,\mathbb{R}\text{)}$ surface group representations, minimal immersions and cyclic surfaces. \emph{Geom. Dedicata} \textbf{180} (2016), 241-285.

\bibitem{Corlette}
K. Corlette, Flat $G$-bundles with canonical metrics. \emph{J. Differential Geom.} \textbf{28} (1988), 361-382.

\bibitem{Donaldson}
S. K. Donaldson, Twisted harmonic maps and the self-duality equations. \emph{Proc. London Math. Soc.} (3) \textbf{55} (1987), 127-131.

\bibitem{DonKron}
S. Donaldson and P. Kronheimer, The geometry of four-manifolds. \emph{Oxford Math. Monographs}, Oxford Science Publications, 1990.

\bibitem{Foscolo}
L. Foscolo, A gluing construction for periodic monopoles. \emph{Int. Math. Res. Not. IMRN} (2017), no. 24, 7504-7550.

\bibitem{Lauraetal}
L. Fredrickson, R. Mazzeo, J. Swoboda and H. Weiss, Asymptotic geometry of the moduli space of parabolic $\text{SL} \left(2, \mathbb{C} \right)$-Higgs bundles. arXiv:2001.03682.


\bibitem{GGMsymplectic}
O. Garc{\'i}a-Prada, P. B. Gothen and I. Mundet i Riera, Higgs bundles and surface group representations in the real symplectic group. \emph{J. Topol.} \textbf{6} (2013), no. 1, 64-118.

\bibitem{GGMHitchin-Kob}
O. Garc{\'i}a-Prada, P. B. Gothen and I. Mundet i Riera, The Hitchin-Kobayashi correspondence, Higgs pairs and surface group representations. arXiv:0909.4487.

\bibitem{GGmunoz}
O. Garc\'{i}a-Prada, P. B. Gothen and V. Mu\~{n}oz, Betti numbers of the moduli space of rank 3 parabolic Higgs bundles. \emph{ Mem. Amer. Math. Soc.} \textbf{187} (2007), no. 879, viii+80 pp.

\bibitem{Goldman}
W. M. Goldman, The symplectic nature of fundamental groups of surfaces. \emph{Adv. Math.} \textbf{54} (1984), no. 2, 200-225.

\bibitem{Gothen}
P. B. Gothen, Components of spaces of representations and stable triples. \emph{Topology} \textbf{40} (2001), no. 4, 823-850.

\bibitem{GW}
O. Guichard and A. Wienhard, Topological invariants of Anosov representations. \emph{J. Topol.} \textbf{3} (2010), no. 3, 578-642.

\bibitem{HaOt}
T. Hartnick and A. Ott, Bounded cohomology, Higgs bundles, and Milnor-Wood inequalities. arXiv: 1105.4323.

\bibitem{He}
S. He, A gluing theorem for the Kapustin-Witten equations with a Nahm pole. \emph{J. Topol.} \textbf{12} (2019), no. 3, 855-915.

\bibitem{Hit92}
N. J. Hitchin, Lie groups and Teichm{\"u}ller space. \emph{Topology} \textbf{31} (1992), no. 3, 449-473.

\bibitem{Hit87}
N. J. Hitchin, The self-duality equations on a Riemann surface. \emph{Proc. London Math. Soc.} (3) \textbf{55} (1987), 59-126.

\bibitem{Konno}
H. Konno, Construction of the moduli space of stable parabolic Higgs bundles on a Riemann surface. \emph{J. Math. Soc. Japan} \textbf{45} (1993), no. 2, 253-276.

\bibitem{Kydonakis} G. Kydonakis, Gluing constructions for Higgs bundles over a complex connected sum. Ph.D. thesis, University of Illinois at Urbana-Champaign (2018).

\bibitem{KSZ}
G. Kydonakis, H. Sun and L. Zhao, Topological invariants of parabolic $G$-Higgs bundles. \emph{Math. Z.} \textbf{297} (2021), no. 1-2, 585-632.

\bibitem{KSZ2}
G. Kydonakis, H. Sun and L. Zhao, The Beauville-Narasimhan-Ramanan correspondence for twisted Higgs $V$-bundles and components of parabolic $\mathrm{Sp}(2n,\mathbb{R})$-Higgs moduli spaces. \emph{Trans. Amer. Math. Soc.} \textbf{374} (2021), no. 6, 4023-4057.  

\bibitem{MSWW}
R. Mazzeo, J. Swoboda, H. Weiss and F. Witt, Ends of the moduli space of Higgs bundles. \emph{Duke Math. J.} \textbf{165} (2016), no. 12, 2227-2271.

\bibitem{Mochizuki}
T. Mochizuki, Kobayashi-Hitchin correspondence for tame harmonic bundles and an application. arXiv:math/0411300.

\bibitem{Mumford}
D. Mumford, Theta characteristics of an algebraic curve. \emph{Ann. Sci. {\'E}cole Norm. Sup.} \textbf{4} (1971), no. 4, 181-192.

\bibitem{Nicolaescuarticle}
L. Nicolaescu, On the Cappell-Lee-Miller gluing theorem. \emph{Pac. J. Math.} \textbf{206} (2002), no. 1, 159-185.

\bibitem{Ramanan}
S. Ramanan, Global calculus. \emph{Graduate Texts in Mathematics} \textbf{65}, American Mathematical Society, 2005.

\bibitem{Safari}
P. Safari, A gluing theorem for Seiberg-Witten moduli spaces. Ph. D. Thesis, Columbia University, 2000.

\bibitem{Simpson-variations}
C. T. Simpson, Constructing variations of Hodge structures using Yang-Mills theory and applications to uniformization. \emph{J. Amer. Math. Soc.} \textbf{1} (1988), 867-918.

\bibitem{Simpson-noncompact}
C. T. Simpson, Harmonic bundles on noncompact curves. \emph{J. Amer. Math. Soc.} \textbf{3} (1990), no. 3, 713-770.

\bibitem{Simpson-Higgs}
C. T. Simpson, Higgs bundles and local systems. \emph{Inst. Hautes {\'E}tudes Sci. Publ. Math.} \textbf{75} (1992), 5-95.

\bibitem{Swoboda}
J. Swoboda, Moduli spaces of Higgs bundles on degenerating Riemann surfaces. \emph{Adv. Math.} \textbf{322} (2017), 637-681.

\bibitem{Taubes}
C. H. Taubes, Self-dual Yang-Mills connections over non self-dual 4-manifolds. \emph{J. Differential Geom.} \textbf{17} (1982), 139-170.

\bibitem{Turaev}
V. G. Turaev, A cocycle of the symplectic first Chern class and Maslov indices. (Russian) \emph{Funktsional. Anal. i Prilozhen.} \textbf{18} (1984), no. 1, 43-48.

\bibitem{Yoshida}
T. Yoshida, Floer homology and splittings of manifolds. \emph{Ann. of Math.} \textbf{134} (1991), 277-324.
\end{thebibliography}
\end{document}